\documentclass[a4paper, reqno]{amsart}
\usepackage[margin=2.65cm]{geometry}
\usepackage{amsmath,amssymb,amsthm, enumerate, mathtools}
\usepackage[dvipsnames]{xcolor}
\usepackage[all,cmtip]{xy}
\usepackage[alphabetic,nobysame,initials]{amsrefs}
\usepackage{relsize,faktor}
\usepackage{tikz}
\usetikzlibrary{
  cd,
  shapes,
  calc,
  positioning,
  arrows,
  decorations.pathreplacing,
  decorations.markings,
}
\tikzset{
    labl/.style={anchor=south, rotate=90, inner sep=.5mm, scale=1.2, yshift=-2mm}
}
\usepackage{tikzsymbols}
\usepackage{graphicx}
		\graphicspath{{Pictures/} }
\usepackage{color} 
\definecolor{darkblue}{rgb}{0,0,0.6}

\usepackage[breaklinks, pdftex, ocgcolorlinks,colorlinks=true, citecolor=darkblue, filecolor=darkblue, linkcolor=darkblue, urlcolor=darkblue]{hyperref}
\usepackage[capitalize,noabbrev]{cleveref}

\newcommand{\hooklongrightarrow}{\lhook\joinrel\longrightarrow}

\newcommand{\longtwoheadrightarrow}  {\relbar\joinrel\twoheadrightarrow}
\newcommand{\Z}{\mathbb{Z}}
\newcommand{\RP}{\mathbb{RP}^2}
\newcommand{\bands}{\mathcal{B}}
\newcommand{\cF}{\mathcal{F}}
\DeclareMathOperator{\km}{km}
\DeclareMathOperator{\pt}{pt}
\DeclareMathOperator{\p}{pm}
\newcommand{\imra}{\looparrowright}
\newcommand{\ra}{\longrightarrow}
\newcommand{\hra}{\hookrightarrow}
\newcommand{\sra}{\twoheadrightarrow}
\newcommand{\TT}{\mathcal{T}} 
\newcommand{\RR}{\mathcal{R}}
\newcommand{\EE}{\mathcal{E}}
\newcommand{\R}{\mathbb{R}}
\newcommand{\N}{\mathbb{N}}

\newcommand{\ol}{\overline}
\newcommand{\sm}{\setminus}
\newcommand{\wt}{\widetilde}
\newcommand{\CP}{\mathbb{CP}}

\newcommand{\tw}{t}
\newcommand{\wh}{\widehat}
\newcommand{\W}{\mathcal{W}}

\DeclareMathOperator{\Ima}{Im}
\makeatletter
\newtheorem*{rep@theorem}{\rep@title}
\newcommand{\newreptheorem}[2]{%
\newenvironment{rep#1}[1]{%
 \def\rep@title{#2 \ref{##1}}%
 \begin{rep@theorem}}%
 {\end{rep@theorem}}}
\makeatother

\DeclareMathOperator{\Imm}{Imm}

\def\drawtw#1#2{\tikz[baseline=-.07ex,cap=round,scale=#1,line width=#2]
  {\draw (0,1ex) ..controls+(-.5ex,0) and +(-.5ex,0)..
    (0ex,0ex) ..controls+(.5ex,0) and +(-.5ex,0)..
    (.85ex,1ex) ..controls+(.5ex,0) and +(.5ex,0).. (.85ex,0);}}
\newsavebox{\ttwbox}\newsavebox{\stwbox}\newsavebox{\sstwbox}
\sbox{\ttwbox}{\drawtw{.9}{.6pt}}
\sbox{\stwbox}{\drawtw{.54}{.5pt}}
\sbox{\sstwbox}{\drawtw{.45}{.45pt}}
\def\twist{{\mathchoice{\usebox{\ttwbox}}{\usebox{\ttwbox}}{\usebox{\stwbox}}{\usebox{\sstwbox}}}}

\DeclareMathOperator{\PD}{PD}

\DeclareMathOperator{\Arf}{Arf}

\DeclareMathOperator{\Id}{Id}

\DeclareMathOperator{\ks}{ks}
\DeclareMathOperator{\Int}{Int}
\DeclareMathOperator{\Alt}{alt}

\newcommand{\Sum}{\displaystyle \sum}

\newcommand*{\newfaktor}[2]{
  \raisebox{0.3\height}{\ensuremath{#1}}
  \mkern-5mu\diagup\mkern-4mu
  \raisebox{-0.7\height}{\ensuremath{#2}}
}

\newcommand{\smfrac}[2]{\newfaktor{#1}{#2}} 

\def\smath#1{\text{\scalebox{.8}{$#1$}}}

\def\sfrac#1#2{\smath{\frac{#1}{#2}}}

\theoremstyle{plain}
\newtheorem{theorem}{Theorem}[section]
	\newtheorem{proposition}[theorem]{Proposition}
	\newtheorem{lemma}[theorem]{Lemma}
	\newtheorem{corollary}[theorem]{Corollary}

\newreptheorem{theorem}{Theorem}
\newreptheorem{lemma}{Lemma}
\newreptheorem{proposition}{Proposition}
\newreptheorem{corollary}{Corollary}

\newtheorem{claim}[theorem]{Claim}
\newtheorem{case}{Case}
\theoremstyle{definition}
	\newtheorem{definition}[theorem]{Definition}
	\newtheorem{example}[theorem]{Example}

	\newtheorem{construction}[theorem]{Construction}
	\newtheorem{convention}[theorem]{Convention}
	\newtheorem{remark}[theorem]{Remark}
	
\newtheoremstyle{theorem-giventitle}
        {}{}              
        {\itshape}                      
        {}                              
        {\bfseries}                     
        {.}                             
        { }                             
        {\thmnote{\bfseries#3}}
\theoremstyle{theorem-giventitle}
\newtheorem{theorem-named}{}
\numberwithin{figure}{section}
\numberwithin{equation}{section}

\begin{document}

\title{Embedding surfaces in $4$-manifolds}

\author[D.~Kasprowski]{Daniel Kasprowski}
\address{School of Mathematical Sciences, University of Southampton, United Kingdom}
\email{d.kasprowski@soton.ac.uk}

\author[M.~Powell]{Mark Powell}
\address{School of  Mathematics and Statistics, University of Glasgow, United Kingdom}
\email{mark.powell@glasgow.ac.uk}

\author[A.~Ray]{Arunima Ray}
\address{Max Planck Institute for Mathematics, Vivatsgasse 7, 53111 Bonn, Germany}
\email{aruray@mpim-bonn.mpg.de}

\author[P.~Teichner]{Peter Teichner}
\address{Max Planck Institute for Mathematics, Vivatsgasse 7, 53111 Bonn, Germany}
\email{teichner@mac.com}

\def\subjclassname{\textup{2020} Mathematics Subject Classification}
\expandafter\let\csname subjclassname@1991\endcsname=\subjclassname
\subjclass{
57K40, 
57N35. 
}
\keywords{Embedding surfaces in 4-manifolds, Kervaire--Milnor invariant}

\begin{abstract}
We prove a surface embedding theorem for $4$-manifolds with good fundamental group in the presence of dual spheres, with no restriction on the normal bundles.
The new obstruction is a Kervaire--Milnor invariant for surfaces and we give a combinatorial formula for its computation. For this we introduce the notion of band characteristic surfaces.
\end{abstract}
\maketitle

\section{Introduction}

We study whether a given map of a surface to a topological $4$-manifold is homotopic to an embedding. Here and throughout the article, embeddings and immersions in the topological category are by definition \emph{locally flat}, meaning they are locally modelled on linear inclusions $\R^2 \hra\R^4$ or $\R^2_+ \hra \R^4$.

Even for maps of $2$-spheres, this question has only been completely addressed in a handful of simple manifolds, such as $S^4$, $\CP^2$~\cite{tristram}*{p.\ 264}, and $S^2 \times S^2$~\citelist{\cite{tristram}*{Theorem~4.5}\cite{Kervaire-Milnor:1961-1}*{Corollary~1}\cite{F}*{Corollary~1.1}}. Lee--Wilczynski~\cites{lee-wilczynski,lee-wilczynski-AJM} and Hambleton--Kreck~\cite{hambleton-kreck:1993} described the minimal genus of an embedded surface in a fixed homology class, in any given simply connected, closed $4$-manifold, assuming that the  fundamental group of the complement is abelian. This was recently extended by~\cite{FMNOPR} to knot traces.
In the relative setting even the simplest case is open: which knots in $S^3$ bound a (locally flat) embedded disc in $D^4$?

The main available tool for proving positive results is Freedman's embedding theorem which shows that maps of discs and spheres to a 4-manifold with \emph{good} fundamental group, with vanishing intersection and self-intersection numbers, and with framed algebraically dual spheres, are regularly homotopic to embeddings (\cref{thm:DET}~\citelist{\cite{F}\cite{FQ}*{Corollary 5.1B}}, see also~\citelist{\cite{PRT20}\cite{Freedman-notes}}). Surgery and the $s$-cobordism theorem for topological 4-manifolds with good fundamental group are consequences of this theorem~\citelist{\cite{Quinn-annulus}\cite{FQ}}.
Our aim, realised in \cref{thm:SET,thm:main} below, is to extend Freedman's theorem to all compact surfaces with algebraically dual spheres, in any connected 4-manifold with good fundamental group. In \cref{subsection:homotopy-vs-reg-homotopy} we explain how to apply this to the question from the opening paragraph of whether a given homotopy class contains an embedding.
In \cref{subsection:knots} we give some applications to knot theory. In particular, we show that every knot bounds an embedded surface of genus one in $M\sm \mathring{D}^4$ for every simply connected closed 4-manifold $M$ not homeomorphic to $S^4$. Recall that for $M=S^4$, this does not hold because the slice genera of knots can be arbitrary large.

Throughout the paper, we will work in the following setting unless otherwise specified.

\begin{convention}\label{convention}
We assume that $M$ is a connected, topological $4$-manifold and that $\Sigma$ is a nonempty compact surface with connected components $\{\Sigma_i\}_{i=1}^m$. The notation $F = \{f_i\}_{i=1}^m \colon (\Sigma,\partial \Sigma)\looparrowright (M, \partial M)$ represents a generic immersion (\cref{def:gen_immersion}) with components $f_i\colon (\Sigma_i,\partial \Sigma_i)\looparrowright (M, \partial M)$.
\end{convention}

By assumption, the map $F$ restricts to an embedding on $\partial \Sigma$ and $F^{-1}(\partial M) = \partial \Sigma$, where $\partial \Sigma$ and $\partial M$ are permitted to be empty. There is no requirement for $\Sigma$ or  $M$ to be orientable, and $M$ could be non-compact.
Weakening the hypotheses of Freedman's theorem to allow for the algebraically dual spheres to be unframed introduces an additional obstruction, the \emph{Kervaire--Milnor invariant} $\km(F)\in\Z/2$ (\cref{def:km=0}), which vanishes in the presence of framed algebraically dual spheres.

\begin{theorem}[Surface embedding theorem]\label{thm:SET}
Let $F=\{f_i\}_{i=1}^m \colon (\Sigma,\partial \Sigma)\looparrowright (M, \partial M)$ be as in \cref{convention}. Suppose that $\pi_1(M)$ is good and that $F$ has algebraically dual spheres $G=\{[g_i]\}_{i=1}^m \subseteq \pi_2(M)$. Then the following statements are equivalent.

\begin{enumerate}[(i)]
\item \label{item-i-SET}  The intersection numbers $\lambda(f_i,f_j)$ for all $i < j$, the self-intersection numbers $\mu(f_i)$ for all $i$, and the Kervaire--Milnor invariant $\km(F)\in\Z/2$, all vanish.
\item \label{item-ii-SET} There is an embedding $\ol{F}=\{\ol{f}_i\}_{i=1}^m \colon (\Sigma, \partial\Sigma)  \hookrightarrow (M,\partial M)$, regularly homotopic to $F$ relative to $\partial \Sigma$, with geometrically dual spheres $\ol{G}=\{\ol{g}_i\colon S^2\looparrowright M\}_{i=1}^m$ such that $[\ol{g}_i]= [g_i]\in \pi_2(M)$ for all $i$.
\end{enumerate}
\end{theorem}

Equivariant intersection and self-intersection numbers of immersed discs and spheres have a long history (see e.g.~\cite{Wall-surgery-book}). In the theorem above, we consider generalised versions for arbitrary compact surfaces, lying in quotients of the group ring $\Z[\pi_1(M)]$, which we denote by $\Gamma_{f_i,f_j} \ni \lambda(f_i,f_j)$ for the intersection numbers and $\Gamma_{f_i} \ni \mu(f_i)$ for the self-intersection numbers.
We describe these quotients in detail in \cref{subsection:intersection-numbers,subsection:self-intersection-numbers}.
As in the simply connected case, these invariants require based maps (\cref{def:based}) but their vanishing as in \cref{thm:SET}~(i) does not depend on the choice of basing. Vanishing of all the $\lambda(f_i,f_j)$ for $i<j$ and all the $\mu(f_i)$ is equivalent to the vanishing of the self-intersection number $\mu(F)$, which is defined as follows.

\begin{definition}\label{def:mu(F)}
Let $F=\{f_i\}_{i=1}^m \colon (\Sigma,\partial\Sigma)\looparrowright (M,\partial M)$ be as in \cref{convention}. Assume in addition that $M$ and $\Sigma$ are based and that $F$ is a based map. The \emph{self-intersection number} of this possibly disconnected immersed surface is given by counting all double points of $F$, as follows:
\[\mu (F):=\sum_{i < j}\lambda(f_i,f_j)+\sum_i\mu(f_i)\in \bigoplus_{i< j}
\Gamma_{f_i,f_j}\oplus \bigoplus_i\Gamma_{f_i}.\]
\end{definition}

The self-intersection number $\mu (F)$ is a regular homotopy invariant that vanishes if and only if there is a collection of Whitney discs that pair all double points of $F$ (\cref{cor:vanishing-pairing}), just like for connected surfaces. The Whitney discs may be chosen to form a \emph{convenient} collection, meaning that all Whitney discs have disjointly embedded boundaries, are framed, and have interiors transverse to~$F$ (\cref{def:convenient-Whitney}).

\begin{definition}\label{def:km}\label{def:km=0}
Let $F\colon (\Sigma,\partial \Sigma)\looparrowright (M, \partial M)$ be as in \cref{convention}.
By definition, $\km(F)\in\Z/2$ vanishes if and only if, after finitely many finger moves taking $F$ to some $F'$, there is a convenient collection of Whitney discs pairing all the double points of $F'$ and whose interiors are disjoint from~$F'$.
\end{definition}

We think of $\mu (F)$ as the primary embedding obstruction, and $\km(F)$ as the secondary embedding obstruction.
Note that $\km(F)=0$ implies $\mu (F)=0$ but $\km(F)$ is always defined even if $\mu (F) \neq 0$. In \cref{sec:computing-km} we will give a combinatorial description of $\km(F)$ in the case that $\mu(F)=0$.

The Kervaire--Milnor invariant is named in homage to~\cite{Kervaire-Milnor:1961-1}, in which Kervaire and Milnor defined an embedding obstruction, and used it to give the first proof that the Whitney trick fails in dimension~4. \cref{section:KM-invariant} gives details on the connection of our Kervaire--Milnor invariant to the original obstruction and other secondary embedding obstructions in the literature.

A group is said to be \emph{good} if it satisfies the \emph{$\pi_1$-null disc property}~\cite{Freedman-Teichner:1995-1} (see also~\cite{Freedman-book-goodgroups}); we shall not repeat the definition. In practice, it suffices to know that virtually solvable groups and groups of subexponential growth are good, and that the class of good groups is closed under taking subgroups, quotients, extensions, and colimits~\cites{Freedman-Teichner:1995-1, Krushkal-Quinn:2000-1}.

If $\Sigma$ is a union of discs or spheres, \cref{thm:SET} follows from \cite{FQ}*{Theorem 10.5\,(1)}.  The latter theorem contained an error found by Stong~\cite{Stong} (cf.\ \cref{theorem:Stong}),
but this is not relevant to \cref{thm:SET} because of the way we defined the Kervaire--Milnor invariant.  It is, however, very relevant to how to compute the Kervaire--Milnor invariant, and Stong's correction will be one of the ingredients in our results (see~\cref{sec:computing-km}).

For an arbitrary $\Sigma$, one could try to prove \cref{thm:SET} by using general position to embed the $0$- and $1$-handles of $\Sigma$ (relative to $\partial\Sigma$) and removing a small open neighbourhood thereof from $M$. This gives a new connected $4$-manifold $M_0$ with the same fundamental group as $M$, and only the $2$-handles $\{h_i\colon  (D^2,S^1)\imra (M_0,\partial M_0)\}$, one for each component $\Sigma_i$ of $\Sigma$, remain to be embedded. One then hopes to apply \cite{FQ}*{Theorem 10.5\,(1)} (i.e.\ \cref{thm:SET} for $\Sigma$ a union of discs) to these maps of $2$-handles to produce the desired embedded surface. The original algebraically dual spheres $\{g_i\}$ for the $\{f_i\}$ perform the same r\^{o}le for the $\{h_i\}$ in $M_0$. The intersection and self-intersection numbers $\lambda$ and $\mu$ remain unchanged, hence they also vanish for $\{h_i\}$. However, the Kervaire--Milnor invariant may behave differently. That is, it may become nonzero for the embedding problem for the discs $\{h_i\}$, whereas it was trivial for the original $F$. We show that this phenomenon can occur in \cref{ex:km-handle}. This difference stems from the fact that in applying \cite{FQ}*{Theorem 10.5\,(1)} we fix an embedding of the $1$-skeleton of $\Sigma$ and try to extend it across the $2$-handles. As usual in obstruction theory, it might be advantageous to go back and change the solution of the problem on the $1$-skeleton.

\subsection{Computing the Kervaire--Milnor invariant}\label{sec:computing-km}

The strength of \cref{thm:SET} versus the above strategy using~\cite{FQ}*{Theorem 10.5\,(1)} lies in our computation of the Kervaire--Milnor invariant for arbitrary compact surfaces.
In the case of discs and spheres, Stong showed that the Kervaire--Milnor invariant vanishes in more situations than claimed by Freedman--Quinn, due to the ambiguity arising from sheet choices when pairing up double points by Whitney discs, when the associated fundamental group element has order two. As we recall in \cref{theorem:Stong}, Stong~\cite{Stong} introduced the notion of an \emph{$r$-characteristic} surface, short for \emph{$\RP$-characteristic} surface (\cref{def:r-char}), to give a criterion, in terms of  copies of $\RP$ immersed in the ambient manifold $M$, to decide whether the sheet changing move is viable. Combined with the work of Freedman and Quinn, this enabled the computation of the Kervaire--Milnor invariant, and therefore answered the embedding problem for every finite union of discs or spheres with algebraically dual spheres, in an ambient $4$-manifold with good fundamental group (see \cref{rem:freedman-quinn-stong-win} for more details).

In order to compute the Kervaire--Milnor invariant $\km(F)$ for general surfaces, we extend the notion of $r$-characteristic surfaces to a notion of \emph{$b$-characteristic} surfaces, short for \emph{band characteristic} (\cref{def:b-char}), defined using bands (annuli and M\"{o}bius bands) immersed in~$M$. The next theorem is a generalisation of Stong's computation of $\km(F)$ to arbitrary compact surfaces.

\begin{definition}\label{def:t}
Let $F=\{f_i\}_{i=1}^m \colon (\Sigma,\partial \Sigma)\looparrowright (M, \partial M)$ be as in \cref{convention} with $\mu (F)=0$. Choose  a convenient collection $\W=\{W_\ell\}$ of Whitney discs that pair all double points of $F$ and define
 \[
 \tw(F,\W) :=\Sum_{\ell,i}\ \lvert \Int{W_\ell}\pitchfork f_i\rvert \mod 2.
 \]
In other words, $\tw(F,\W) $ is the mod 2 count of transverse intersections between $F$ and the interiors of the Whitney discs in $\W$.
\end{definition}

We will often apply this definition to the restriction $F^\twist$ of $F=\{f_1,\dots,f_m\}$ to the sub-surface $\Sigma^\twist \subseteq \Sigma$, which includes a component $\Sigma_i$ of $\Sigma$ precisely if its image does not admit a framed immersion $g_i\colon S^2\looparrowright M$ with $\lambda(f_j,g_i)=\delta_{ij}$ for all $j=1,\dots,m$ (\cref{def:Ft}). The main result of the article is as follows.

\begin{theorem}\label{thm:main}
Let $F\colon (\Sigma,\partial \Sigma)\looparrowright (M, \partial M)$ be as in \cref{convention}. Suppose that $\mu (F)=0$ and that $F$ has algebraically dual spheres. If $F^\twist$ is not $b$-characteristic then $\km(F)=0$. If $F^\twist$ is $b$-characteristic then the secondary embedding obstruction satisfies \[\km(F) = \tw(F^{\twist},\W^{\twist})\in \Z/2\]
for every convenient collection of Whitney discs $\W^\twist$ pairing all the double points of $F^{\twist}$.
\end{theorem}

The main novelty in this theorem lies in finding the right condition on~$F$ that makes the combinatorial formula $\tw(F^{\twist},\W^{\twist})$ independent of the choice of Whitney discs, namely that $F^\twist$ is $b$-characteristic. Note that if $\pi_1(M)$ is good, then for $\km(F)=0$ in the previous theorem, \cref{thm:SET} gives an embedding regularly homotopic to $F$.
In practice, it can often be easy to determine if a given surface is $b$-characteristic, as demonstrated by the following corollaries, derived in \cref{sec:applications} as consequences of the more general  \cref{prop:application}.

\begin{corollary}\label{cor:simply-connected-pos-genus}
	If $M$ is a simply connected $4$-manifold and $\Sigma$ is a connected, oriented surface with positive genus, then any generic immersion $F\colon (\Sigma,\partial\Sigma)\imra (M,\partial M)$ with vanishing self-intersection number is not $b$-characteristic.
Thus if~$F$ has an algebraically dual sphere then $\km(F)=0$, and since $\pi_1(M)$ is good the map~$F$ is regularly homotopic, relative to~$\partial \Sigma$, to an embedding.
\end{corollary}

This corollary in particular implies that for every simply connected $4$-manifold $M$ with boundary a disjoint union of homology spheres every primitive class in $H_2(M;\Z)$ can be represented by an embedded torus. This recovers~\cite{lee-wilczynski-AJM}*{Theorem~1.1} in the case of divisibility $d=1$. We also have the following extension to the case of arbitrary $4$-manifolds.

\begin{corollary}\label{cor:stabilisation}
Let $F\colon (\Sigma,\partial \Sigma)\looparrowright (M, \partial M)$ be as in \cref{convention},  with $\mu(F)=0$ and $\Sigma$ connected.
If $F'$ is obtained from~$F$ by an ambient connected sum with an embedding $S^1 \times S^1 \hookrightarrow S^4$, then $F'$ is not $b$-characteristic.
Thus if $F$ has an algebraically dual sphere then $\km(F')=0$, and if $\pi_1(M)$ is good then~$F'$ is regularly homotopic, relative to~$\partial \Sigma$, to an embedding.
\end{corollary}

See \cref{prop:simply-connected-nonorientable,prop:stabilisation-crosscap} for the nonorientable analogues of these two results. In particular, the latter concerns the case where we replace the embedded torus in \cref{cor:stabilisation} by an embedded $\RP$.

 \subsection{Band characteristic maps}\label{sec:intro-bchar}

We briefly explain how the notion of a map being $b$-characteristic arises in the context of embedding general surfaces. Given $F\colon \Sigma\looparrowright M$ as in \cref{convention}, assume that its double points are paired by a convenient collection of Whitney discs~$\mathcal{W}$. Then the interiors of the discs in $\mathcal{W}$ could be tubed into spheres in $M$, potentially changing the count~$t$ from \cref{thm:main}. The condition that $F$ is \emph{$s$-characteristic}, short for \emph{spherically characteristic} (\cref{def:s-char}), precisely ensures that the count is preserved under this move.

\begin{figure}[htb]
	\centering
\begin{tikzpicture}
        \node[anchor=south west,inner sep=0] at (0,0){	\includegraphics{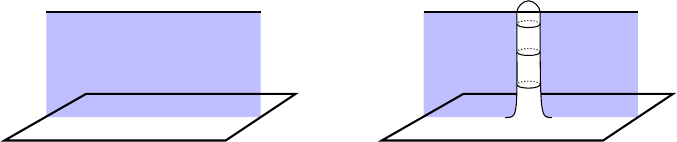}};
		\node at (-0.2,0) {$F$};
		\node at (6.1,0) {$F'$};
		\node at (2.5, 2.4) {$F$};
		\node at (8,2.4) {$F'$};
		\node at (2.5, 1.5) {$B$};
		\node at (8,1.5) {$W_B$};
	\end{tikzpicture}
\caption{Two portions of the immersion $F$ and part of a band $B$  are shown on the left. A finger move produces $F'$ with two new double points, paired by $W_B$.}
	\label{fig:band-gives-whitney-disc}
\end{figure}
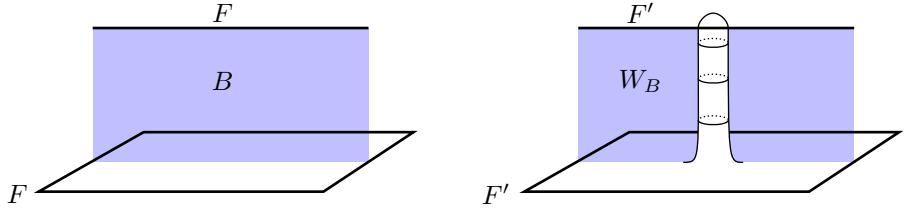

Similarly, consider a band, i.e.\ an annulus or M\"{o}bius band, immersed in $M$ with boundary lying on $F(\Sigma)$ minus the double points, as in \cref{fig:band-gives-whitney-disc}.
Then as shown in the figure we may perform a finger move on $F$ along a fibre of the band, creating $F'$ with two new intersections, paired by a new Whitney disc $W_B$ arising from the band $B$ (see~\cref{fig:band-gives-whitney-disc}). We call this move the \emph{band fibre finger move} and give further details in \cref{construction:finger-move}. Adding $W_B$ to $\mathcal{W}$ might in principle change the count $t$, but the requirement that $F^\twist$ is $b$-characteristic maps ensures it does not. In the case that $\Sigma$ has only simply connected components, the boundary of the band is null-homotopic in $\Sigma$, and therefore the band can be closed off by discs to produce either a sphere (from an annulus) or an $\RP$ (from a M\"{o}bius band). Thus in this case it suffices to consider $r$-characteristic maps.

However, for general $\Sigma$ there may exist a band in $M$ with a boundary curve that is nontrivial in $\pi_1(\Sigma)$. This necessitates the new notion of $b$-characteristic maps, which by definition requires that a function $\Theta \colon \mathcal{B}(F) \to \Z/2$ vanishes (\cref{def:thetaA,def:thetaA-nonorientable}), where $\mathcal{B}(F)$ consists of the homology classes in $H_2(M,\Sigma;\Z/2)$ that can be represented by certain immersed bands in $M$ with boundary on $\Sigma$ (\cref{def:band}). These additional conditions on the bands have to do with the first Stiefel--Whitney classes of $M$ and $\Sigma$; when both are oriented, $\bands(F)$ consists precisely of the classes in $H_2(M,\Sigma;\Z/2)$ that are represented by maps of annuli and M\"obius bands.
 Roughly speaking, the vanishing of $\Theta$ means that every band with boundary on $\Sigma$ intersects $\Sigma$ evenly many times in its interior. Intersections among the boundary components of the bands and a relative Euler number also play a r\^{o}le: see \cref{sec:intro-km-section,sec:km-sec3} for details. If $\Theta \equiv 0$ then for every band $B$, adding $W_B$ does not change the $t$-count, and in fact $t$ is well defined if and only if $F$  is $b$-characteristic (\cref{lem:fiber-move}).    See \cref{ex:km-handle} for a map which is $r$-characteristic but not $b$-characteristic.

The first step for deciding whether $F$ is $b$-characteristic is to determine the subset $\mathcal{B}(F)$.
In general this could be difficult, but in practice it is often soluble.
 If this can be done, then, as shown in \cref{fig:flowchart}, by computing $\lambda_{\Sigma}|_{\partial \mathcal{B}(F)}$ and $\Theta \colon \mathcal{B}(F) \to \Z/2$, we can decide whether $F$ is $b$-characteristic. Both of these are functions on a finite group, so in principle these computations are manageable.

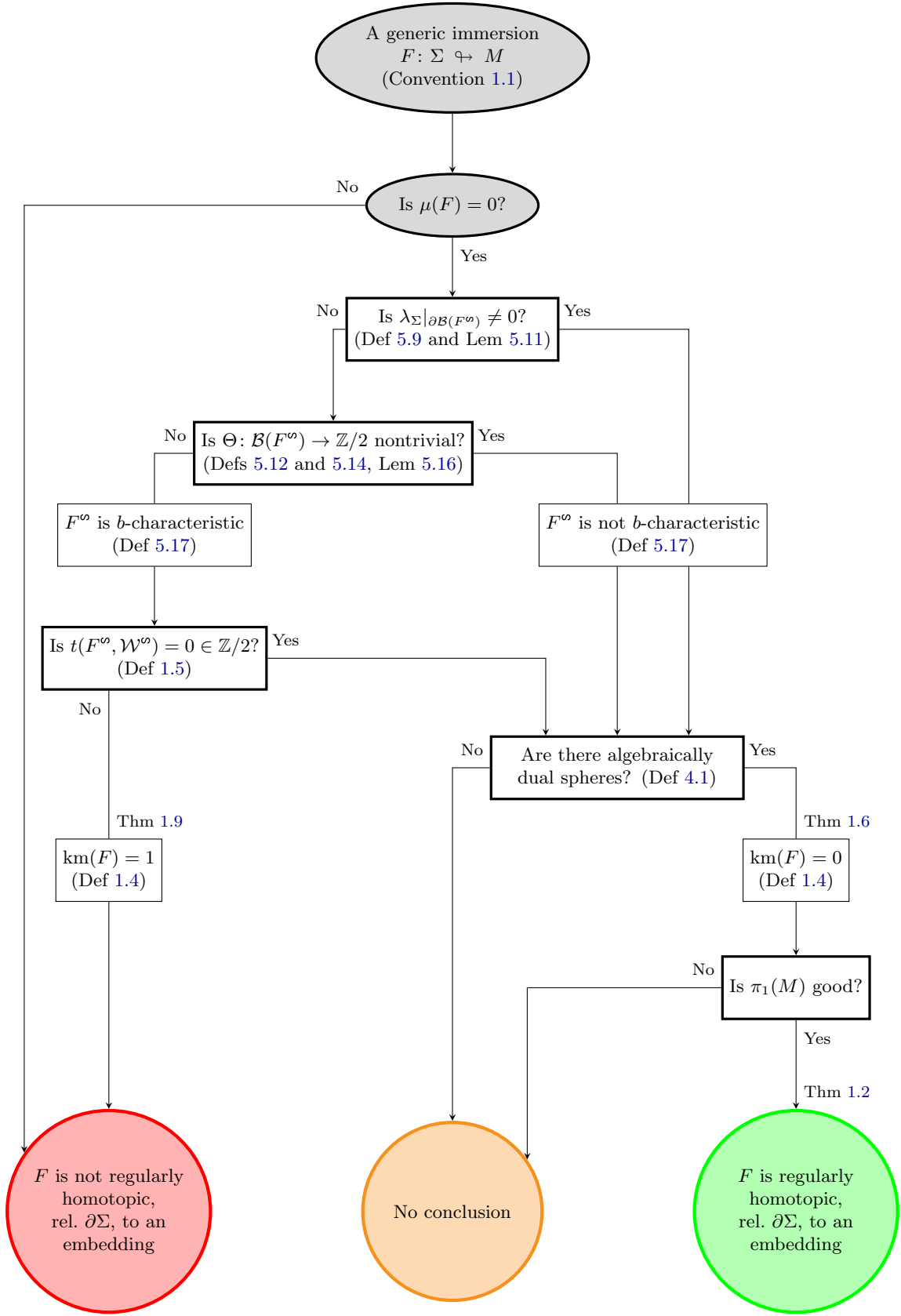
\begin{figure}[p!]
	\centering
	\begin{tikzpicture}[node distance=10mm and 5mm]
		\tikzstyle{every node}=[font=\small, align=center, minimum height=6mm, draw, text centered, thick]
		\tikzstyle{box} = [rectangle]
		\tikzstyle{startbox} = [ellipse, very thick, text centered, minimum height=3em, fill=gray!30, text width=30mm]
		\tikzstyle{prim} = [ellipse, very thick, text centered, minimum height =3em, fill=gray!30,]
		\tikzstyle{decision} = [rectangle, draw, very thick, text centered, minimum height=3em]
		\tikzstyle{duals} = [rectangle, draw, very thick, text centered, text width =4cm, minimum height=3em]
		\tikzstyle{tourist} = [rectangle, draw, thin, text centered, minimum height=3em]
		\tikzstyle{arrow-text}=[draw=none, font=\footnotesize]
		\tikzstyle{end} = [circle, ultra thick, text centered, text width=2.75cm]
		\tikzset{>=stealth}
		
		\node at (0,19.5)(start)  [startbox] {A generic immersion $F\colon \Sigma\looparrowright M$ (\cref{convention})} ;
		\node (prim) [prim, below=1cm of start] {Is $\mu (F)=0$?};
		\node (intform) [decision,below=of prim] {Is $\lambda_\Sigma\vert_{\partial \bands(F^\twist)}\neq 0$?\\
			(Def~\ref{def:band} and Lem~\ref{lem:dependence_on_A})} ;
		\node (theta) [decision,below=of intform, xshift=-2cm] {Is $\Theta\colon \bands(F^\twist)\to \Z/2$ nontrivial?\\ (Defs~\ref{def:thetaA} and \ref{def:thetaA-nonorientable}, Lem~\ref{lem:well-defined})};
		\node (bchar) [tourist, below=of theta, xshift=-3cm, yshift=7mm]{$F^\twist$ is $b$-characteristic\\ (Def~\ref{def:b-char})};
		\node (t) [decision,below=of bchar] {Is $t(F^\twist, \mathcal{W}^\twist) = 0\in \Z/2$?\\(Def~\ref{def:t})};
		
		\node at (0,0) (mystery) [end, draw=BurntOrange, fill=BurntOrange!30] {No conclusion} ;
		\node (canembed) [end,right=2.5cm of mystery, draw=green, fill=green!30] {$F$ is regularly homotopic, rel.\ $\partial \Sigma$, to an embedding} ;
		\node (cannotembed) [end,left=2.5cm of mystery, draw=red, fill=red!30] {$F$ is not regularly homotopic, rel.\ $\partial \Sigma$, to an embedding} ;
		\node (km1) [tourist,above=3.5cm of cannotembed] {$\km(F)=1$\\ (Def~\ref{def:km})} ;
		\node (km0) [tourist,above=3.5cm of canembed] {$\km(F)=0$\\ (Def~\ref{def:km})} ;
		
		\node at ($(mystery.east-|km0)-(3cm,-7.5cm)$) (duals) [duals] {Are there algebraically dual spheres? (Def~\ref{defn:alg-dual-spheres})} ;
		
		\node at (bchar-|duals) (notbchar) [tourist, xshift=6mm]{$F^\twist$ is not $b$-characteristic\\ (Def~\ref{def:b-char})};
		
		\node (good) [decision, above=1.5cm of canembed]{Is $\pi_1(M)$ good?};
		
		\draw[->] (start) -- (prim);
		
		\draw[->] (prim) -- node[arrow-text, pos=0, yshift=-3mm, anchor=west]{Yes} (intform);
		\draw[->] (prim) -| node[arrow-text, anchor=south, pos=0, xshift=-3mm]{No}(cannotembed.145);
		
		\draw[->] (intform.west) -| node[arrow-text, anchor=south, pos=0, xshift=-3mm]{No}(theta);
		\draw[-] (intform.east) -| node[arrow-text, anchor=south, pos=0, xshift=3mm]{Yes}([xshift=1.2cm] notbchar.north-|duals.north);
		
		\draw[-] (theta.west) -| node[arrow-text, anchor=south, pos=0, xshift=-3mm]{No} (bchar);
		\draw[-] (theta.east) -| node[arrow-text, anchor=south, pos=0, xshift=3mm]{Yes}(notbchar.north-|duals.north);
		
		\draw[->] (bchar) -- (t);
		
		\draw[-] (t.south-|km1) -| node[arrow-text, anchor=east, pos=0, yshift=-3mm]{No} node[arrow-text, anchor=west, pos=1, yshift=3mm]{Thm~\ref{thm:embedding-obstruction}}(km1);
		\draw[->] (t.east) -| node[arrow-text, anchor=south, pos=0, xshift=3mm]{Yes}([xshift=-1.2cm] duals.north);
		
		\draw[->] (km1) -- (cannotembed);
		
		\draw[->] (notbchar.south-|duals.north) -- (duals);
		\draw[->] ([xshift=1.2cm] notbchar.south-|duals.north) -- ([xshift=1.2cm] duals.north);
		
		\draw[-] (duals) -| node[arrow-text, anchor=south, pos=0, xshift=3mm]{Yes}node[arrow-text, anchor=west, pos=1, yshift=3mm,]{Thm~\ref{thm:main}}(km0);
		\draw[->] (duals.west) -| node[arrow-text, anchor=south, pos=0, xshift=-3mm,]{No}(mystery.90);
		
		\draw[->] (km0)-- (good);
		
		\draw[->] (good) -- node[arrow-text, anchor=west, pos=0, yshift=-3mm]{Yes} node[arrow-text, anchor=west, pos=1, yshift=3mm]{Thm~\ref{thm:SET}} (canembed);
		\draw[->] (good.west)-|node[arrow-text, anchor=south, pos=0, xshift=-3mm]{No} (mystery.35);
	\end{tikzpicture}
	\caption{A  flowchart deciding whether a generic immersion $F$ is regularly homotopic, relative to the boundary, to an embedding.}\label{fig:flowchart}
\end{figure}

\subsection{An embedding obstruction without dual spheres}

Irrespective of whether $F$ has algebraically dual spheres, we obtain a secondary embedding obstruction in the $b$-characteristic case.

\begin{theorem}\label{thm:embedding-obstruction}
Let $F\colon (\Sigma,\partial \Sigma)\looparrowright (M, \partial M)$ be as in \cref{convention} with $\mu (F)=0$. Let $\W$ be a convenient collection of Whitney discs for the double points of $F$. Then
$F$ is $b$-characteristic if and only if for every $F'$ regularly homotopic to $F$ and convenient collection $\W'$ for the double points of $F'$, we have $\tw(F,\W)=\tw(F',\W')$.

For $b$-characteristic $F$, we denote the resulting regular homotopy invariant by $\tw(F) \in \Z/2$. Then if $\km(F)=0$, for instance if $F$ is an embedding, then $\tw(F) =0$.
\end{theorem}
Note that, in particular, if $F$ is $b$-charactertistic and a map $H$ is regularly homotopic to $F$, then $H$ is $b$-characteristic (\cref{lem:bchar-reg-htpy}). If $F$ is not $b$-characteristic, it is still possible that some restriction $F'$ of $F$ to a union of connected components $\Sigma'\subseteq\Sigma$ is $b$-characteristic. Then we obtain an obstruction for embedding $F'$ and, as a consequence, for embedding $F$. A frequent example of this phenomenon is $F' = F^\twist$ from \cref{thm:main}.
Note that by \cref{lem:F=F-twist}, if $F$ is $b$-characteristic then $F=F^\twist$.

As part of our analysis of the obstruction $t$, in \cref{sec:applications} we shall prove the following additivity properties.

\begin{proposition}
	\label{prop:connectedsum-intro}
Let $M_1$ and $M_2$ be oriented $4$-manifolds. Let $F_1\colon (\Sigma_1, \partial \Sigma_1)\looparrowright (M_1, \partial M_1)$ and $F_2\colon (\Sigma_2, \partial \Sigma_2)\looparrowright (M_2, \partial M_2)$ be generic immersions of connected, compact, oriented surfaces, each with vanishing self-intersection number. If $F_i$ is $b$-characteristic for each $i$ then both the disjoint union
\[F_1\sqcup F_2\colon \Sigma_1 \sqcup \Sigma_2\looparrowright M_1\# M_2\]
and any interior connected sum
\[F_1\# F_2\colon \Sigma_1 \# \Sigma_2 \looparrowright M_1\# M_2\]
 are $b$-characteristic, and satisfy
 \[\tw(F_1 \sqcup F_2)=\tw(F_1\#F_2)=\tw(F_1)+\tw(F_2).\]
\end{proposition}

\cref{thm:embedding-obstruction} and \cref{prop:connectedsum-intro} imply the following corollary.

\begin{corollary}\label{cor:arbitrary-genus-intro}
	For any $g$, there exists a smooth, closed $4$-manifold $M_g$, a closed, connected, oriented surface $\Sigma_g$ of genus $g$, and a smooth, $b$-characteristic, generic immersion $F \colon \Sigma_g \imra M_g$ with $\tw(F)\neq 0$, and therefore $\km(F)\neq 0$.
\end{corollary}

By contrast, we show in \cref{ex:pi1Z} that every map of a closed surface to $S^1\times S^3$ is homotopic to an embedding. One could ask whether there exists a $4$-manifold $M$ and immersions $\Sigma_g\looparrowright M$ with nontrivial Kervaire--Milnor invariant, for every $g$. However, as a partial negative answer we will show in~\cref{prop:so-big} that a $b$-characteristic generic immersion $F\colon \Sigma\looparrowright M$ from a closed surface $\Sigma$ to a compact $4$-manifold $M$ with abelian fundamental group with $n$ generators must satisfy $\chi(\Sigma)\geq -2n$.

\subsection{Homotopy versus regular homotopy}\label{subsection:homotopy-vs-reg-homotopy}
\cref{thm:SET,thm:embedding-obstruction,thm:main} together give a framework for deciding whether or not an immersed surface is regularly homotopic to an embedding, as illustrated by the flowchart in \cref{fig:flowchart}. However, in the first sentence of the article, we began by considering whether a given continuous map is homotopic to an embedding. We explain now how to extend the framework of \cref{fig:flowchart} to decide this, for maps of surfaces that admit algebraically dual spheres.

For a map $f$ from a connected surface to a $4$-manifold, we will show in \cref{theorem:generic-immersions-bijection} that there are either infinitely many or precisely two regular homotopy classes in the homotopy class of $f$, according to whether $f^*(w_1(M))$ is trivial or nontrivial respectively.
Our strategy is to make a judicious choice of regular homotopy class to which we apply our previous theory.

Begin with a continuous map $F\colon \Sigma\to M$ that restricts to an embedding on $\partial\Sigma$ and satisfies $F^{-1}(\partial M) = \partial \Sigma$, where $\Sigma, M$ are as in \cref{convention}. Denote the components of $F$ by $f_i\colon (\Sigma_i,\partial\Sigma_i)\to (M, \partial M)$, and suppose that $F$ has algebraically dual spheres.  Note that homotopies preserve the intersection numbers $\lambda(f_i,f_j)$, but might not preserve the self-intersection number $\mu(f_i)$, since adding a local cusp in $f_i$ changes $\mu(f_i)_1$, the coefficient of $1 \in \pi_1(M)$,  by $\pm 1$. Depending on the behaviour of the orientation characters of $M$ and $\Sigma$, the coefficient $\mu(f_i)_1$ lies in either $\Z$ or $\Z/2$, and is preserved under regular homotopy (see \cref{lem:mu1,prop:regular-homotopy-inv-mu}).

Now, in order to decide whether $F$ is homotopic to an embedding, we will either find a generic immersion in the homotopy class of $F$ which is regularly homotopic to an embedding, or we will show that this is impossible.
First, by performing a homotopy we may assume without loss of generality that $F$ is a generic immersion such that $\mu(f_i)_1=0$ for every component $f_i$ of~$F$.
If $\mu(F)\neq 0$, then $F$ is not homotopic to an embedding.
On the other hand if $\mu(F)=0$, we have the two following cases.
Below, $(f_i)_{\bullet}$ is the map induced on fundamental groups by $f_i$ using some choice of path connecting $f_i(\Sigma_i)$ to the basepoint of~$M$.

\begin{case}
$w_1(\Sigma)|_{\ker (f_i)_{\bullet}}$  is trivial for every $f_i\in F^\twist$.
\end{case}

\noindent By \cref{theorem:generic-immersions-bijection} the regular homotopy class of $F^\twist$ is uniquely determined by the condition that $\mu(f_i)_1=0$ for each $i$ with $f_i \in F^{\twist}$.
Run the analysis in \cref{fig:flowchart} on $F$ to determine whether it is regularly homotopic to an embedding. Note that the outcome of this analysis only depends on the regular homotopy class of $F^\twist$, rather than all of $F$. In particular, if $\pi_1(M)$ is good then $F$ is homotopic to an embedding if and only if $F$ is regularly homotopic to an embedding.
\begin{case}
There exists $f_i\in F^\twist$ with $w_1(\Sigma)|_{\ker (f_i)_{\bullet}}$ nontrivial.
\end{case}

\noindent In this case we use the following theorem, which we prove in \cref{sec:homotopy-reg-homotopy}.

\begin{theorem}
	\label{theorem:changing-t-intro}
Let $F=\{f_i\}_{i=1}^m \colon (\Sigma,\partial\Sigma)\looparrowright (M,\partial M)$ be as in \cref{convention} with $\mu (F)=0$.  Suppose that there is at least one $f_i \in F^{\twist}$ with  $w_1(\Sigma)|_{\ker (f_i)_\bullet}$ nontrivial.  Then there exists a generic immersion $F'$ homotopic to $F$ with $\mu(F')=0$, and a convenient collection of Whitney discs $\W'$ such that $t((F')^{\twist},(\W')^{\twist})=0$.
Thus if $F'$ has algebraically dual spheres then $\km(F')=0$, and if moreover $\pi_1(M)$ is good then $F'$ is regularly homotopic, relative to~$\partial \Sigma$, to an embedding.
\end{theorem}

Using this theorem, we can immediately conclude that our $F$ as in Case 2 is homotopic to an embedding, as long as $\pi_1(M)$ is good. Notably, it is not relevant in this case whether $F^\twist$ is $b$-characteristic. This completes the analysis of whether a given continuous map of a surface into a $4$-manifold is homotopic to an embedding.

We now sketch the proof of \cref{theorem:changing-t-intro}. By the vanishing of $\mu (F)$, there is a convenient collection of Whitney discs $\W$ for $F$ and therefore for $F^\twist$. In case $t(F^\twist, \W^\twist)=0$ the proof is completed by setting $F'=F$. In case $t(F^\twist, \W^\twist)=1$, we use \cref{const:change-t} to find another generic immersion $F'$ homotopic to $F$.  Briefly, \cref{const:change-t} involves creating four new double points in the component $f_i$ with nontrivial $w_1(\Sigma)|_{\ker (f_i)_\bullet}$ using local cusps, and then cancelling them using a suitable choice of Whitney arcs and discs. For further details on the proof, see~\cref{sec:homotopy-reg-homotopy}.

\cref{theorem:changing-t-intro} also has the following immediate corollaries. These are the nonorientable analogues of  \cref{cor:simply-connected-pos-genus,cor:stabilisation}. They provide homotopies to embeddings rather than regular homotopies, and again it is not relevant whether $F^\twist$ is $b$-characteristic.

\begin{corollary}\label{prop:simply-connected-nonorientable}
	If $M$ is a simply connected $4$-manifold and $\Sigma$ is a connected, nonorientable surface, then a generic immersion $F\colon (\Sigma,\partial\Sigma)\imra (M,\partial M)$ with vanishing self-intersection number and an algebraically dual sphere is homotopic, relative to~$\partial \Sigma$, to an embedding.
\end{corollary}

\begin{corollary}\label{prop:stabilisation-crosscap}
Let $F\colon (\Sigma,\partial \Sigma)\looparrowright (M, \partial M)$ be as in \cref{convention}, with $\Sigma$ connected and $\pi_1(M)$ good. Suppose that $F$ has vanishing self-intersection number and an algebraically dual sphere.  If $F'$ is obtained from $F$ by an ambient connected sum with any embedding $\RP \hookrightarrow S^4$, then $F'$ is  homotopic, relative to~$\partial \Sigma$, to an embedding.
\end{corollary}

As in our analysis for the embedding problem up to regular homotopy, our techniques are primarily applicable in the presence of algebraically dual spheres and good fundamental group of the ambient space. It is however sometimes possible to conclude that a map without algebraically dual spheres is homotopic to an embedding. For example we show in \cref{ex:pi1Z} that every map of a closed surface to $S^1\times S^3$ is homotopic to an embedding.

\subsection{Applications to knot theory}\label{subsection:knots}

\cref{thm:SET} can be applied to the problem of finding embedded surfaces in general $4$-manifolds bounded by knots in their boundary. Given a closed $4$-manifold $M$, let $M^\circ$ denote the punctured manifold $M\sm \mathring{D}^4$. The \emph{$M$-genus} of a knot $K\subseteq S^3= \partial M^\circ$, denoted by $g_M(K)$,  is the minimal genus of an embedded orientable surface bounding $K$ in $M^\circ$. If $M$ is smooth, we also consider the \emph{smooth $M$-genus}, denoted by $g_M^\mathrm{Diff}(K)$, the minimal genus of a smoothly embedded orientable surface with boundary $K$. The quantities $g_{S^4}$ and $g^\mathrm{Diff}_{S^4}$ coincide with the topological and smooth slice genus of knots in $D^4$ respectively. Note that $g_{\ol{M}}(K) = g_M(\ol{K})$, so (2) and (3) below imply a corresponding results for $\ol{\CP}^2$ and $\ol{*\CP}^2$ respectively.

\begin{corollary}\label{cor:M-slicing}
For every knot $K\subseteq S^3$,
\begin{enumerate}
	\item $g_M(K)=0$ for every simply connected $4$-manifold $M$ not homeomorphic to one of $S^4$, $\CP^2$, or $*\CP^2$;
	\item $g_{\CP^2}(K)\leq 1$ and $g_{\CP^2}( \#^3 T(2,3))=1$; and
    \item $g_{*\CP^2}(K)\leq 1$ and $g_{*\CP^2}( \#^2 T(2,3))=1$.
\end{enumerate}
\end{corollary}

See \cref{sec:applications} for the proof.
The smooth $\CP^2$-genus has been studied by~\cites{yasuhara91,yasuhara92,aitnouh,pichelmeyer,MMRS},
and differs dramatically from the topological result in \cref{cor:M-slicing}\,(2); in particular, it can be arbitrarily high~\cite{MMRS}.

\cref{cor:M-slicing}\,(1) is straightforward to prove if $M$ topologically splits as a connected sum with $S^2\times S^2$ or $S^2\wt\times S^2$, because  $g_{S^2\times S^2}(K)=g_{S^2\wt \times S^2}(K)=0$ for all $K$ by the Norman trick~\cite{normantrick}*{Corollary~3, Remark}. For the $K3$ surface, this implies that $g_{K3}(K)=0$ for all knots $K$. On the other hand, it is an open question whether there exists a $K$ with $g^\mathrm{Diff}_{K3}(K) \neq 0$~\cite{manolescu-marengon-piccirillo}*{Question~6.1}.

Given a knot $K\subseteq S^3$ and an integer $n\in \Z$, we build the \emph{$n$-trace} $X_n(K)$ by attaching a $2$-handle $D^4$ along $K$ with framing $n$. The minimal genus of an embedded surface representing a generator of $H_2(X_n(K);\Z)$ is called the \emph{$n$-shake genus} of $K$, denoted by $g^\mathrm{sh}_n(K)$. Similarly, the \emph{smooth $n$-shake genus} of $K$ is denoted by $g^\mathrm{sh, Diff}_n(K)$. We recover the following result of~\cite{FMNOPR}.

\begin{corollary}[\cite{FMNOPR}*{Proposition~8.7}]\label{cor:shake-genus}
For any knot $K\subseteq S^3$, $g^\mathrm{sh}_{\pm 1}(K)=\Arf(K)\in\{0,1\}$.
\end{corollary}

By contrast, the smooth $\pm1$-shake genus of a knot can be arbitrarily high. For example, for $q\geq 5$ we have that $g^\mathrm{sh, Diff}_{\pm 1}(T(2,q))\geq \tfrac{q+1}{2}$, by the slice-Bennequin inequality~\citelist{\cite{lisca-matic}\cite{cochran-ray}*{Corollary~5.2}}.

\subsection*{Outline of the paper}
In \cref{sec:basic} we describe the primary embedding obstructions arising from the theory of equivariant intersection numbers for surfaces in $4$-manifolds.  In \cref{section:KM-invariant} we define the Kervaire--Milnor invariant carefully.  We prove \cref{thm:SET} in \cref{sec:thm1.1}. In \cref{sec:intro-km-section} we explain our combinatorial method for computing the Kervaire--Milnor invariant, and define $b$-characteristic surfaces, postponing almost all proofs to \cref{sec:km}. In \cref{sec:homotopy-reg-homotopy} we give the proof of \cref{theorem:changing-t-intro}.  In \cref{sec:proof-of-main-theorem}, we prove \cref{thm:main,thm:embedding-obstruction}.
Finally, in \cref{sec:applications}, we prove \cref{cor:simply-connected-pos-genus,cor:stabilisation,prop:connectedsum-intro,cor:arbitrary-genus-intro,cor:M-slicing,cor:shake-genus}, and we give further applications and examples.

\subsection*{Acknowledgements}
We are grateful to Rob Schneiderman and the anonymous referee for several insightful comments on a previous version, and to Allison N.~Miller and Andrew Nicas for helpful conversations.
Much of this research was conducted at the Max Planck Institute for Mathematics.
DK was supported by the Deutsche Forschungsgemeinschaft (DFG, German Research Foundation) under Germany's Excellence Strategy - GZ 2047/1, Projekt-ID 390685813.
MP was partially supported by EPSRC New Investigator grant EP/T028335/2 and EPSRC New Horizons grant EP/V04821X/2.

\section{Generic immersions and intersection numbers}
\label{sec:basic}

In \cref{sec:generic-immersions} we carefully define and study generic immersions of surfaces in 4-manifolds in the topological category. We show they admit well-behaved normal bundles, and introduce generic homotopies and ambient isotopies between them.

In \cref{subsection:intersection-numbers,subsection:self-intersection-numbers} we study equivariant intersection and self-intersection numbers of generically immersed surfaces. In the case of immersions of spheres and discs, these have a long history, in particular in surgery theory (see e.g.~\cite{Wall-surgery-book}).
For the first time in the literature, as far as we are aware, we give a careful account of intersection and self-intersection numbers in full generality, for compact surfaces and for any possible combination of orientation characters.  The specific groups in which these numbers live depend on the input surfaces, and it is somewhat subtle to describe them.
A preliminary version for orientable surfaces was considered in e.g.~\cite{COT1}*{Section~7}, and the self-intersection number for annuli was considered by Schneiderman in~\cite{schneiderman}.

In \cref{subsection:Whitney-discs} we discuss Whitney discs, which arise if the intersection and self-intersection numbers vanish. We define the important notion of a convenient collection of Whitney discs.  In \cref{subsection:hom-vs-reg-hom}, \cref{theorem:generic-immersions-bijection} explains the difference between homotopy and regular homotopy of generic immersions of surfaces, in terms of the Euler number of the normal bundle or the self-intersection number.

\subsection{Topological generic immersions}\label{sec:generic-immersions}

We start with the definition of an immersion of manifolds in the topological setting.  For $m \geq 0$ let $\R^m_+ := \{(x_1,\dots,x_m) \in \R^m \mid x_1 \geq 0\}$.
For $k \leq n$ we consider the standard inclusions:
\begin{align*}
  \iota \colon \R^k_{\phantom{+}} &= \R^k_{\phantom{+}} \times \{0\} \hooklongrightarrow \R^k_{\phantom{+}} \times \R^{n-k} = \R^n, \\
  \iota_\mathsmaller{+} \colon \R^k_+ &= \R^k_+ \times \{0\} \hooklongrightarrow \R^k_{\phantom{+}} \times \R^{n-k} = \R^n, \text{ and }\\
  \iota_\mathsmaller{++} \colon \R^k_+ &= \R^k_+ \times \{0\} \hooklongrightarrow \R^k_+ \times \R^{n-k} = \R^n_+.
\end{align*}

\begin{definition}\label{defn:immersion}
A continuous map $F \colon \Sigma^k \to M^n$ between topological manifolds of dimensions $k\leq n$ is an \emph{immersion} if locally it is a flat embedding, that is if for each point $p\in \Sigma$ there is a chart $\varphi$ around $p$ and a chart $\Psi$ around $F(p)$ fitting into one of the following commutative diagrams. The first diagram is for $p \in \Int\Sigma$ and $F(p) \in \Int M$, the second diagram is for $p \in \partial \Sigma$ and $F(p) \in \Int M$, and the third is for $p \in \partial \Sigma$ and $F(p) \in \partial M$. In particular $F$ is required to map interior points of $\Sigma$ to interior points of~$M$.
\begin{equation}\label{diagram:above}
\begin{gathered}
\xymatrix{
\R^k \ar[r]^{\iota} \ar[d]^{\varphi} & \ar[d]^{\Psi}  \R^n   \\
\Sigma \ar[r]^F & M }
\hspace{2em}
\xymatrix{
\R^k_+ \ar[r]^{\iota_+} \ar[d]^{\varphi} & \ar[d]^{\Psi}  \R^n   \\
\Sigma \ar[r]^F & M }
\hspace{2em}
\xymatrix{
\R^k_+ \ar[r]^{\iota_{++}} \ar[d]^{\varphi} & \ar[d]^{\Psi}  \R^n_+   \\
\Sigma \ar[r]^F & M }
\end{gathered}
\end{equation}
Some authors prefer to call this notion a \emph{locally flat immersion}.
\end{definition}

\begin{definition}\label{def:normal-bundle}
A (linear) \emph{normal bundle} for an immersion $F\colon \Sigma^k \to M^n$ is an $(n-k)$-dimensional real vector bundle $\pi \colon \nu_F \to \Sigma$, together with an immersion $\wt{F} \colon \nu_F \to M$ that restricts to $F$ on the zero section $s_0$, i.e.\ $\wt F\circ s_0 = F$, and such that each point $p\in\Sigma$ has a neighbourhood $U$ such that $\wt F|_{\pi^{-1}(U)}$ is an embedding.
\end{definition}

We now restrict to the relevant dimensions for this paper, $k=2$ and $n=4$, and take $M$ to be a connected topological 4-manifold as in \cref{convention}.
The \emph{singular set} of an immersion $F \colon \Sigma \to M$ is the set \[\mathcal{S}(F) := \{m \in M \mid |F^{-1}(m)| \geq 2\}.\]
Recall that a continuous map is said to be \emph{proper} if the inverse image of every compact set in the codomain is compact.

\begin{definition}\label{def:gen_immersion}
Let $\Sigma$ be a surface, possibly noncompact.  A continuous, proper map $F \colon \Sigma \to M$ is said to be a (topological) \emph{generic immersion}, denoted $F\colon \Sigma\looparrowright M$, if it is an immersion and the singular set is a closed, discrete subset of $M$ consisting only of transverse double points, each of whose preimages lies in the interior of $\Sigma$. In particular whenever $m \in \mathcal{S}(F)$, there are exactly two points $p_1, p_2 \in \Sigma$ with $F(p_i)=m$, and there are disjoint charts $\varphi_i$ around $p_i$, for $i=1,2$, where $\varphi_1$ is as in the left-most diagram of \eqref{diagram:above}, and $\varphi_2$ is the same, with respect to the same chart $\Psi$ around $m$, but with $\iota$ replaced by
\[
\iota' \colon \R^2 = \{0\} \times \R^2 \hooklongrightarrow \R^2 \times \R^2 = \R^4.
\]
\end{definition}

\begin{theorem}\label{thm:plumbed-normal-bundle}
A generic immersion $F \colon \Sigma \imra M$, for possibly noncompact $\Sigma$, has a normal bundle as in \cref{def:normal-bundle} with the additional property that $\wt F$ is an embedding outside a neighbourhood of $F^{-1}(\mathcal{S}(F))$, and near the double points $\wt{F}$ plumbs two coordinate regions $\pi^{-1}(\varphi_i(\R^2))\cong \varphi_i(\R^2)\times \R^2$, $i=1,2$, together i.e.\  $\wt F\circ ( \varphi_1(x),y) = \wt F\circ (\varphi_2(y),x)$.
\end{theorem}

\begin{proof}
Let $\partial_1\Sigma\subseteq \partial\Sigma$ denote the union of the components of $\partial \Sigma$ mapped to $\partial M$. Then since $F|_{\partial_1 \Sigma}$ is an embedding of a 1-manifold in a 3-manifold, it has a normal bundle. We extend this to a collar neighbourhood of $F(\partial_1 \Sigma)$ contained in a collar neighbourhood of $\partial M$.  To do this, first note that since (i) $\partial M$ is closed in $M$, (ii) $\mathcal{S}(F)$ is closed and contained in $\Int(M)$, and (iii) manifolds are normal, it follows that there is an open neighbourhood of $\partial M$ disjoint from $\mathcal{S}(F)$. Then argue as in Connelly's proof~\cite{Connelly} that boundaries of manifolds have collars, to obtain a homeomorphism of pairs
\[G \colon \big(M,F(\Sigma)\big) \xrightarrow{\phantom{5}\cong\phantom{5}} \big(M \cup (\partial M \times [0,1]),F(\Sigma) \cup (F(\partial_1 \Sigma) \times [0,1])\big).\]
The normal bundle over the boundary extends to a collar in the codomain, hence its pull-back extends to a collar in the domain.

Next, let $\partial_2\Sigma\subseteq \partial\Sigma$ denote the union of the components of $\partial \Sigma$ mapped to $\Int{M}$. We see that $F(\partial_2 \Sigma)$ has a normal bundle by~\cite{FQ}*{Theorem~9.3}, as a submanifold of $M$. Let $\wt{F}$ be the embedding of the total space, as in \cref{def:normal-bundle}.
By using the inward pointing normal for $\partial_2\Sigma$ in $\Sigma$, we obtain an orthogonal decomposition of each fibre as $\nu_{\partial_2 \Sigma \hookrightarrow \Sigma} \oplus V$, where $V$ is a 2-dimensional subspace. Then translates of $V$ in the direction of the inward pointing normal give rise to a normal bundle on the intersection of a collar of $\partial_2 \Sigma$ with the image of the normal bundle of $\partial_2 \Sigma$ under~$\wt{F}$.

Now we want to extend the normal bundle that we have just constructed on a neighbourhood of $\partial \Sigma$ to the rest of $\Sigma$.
First we will produce a normal bundle in a neighbourhood of both preimages of each double point, and then finally we will extend the normal bundle to the rest of the interior of $\Sigma$.

Let $m \in \mathcal{S}(F)$ be a double point of $F$, so that there exist $p_1,p_2\in \Sigma$ with $F(p_i)=m$, $i=1,2$. By the definition of a generic immersion, there is a chart $\Psi$ for $M$ at $m$, and charts $\varphi_i$ around $p_i$, such that $F\circ \varphi_1(x)=\Psi(x,0)$ and $F\circ \varphi_2(y)=\Psi(0,y)$. We assume that $F(\Sigma) \cap \Psi(\R^4) = F(\varphi_1(\R^2)) \cup F(\varphi_2(\R^2))$, and moreover that the images of the charts for different elements of $\mathcal{S}(F)$ do not overlap one another, and also are disjoint from the images of the normal bundles already constructed close to $\partial \Sigma$.
Then we take a trivial $\R^2$-bundle over each $\varphi_i(\R^2)$, and we define the map $\wt{F}$ on $\varphi_1(\R^2)\times \R^2$ and $\varphi_2(\R^2)\times \R^2$ by setting $\wt{F}(\varphi_1(x),y)=\Psi(x,y)$ and $\wt{F}(\varphi_2(x),y)=\Psi(y,x)$. Then $\wt F\circ ( \varphi_2(y),x) = \Psi(x,y) = \wt F\circ (\varphi_1(x),y)$ as needed.

Let $U^m_1$ and $U^m_2$ be open neighbourhoods in $\Sigma$ of $p_1$ and $p_2$, contained within the images of $\varphi_1$ and $\varphi_2$ above, respectively. Define $\Sigma':= \Sigma \sm \bigcup_{m\in \mathcal{S}(F)} (U^m_1\cup U^m_2)$. Then the restriction of $F$ gives an embedding of $\Sigma'$ in $M$. We already have a normal bundle defined on a neighbourhood of $\partial\Sigma'$. Apply~\cite{FQ}*{Theorem~9.3A} to extend the given normal bundle on $\ol{U^m_1\cup U^m_2}$ and $\partial \Sigma$ to all of $\Sigma'$ and therefore we have a normal bundle on all of $\Sigma$.
\end{proof}

\begin{remark}
  Freedman--Quinn~\cite{FQ}*{Theorem~9.3} produce an \emph{extendable} normal bundle for every submanifold of a 4-manifold. The extendability condition is technical with an important consequence: extendable normal bundles are unique up to isotopy.  One can always find an extendable normal bundle embedded in the total space of any given normal bundle.
\end{remark}

The proof of \cref{thm:plumbed-normal-bundle} also applies in the following more general setting where $\partial_2\Sigma$ is not embedded in $M$, but $F\vert_{\partial_2\Sigma}$ factors as a composition of generic immersions $\partial_2\Sigma\imra S\imra M$, for some surface $S$.   We will use this case in the definition of $b$-characteristic maps in \cref{sec:intro-km-section}, so we introduce nomenclature.

\begin{definition}\label{defn:generic-immersion-of-pairs}
  Let $g \colon S \imra M$ be a generic immersion of a surface in a 4-manifold $M$. Let $(B,Z)$ be a pair consisting of a surface $B$ and a collection $Z \subseteq \partial B$  of connected components of its boundary. A map $H \colon B\to M$
  is called a \emph{generic immersion of pairs} if $H(Z)\subseteq g(S)$ and
   \begin{enumerate}[(i)]
     \item $H|_{B \sm Z}$ is a generic immersion that is transverse to $g$ and has image disjoint from $H(Z)$;
     \item $H(B)$ is disjoint from the double points of $g$, which implies there is a unique map $h \colon Z\to S$ with $g\circ h=H|_Z$;
      \item the map $h$ is a generic immersion; and
     \item there is a collar $N$ of $Z$ in $B$ with $H(N \sm Z) \subseteq M \sm g(S)$.
   \end{enumerate} \end{definition}

We denote such maps by $H\colon (B,Z)\looparrowright (M,S)$, and sometimes identify $h$ with $H|_Z$.

\begin{corollary}\label{remark:bdy-not-embedded-normal-bundle}\label{rem:not-so-generic}
Let $g \colon S \imra M$ be a generic immersion of a surface in a 4-manifold $M$ and let $H\colon (B,Z)\looparrowright (M,S)$ be a generic immersion of pairs. Then $H$ admits a normal bundle, i.e.\  a normal bundle for $B$ in $M$ such that the restriction to $Z$ contains a normal bundle for $Z$ in $S$.
\end{corollary}

\begin{proof}
Note that $Z$ has a normal bundle in $S$, and then the sum of this with the normal bundle of $S$ in $M$ guaranteed by \cref{thm:plumbed-normal-bundle} gives rise to a normal bundle for $Z$ in $M$. The rest of the proof proceeds as before.
\end{proof}

Observe that smooth generic immersions are topological immersions. Next we show that when both notions make sense they coincide, which justifies the terminology.

\begin{theorem}\label{thm:smooth-normal-bundle}
Consider a smooth compact surface $\Sigma$ and a $($topological$)$ generic immersion $F \colon
\Sigma \imra M$.
If $M$ is non-compact then let $M':= M$, and if $M$ is compact then choose $p\in M \sm F(\Sigma)$ and set $M' := M \sm \{p\}$.
Then $F$ is a smooth generic immersion in some smooth structure on $M'$.
\end{theorem}

We know that $M'$ has a smooth structure by~\citelist{\cite{FQ}*{Theorem~8.2}\cite{Quinn-annulus}*{Corollary~2.2.3}}.

\begin{proof}
Fix a smooth structure on $\partial M$ such that the generic immersion $F$ restricted to those connected components of $\partial \Sigma$ that map to $\partial M$ is a smooth embedding. To find such a smooth structure, first use the standard smooth structure on the normal bundle of $F|_{\partial \Sigma}$, and then extend this to a smooth structure on all of $\partial M$. Since any two smooth structures on a 3-manifold are isotopic this could also be arranged by an isotopy of $\partial \Sigma$, but our aim is to use the given map without isotoping it.

By~\cref{thm:plumbed-normal-bundle} there is a normal bundle $(\nu_F,\wt{F})$ for $F$. Let $D(\nu_F) \to \Sigma$ be the (closed) disc bundle.  This yields a regular neighbourhood $N(F) := \wt{F}(D(\nu_F))$ of $F(\Sigma)$, a codimension zero submanifold of~$M'$.
The regular neighbourhood $N(F)$ can be identified with a smooth manifold obtained from $D(\nu_F)$ after the requisite plumbing operations and smoothing corners.  Use such an identification to fix a smooth structure on $N(F) \subseteq M$. With respect to this smooth structure the map $\Sigma\to N(F)$ is a smooth generic immersion.

In addition, the boundary of $N(F)$ inherits a smooth structure.  The complement in $M'$ of $\Int N(F) \cup (N(F) \cap \partial M')$ is a connected, noncompact $4$-manifold with a prescribed smooth structure on its boundary. Then the interior has a compatible smooth structure by~\citelist{\cite{FQ}*{Theorem~8.2}\cite{Quinn-annulus}*{Corollary~2.2.3}}, giving rise to a smooth structure on all of $M'$. Since the smooth structure on $N(F)$ is unaltered, $F$ become a smooth generic immersion, as desired.
\end{proof}

Recall that an \emph{isotopy} of homeomorphisms of a manifold $M$ is a map $H\colon M\times[0,1]\to M$ such that the track $M \times [0,1] \to M \times [0,1]$ given by $(m,t)\mapsto (H(m,t),t)$ is a homeomorphism.

\begin{definition}\label{def:ambient-isotopy-immersions}
 An \emph{ambient isotopy} between generic immersions $F, G\colon\Sigma\imra M$ consists of two isotopies $H_\Sigma\colon \Sigma\times[0,1]\to\Sigma$ and $H_M\colon M\times[0,1]\to M$ such that
	\begin{enumerate}
		\item $H_\Sigma(-,0)$ and $H_M(-,0)$ are both the identity; and
		\item $G(x)=H_M(F(H_\Sigma(x,1)),1)$ for all $x\in\Sigma$.
	\end{enumerate}
\end{definition}

This is motivated by the smooth result which states that two generic immersions are ambiently isotopic (in the sense of \cref{def:ambient-isotopy-immersions} but with homeomorphism replaced by diffeomorphism in the definition of an isotopy) if and only if they are connected by a path in the space of generic immersions~\cite{GoGu}*{Chapter~III,~Theorem~3.11}.
Note that for embeddings one does not need the isotopy~$H_{\Sigma}$.

Mirroring the smooth notion, a \emph{generic homotopy} between generically immersed surfaces in a 4-manifold is by definition a sequence of ambient isotopies, finger moves, Whitney moves, and cusp homotopies. The moves in question are defined in local coordinates exactly as in the smooth setting. A \emph{regular homotopy} between generically immersed surfaces in a $4$-manifold is by definition a sequence of ambient isotopies, finger moves, and Whitney moves. The following proposition explains that maps of surfaces in a $4$-manifold can be assumed to be generic immersions, and homotopies between generic immersions may be assumed to be generic as well.

\begin{proposition}[\cite{PRT20}*{Proposition~3.1}]\label{prop:hom-gen-immersion}
Let $\Sigma$ be a compact surface and let $M$ be a topological 4-manifold.
\begin{enumerate}
 \item\label{item:hom-to-gen-imm-1}
 Every map $(\Sigma,\partial\Sigma) \to (M,\partial M)$ is homotopic $($relative to the embedded boundary$)$ to a generic immersion.
\item\label{item:hom-to-gen-imm-2} Every homotopy $(\Sigma,\partial\Sigma) \times [0,1] \to (M,\partial M)$ that restricts to a  generic immersion on $\Sigma \times \{0,1\}$, is homotopic $($relative to the boundary$)$  to a generic homotopy.
\end{enumerate}
\end{proposition}

Briefly, the proposition is proven as follows. Homotope the maps away from a point of $M$ using cellular approximation, remove that point, choose a smooth structure on the complement of the point, and then apply the smooth theory of generic immersions, combining~\cite{Hirsch-Diff-Top}*{Theorems~2.2.6~and~2.2.12} with~\cite{GoGu}*{Chapter~III, Corollary~3.3}.

\subsection{Intersection numbers}\label{subsection:intersection-numbers}

We define intersection numbers between compact, connected surfaces in 4-manifolds.
In order to accommodate the fundamental group in equivariant intersection numbers, we need to use basings.

\begin{definition}\label{def:based}
We call a manifold $X$ \emph{based} if it is equipped with basepoints $p_i \in X_i$ for each connected component $X_i \subseteq X$, together with a local orientation at each $p_i$. A generic immersion $F\colon X\to Y$ between based manifolds with $Y$ connected is said to be \emph{based} if it is equipped with \emph{whiskers}, i.e.~ paths in $Y$ from the basepoint of $Y$ to $F(p_i)$, for each basepoint $p_i$ of $X$.
\end{definition}

For the remainder of this section, let $M$ be a connected, based $4$-manifold and let $\Sigma$ and $\Sigma'$ be based, compact, connected surfaces, unless specified otherwise.

Let $f\colon \Sigma\to M$ and $g\colon \Sigma' \to M$ be based maps that are \emph{transverse}, i.e.\ around each intersection point $f(s) = g(s')$, $s \in \Sigma$, $s' \in \Sigma'$, there are  coordinates that make $f$ and $g$ (in a neighbourhood of $s$ in $\Sigma$ and a neighbourhood of $s'$ in $\Sigma'$) resemble the standard inclusions $\R^2 \times \{0\}$ and $\{0\} \times \R^2$ into $\R^4$ respectively, as in~\eqref{diagram:above}.
We assume that these intersections are the only singularities between $f$ and $g$ and that $f(\partial \Sigma)$ and $g(\partial \Sigma')$ are disjoint.

Let $v_f$ and $v_g$ be whiskers for $f$ and $g$. The intersection number $\lambda(f,g)$ is the sum of signed fundamental group elements
\[
\lambda(f,g):=\sum_{p\in f \pitchfork g} \varepsilon(p) \cdot \eta(p)
\]
as follows.  A priori this is the formal sum of a list of elements of the set $\{\pm 1\} \times \pi_1(M)$.  It will ultimately give rise to an element of a quotient of $\Z[\pi_1(M)]$, given in \cref{defn:Gamma-f-g}, after we factor out the effect of finger and Whitney moves and the effect of the choice of the paths $\gamma_f^p$ and $\gamma_g^p$ in the first bullet point below.

Fix $p\in f\pitchfork g$. Next we define $\varepsilon(p) \in \{\pm 1\}$ and $\eta(p) \in \pi_1(M)$.  We use $\ast$ to denote concatenation of paths.
\begin{itemize}
\item Let $\gamma_f^p$ be a path in $\Sigma$ from the basepoint to $f^{-1}(p)$ and let $\gamma_g^p$ be a path in $\Sigma'$ from the basepoint to $g^{-1}(p)$.
\item The sign $\varepsilon(p)\in\{\pm 1\}$ is determined as follows. Transport the local orientation of $\Sigma$ at the basepoint to $f^{-1}(p)$ along $\gamma_f^p$, and the local orientation of $\Sigma'$ at the basepoint to $g^{-1}(p)$ along $\gamma_g^p$. This induces a local orientation at $p$, by ordering $f$ before $g$. Another local orientation is obtained by transporting the local orientation at the basepoint of $M$ to $p$ along the concatenated path $v_g \ast (g \circ \gamma_g^p)$. We define $\varepsilon(p)= + 1$ when the two local orientations match at $p$, and $-1$ otherwise.
\item The element $\eta(p)\in \pi_1(M)$ is by definition the concatenation $v_f \ast (f \circ \gamma_f^p) \ast (g \circ \gamma_g^p)^{-1} \ast  v_g^{-1}$.
\end{itemize}

For a generic immersion $f\colon \Sigma\looparrowright M$, we define $\lambda(f,f):= \lambda(f,f^+)$, where $f^+$ is a push-off of $f$ along a section of its normal bundle transverse to the zero section. In case the embedding $f\vert_{\partial\Sigma}$ is equipped with a specified framing for its normal bundle, then $f^+$ is defined to be a push-off of $f$ along a section restricting to the first vector of that framing on $\partial \Sigma$.

If $f_t$ is a homotopy of $f$ that is transverse to $g$ for all $t$ then $\lambda(f_t,g)$ is independent of $t$ as a set of signed fundamental group elements, assuming the above choices of $\gamma_{f_t}^p$ are made carefully. However, if $f_t$ describes a finger move of $f$ into $g$, there is a single time $t_0$ at which $f_{t_0}$ and $g$ are not transverse, because there is a tangency. After the tangency, two new intersection points $p$ and $q$ arise. These have the same group element $\eta(p) = \eta(q)$ and opposite signs $\varepsilon(p) = - \varepsilon(q)$, with appropriate choices of $\gamma_{f_t}^p$ and $\gamma_{f_t}^q$. Similarly, a Whitney move reduces the intersections between $f$ and $g$ by such a pair. To get a regular homotopy invariant notion, it is thus important to specify the home of $\lambda(f,g)$ carefully.

For $\Sigma$ and $\Sigma'$ simply connected, the sum $\lambda(f,g)$ is usually considered as an element of $\Z[\pi_1(M)]$ and is independent of the choice of $\{\gamma_f^p\}_p$ and $\{\gamma_g^p\}_p$. In the abelian group $\Z[\pi_1(M)]$ the relations $-a + a=0$ for each $a\in \pi_1(M)$ are built in, and if one identifies the sign $\varepsilon(p)$ with the inverse in this abelian group then finger moves and Whitney moves do not change $\lambda(f,g)$ as an element in the group ring.

For non-simply connected $\Sigma$, $\Sigma'$, the homotopy class of $\gamma_f^p$ and  $\gamma_g^p$ may be changed by wrapping around nontrivial elements in $\pi_1(\Sigma)$ or $\pi_1(\Sigma')$. This wrapping may also change the induced local orientations at the intersection points of $f$ and $g$. We describe this in more detail next. Let $w^M\colon \pi_1(M)\to \{\pm 1\}$, $w^\Sigma\colon \pi_1(\Sigma)\to \{\pm 1\}, w^{\Sigma'}\colon \pi_1(\Sigma')\to \{\pm 1\}$ denote the orientation characters.

\begin{definition}\label{defn:Gamma-f-g}
Let $\Gamma_{f,g}$ be the abelian group generated by the elements of $\pi_1(M)$ and with relators
\begin{equation}\label{eq:lambda-relations}
\gamma -  w^{\Sigma}(\alpha)w^{\Sigma'}(\beta) w^M(g_\bullet(\beta))\cdot  f_\bullet(\alpha) \ast \gamma \ast g_\bullet(\beta),
\end{equation}
for all $\alpha \in \pi_1(\Sigma), \beta \in \pi_1(\Sigma'), \gamma \in \pi_1(M)$.
Here $f_\bullet(\alpha):=v_f\ast (f\circ \alpha) \ast v_f^{-1}$ and $g_\bullet(\beta):=v_g\ast (g\circ \beta)\ast v_g^{-1}$ are elements of $\pi_1(M)$.
\end{definition}

For  transverse $f \colon \Sigma\to M$ and $g \colon \Sigma'\to M$, the intersection number $\lambda(f,g)\in \Gamma_{f,g}$ is well defined.
The relations precisely account for wrapping around elements of $\pi_1(\Sigma)$ or $\pi_1(\Sigma')$ as described above. We will show in \cref{prop:regular-homotopy-inv-lambda} that this target also makes $\lambda(f,g)$ a homotopy invariant.

\begin{remark}\label{remark:examples-Gamma-fg}
In the case that $M$, $\Sigma$, and $\Sigma'$ are all oriented,
\[\Gamma_{f,g} \cong \Z[f_\bullet(\pi_1(\Sigma))\backslash\pi_1(M)/g_\bullet(\pi_1(\Sigma'))],\]
 the free abelian group generated by the double coset quotient of $\pi_1(M)$ by left and right multiplication by the images of loops in $\Sigma$ and $\Sigma'$ respectively.

In general, due to the signs introduced by the orientation characters, there may be torsion in $\Gamma_{f,g}$.
For example, consider $f\colon \RP\looparrowright M$ with $f_\bullet(\mathbb{RP}^1)=1$, without any assumption on $M$. Then for every  $g\colon \Sigma'\to M$, the group \[\Gamma_{f,g} \cong (\Z/2)[\pi_1(M)/g_\bullet(\pi_1(\Sigma'))],\]
is $2$-torsion, due to the relations $\gamma =w^{\RP}(\mathbb{RP}^1)\cdot f_\bullet(\mathbb{RP}^1)\ast \gamma=-\gamma$ for every $\gamma \in \pi_1(M)$, arising from setting $\alpha:= \mathbb{RP}^1$ and $\beta:=1$ in~\eqref{eq:lambda-relations}.
\end{remark}

To understand $\Gamma_{f,g}$ better, we introduce some notation. Write $\pm \pi_1(M) := \{\pm 1\} \times \pi_1(M)$.  There is a natural inclusion $\pm \pi_1(M) \to \Z[\pi_1(M)]$ given by $(\pm 1,\gamma) \mapsto \pm \gamma$.
Write $[a]\in \Gamma_{f,g}$ for the equivalence class of $a\in\Z[\pi_1(M)]$, and let $\sim$ denote the equivalence relation on $\pm \pi_1(M)$ induced by the composition $\pm \pi_1(M) \hookrightarrow \Z[\pi_1(M)]\sra \Gamma_{f,g}$, i.e.\ for $a,b \in \pm \pi_1(M)$,  $a \sim b$ if and only if the images of $a$ and $b$ in $\Gamma_{f,g}$ coincide.
The following lemma is immediate from the definitions.

\begin{lemma}\label{lem:double coset}
Let $\gamma_1,\gamma_2\in \pi_1(M)$. One of the relations $[\gamma_1]= \pm [\gamma_2] \in \Gamma_{f,g}$ holds if and only if $\gamma_1$ and $\gamma_2$ represent the same element in the double coset
$ f_\bullet(\pi_1(\Sigma))\backslash \pi_1(M) / g_\bullet(\pi_1(\Sigma'))$.
\end{lemma}

Let $\p \colon \smfrac{\pm \pi_1(M)}{\sim} \sra  f_\bullet\pi_1(\Sigma)\backslash \pi_1(M) / g_\bullet\pi_1(\Sigma')$ be the map sending $\pm \gamma$ to the class of $\gamma$. We write $|\gamma| := \p(\gamma)$.  Here one should think that $\p$ stands for dividing out ``plus-minus''.  Note that $\p$ has fibres of order 1 or 2 and we can decompose the double coset as a disjoint union $B_1\sqcup B_2$ according to this distinction, where $\p$ gives a bijection $\p^{-1}(B_1) \leftrightarrow B_1$ while $\p^{-1}(B_2) \to B_2$ is $2$-$1$.

\begin{remark}
We give examples in the cases from \cref{remark:examples-Gamma-fg}.
  In the case that $M$, $\Sigma$, and $\Sigma'$ are all oriented, then $B_1  = \emptyset$ and
  $B_2 = f_\bullet(\pi_1(\Sigma))\backslash\pi_1(M)/g_\bullet(\pi_1(\Sigma'))$.
  If we have $f\colon \RP\looparrowright M$ with $f_\bullet(\mathbb{RP}^1)=1$, and $g\colon \Sigma'\to M$ is arbitrary, then
  $B_1 = \pi_1(M)/g_\bullet(\pi_1(\Sigma'))$ and $B_2 = \emptyset$.
\end{remark}

 Choose a section $s$ of $\p$.
For each $s(b) \in \p^{-1}(B_2)$ we denote the other element of $\p^{-1}(b)$ by $-s(b)$.  Their images in $\Gamma_{f,g}$ are indeed inverse to one another, which motivates the notation.

\begin{lemma}\label{lem:Gammafg}
Fix a section $s$ for $\p$ as above. The abelian group $\Gamma_{f,g}$ is a direct sum
$\Gamma_{f,g} = FA \oplus V$ of a free abelian group $FA$ on the set $s(B_2)\subseteq \smfrac{\pm\pi_1(M)}{\sim}\subseteq \Gamma_{f,g}$ and a $\Z/2$-vector space $V$ with basis $s(B_1)=\p^{-1}(B_1)\subseteq \smfrac{\pm\pi_1(M)}{\sim}$.

Reading off the coefficients in this decomposition gives homomorphisms $c_{s(b)} \colon \Gamma_{f,g}\to\Z$ for each $b\in B_2,$ and $c_b\colon \Gamma_{f,g}\to\Z/2$ for each $b\in B_1$, yielding a decomposition of $\Gamma_{f,g}$ as a direct sum of copies of $\Z$ and $\Z/2$.
 \end{lemma}

In particular, the homomorphisms $c_{s(b)}$ and $c_b$ determine the isomorphisms displayed in \cref{remark:examples-Gamma-fg}.

\begin{proof}
Starting with the free abelian group with basis $\pi_1(M)$, a relator in~\eqref{eq:lambda-relations} does one of the following three things.
\begin{enumerate}
\item It identifies two distinct basis elements $\gamma_1$ and $\gamma_2$ if and only if $\gamma_2 = f_\bullet(\alpha) \ast \gamma_1 \ast g_\bullet(\beta)\in \pi_1(M)$ and $w^{\Sigma}(\alpha)w^{\Sigma'}(\beta) w^M(g_\bullet(\beta))=1$ for some $\alpha\in \pi_1(\Sigma)$ and $\beta\in\pi_1(\Sigma')$.
\item It identifies a basis element $\gamma_1$ with the inverse $-\gamma_2$ of another basis element $\gamma_2\neq \gamma_1$ if and only if $\gamma_2 = f_\bullet(\alpha)\ast \gamma_1 \ast g_\bullet(\beta)$ and $w^{\Sigma}(\alpha)w^{\Sigma'}(\beta) w^M(g_\bullet(\beta))=-1$ for some $\alpha\in \pi_1(\Sigma)$ and $\beta\in\pi_1(\Sigma')$.
\item It identifies a basis element $\gamma$ with its inverse $-\gamma$ if and only if
\[\gamma = f_\bullet(\alpha) \ast \gamma \ast g_\bullet(\beta)\]
and
\[w^{\Sigma}(\alpha)w^{\Sigma'}(\beta) w^M(g_\bullet(\beta))=-1
\]
for some $\alpha\in \pi_1(\Sigma)$ and $\beta\in\pi_1(\Sigma')$.
\end{enumerate}
The first two types of relators reduce the basis to the double coset \[f_\bullet\pi_1(\Sigma)\backslash \pi_1(M) / g_\bullet\pi_1(\Sigma').\] The third type adds the relations $2[\gamma]=0$ to the fundamental group elements $\gamma$ in question, which then generate $V$, because these are exactly the $\gamma$ such that $-[\gamma]=[\gamma]$, i.e.\ where $\#\p^{-1}(|\gamma|)=1$.
Those $\gamma$ where $\#\p^{-1}(|\gamma|)=2$ remain of infinite order and generate $FA$.
Note that the second type of relator forces us to choose the section $s$ in order to write down a consistent basis for $FA$.
\end{proof}

The subgroup $V$ and its basis clearly do not depend on our choice of section, but
 the basis of $FA$ depends on this choice. If we change the section $s$ at a point $b\in B_2$ to $s'$ so that $s'(b)= - s(b)$, the associated basis element changes to its inverse. It follows that the subgroup $FA$ of $\Gamma_{f,g}$ does not depend on the choice of $s$.
It also follows that the coefficient maps $c_{s(b)}$ only depend on $s$ up to sign and satisfy $c_{-s(b)} = - c_{s(b)}$.

For a given $a\in \smfrac{\pm\pi_1(M)}{\sim}$ we can choose a section $s$ as above with $s(\p(a)):=a$ and hence we get a coefficient map $c_a$ that is independent of the other values of $s$.
For example, we can take $a:=[\gamma] \in \smfrac{\pm\pi_1(M)}{\sim}$ with $\gamma\in\pi_1(M)$ to get $c_\gamma$.

\begin{definition}\label{def:lambda-coeff-at-gamma}
For $\gamma\in \pi_1(M)$, we write $\lambda(f,g)_\gamma:= c_{\gamma}(\lambda(f,g))$. This quantity lies in $\Z$ (respectively, $\Z/2$) when $|\gamma|$ lies in $B_2$ (respectively $B_1$), or equivalently when $[\gamma]$ has infinite order (respectively, order~2) in $\Gamma_{f,g}$. The values does not depend on the choice  of $s$  and satisfies  $c_{\gamma_1} = - c_{\gamma_2}$ whenever $[\gamma_1] = -[\gamma_2]$.
\end{definition}

The following can be proven using~\cref{prop:hom-gen-immersion} (see e.g.~\citelist{\cite{FQ}*{Section~1.7}\cite{DET-book-detintro}} for the case of discs and spheres).

\begin{proposition}\label{prop:regular-homotopy-inv-lambda}
Let $f\colon \Sigma\to M$ and $g\colon \Sigma' \to M$ be based maps that are transverse to one another. The intersection number $\lambda(f,g)$ is preserved by homotopies that are ambient isotopies near $\partial\Sigma\sqcup\partial \Sigma'$.
\end{proposition}

\begin{remark}
The geometric definition of $\lambda$ given above has a well known algebraic version in the case that $f$ and $g$ correspond to classes in $H_2(M,\partial M;\Z[\pi_1(M)])$. This extends to the case of positive genus, as we now sketch. We restrict ourselves to the case that $M$, $\Sigma$, and $\Sigma'$ are closed and oriented for convenience.

Choose a basepoint in the universal cover of $M$, lifting the basepoint of $M$. The maps $f$ and $g$ lift uniquely (with respect to this choice of basepoint) to covers $\widehat{M}$ and $\widehat{M}'$, corresponding to the subgroups $f_\bullet(\pi_1(\Sigma))$ and $g_\bullet(\pi_1(\Sigma'))$ respectively. These lifts represent classes
\[
[f]\in H_2(\widehat{M};\Z)\cong H_2\big(M;\Z[\pi_1(M)/f_\bullet\pi_1(\Sigma)]\big)
\]
and
\[[g]\in H_2(\widehat{M}';\Z)\cong H_2\big(M;\Z[\pi_1(M)/g_\bullet\pi_1(\Sigma')]\big).
\]

Then we have \[\PD^{-1}([f])\smile \PD^{-1}([g]) \in H^4\big(M;\Z[\pi_1(M)/f_\bullet\pi_1(\Sigma)] \otimes_\Z \Z[\pi_1(M)/g_\bullet\pi_1(\Sigma')]\big).\]
By Poincar\'{e} duality, this yields an element in \[H_0\big(M;\Z[\pi_1(M)/f_\bullet\pi_1(\Sigma)] \otimes_\Z \Z[\pi_1(M)/g_\bullet\pi_1(\Sigma')]\big),\] which is isomorphic as an abelian group to \[\Z[f_\bullet\pi_1(\Sigma)\backslash\pi_1(M)] \otimes_{\Z[\pi_1(M)]} \Z[\pi_1(M)/g_\bullet\pi_1(\Sigma')].\]
Here $\Z[f_\bullet(\pi_1(\Sigma))\backslash\pi_1(M)]$ denotes $\Z[\pi_1(M)/f_\bullet\pi_1(\Sigma)]$ considered as a right $\Z[\pi_1(M)]$-module.

Finally we have the isomorphism
\begin{align*}
\Z[f_\bullet\pi_1(\Sigma)\backslash\pi_1(M)] \otimes_{\Z[\pi_1(M)]} \Z[\pi_1(M)/g_\bullet\pi_1(\Sigma')]	&\ra \Z[f_\bullet\pi_1(\Sigma)\backslash \pi_1(M) / g_\bullet\pi_1(\Sigma')]\\
[a]\otimes [b]	&\longmapsto [ab] \quad \text { for } a,b\in \pi_1(M).
\end{align*}
We shall not prove that this formulation agrees with the geometric definition.
\end{remark}

\subsection{Self-intersection numbers}\label{subsection:self-intersection-numbers}

Next we turn to the \emph{self-intersection number} for a based generic immersion $f\colon \Sigma\looparrowright M$ of a connected surface $\Sigma$, with whisker $v_f$.
The definition of $\mu$, given below, is similar to that of $\lambda$ in the previous subsection, except that there is no longer a clear choice of which sheet to consider first at a given double point. Consequently, the values of $\mu$ lie in a further quotient of the group $\Gamma_{f,f}$ from \cref{defn:Gamma-f-g}.

We write $f \pitchfork f \subseteq M$ for the set of double points of $f$. We record the self-intersections of $f$ by the sum of signed group elements \[\mu(f):= \sum_{p\in f \pitchfork f} \varepsilon(p)\cdot \eta(p)\] as follows.
\begin{itemize}
\item For $p=f(x_1)=f(x_2)$ for $x_1\neq x_2\in \Sigma$, let $\gamma_1^p$ and $\gamma_2^p$ be paths in $\Sigma$ from the basepoint to $x_1$ and $x_2$ respectively.
\item The sign $\varepsilon(p)\in\{\pm 1\}$ is defined as follows. Transport the local orientation of $\Sigma$ at the basepoint to $x_1$ along $\gamma_1^p$, and along $\gamma_2^p$ to $x_2$. This induces a local orientation at $p$. Another local orientation is obtained by transporting the local orientation at the basepoint of $M$ to $p$ along the concatenated path $v_f \ast (f \circ \gamma_2^p)$. We define $\varepsilon(p)= 1$ when the two local orientations match at $p$, and $-1$ otherwise.
\item The element $\eta(p)\in \pi_1(M)$ is given by the concatenation $v_f \ast (f \circ \gamma_1^p) \ast (f \circ \gamma_2^p)^{-1} \ast  v_f^{-1}$.
\end{itemize}
There is a similar discussion about homotopy invariance of $\mu(f)$ as for $\lambda(f,g)$ earlier: homotopies $f_t$ that are generic immersions for all $t$ preserve the formal sum of signed elements but finger moves and Whitney moves (of $f$ with itself) create pairs $-\eta(p)+\eta(p)$ so it is convenient to use abelian groups. This takes care of regular homotopies of $f$ but there is an additional subtlety for cusp homotopies $f_t$ where there is exactly one time $t_0$ for which $f_{t_0}$ is not an immersion. These issues will be discussed carefully below.

For simply connected $\Sigma$, the self-intersection invariant $\mu(f)$ is well defined in the quotient (as an abelian group) of $\Z[\pi_1(M)]$ obtained by introducing the relators
\[
\gamma - w^M(\gamma)\cdot\gamma^{-1}\]
 for all $\gamma\in\pi_1(M)$. For general $\Sigma$, the quantity $\mu(f)$ is well defined in the abelian group
\begin{equation}\label{eqn:wt-Gamma-ff}
  \Gamma_{f} := \Gamma_{f,f} / \langle \gamma- w^M(\gamma)\cdot\gamma^{-1}\rangle,
\end{equation}
i.e.\ in this quotient of $\Gamma_{f,f}$ from \cref{defn:Gamma-f-g}. Here, as above $w^M \colon \pi_1(M) \to \{\pm 1\}$ is the orientation character.

We now change our notation slightly from the discussion of $\lambda(f,g)$ in order to work in this further quotient.
Let $\sim$ denote the equivalence relation on $\pm\pi_1(M)$ induced by the composition
\[\pm\pi_1(M) \hooklongrightarrow \Z[\pi_1(M)]\longtwoheadrightarrow \Gamma_{f}\]
sending $a\mapsto [a]$ and let $|\pi_1(M)|$ be the quotient of $\pm\pi_1(M)$ obtained by identifying $\gamma_1$ and $\gamma_2$ whenever
$[\gamma_1]= \pm [\gamma_2] \in \Gamma_f$.
Then we obtain the following analogues of \cref{lem:double coset,lem:Gammafg}.

\begin{lemma}\label{lem:double coset f}
Let $\gamma_1,\gamma_2\in\pi_1(M)$. One of the relations $[\gamma_1]= \pm [\gamma_2]$ holds if and only if $\gamma_1$ and $\gamma_2$ represent the same element in the quotient of the double coset by inversion. In other words, the identity map induces a bijection
\[
|\pi_1(M)| \longleftrightarrow \smfrac{\left(f_\bullet\pi_1(\Sigma)\backslash \pi_1(M) / f_\bullet\pi_1(\Sigma)\right)}{\approx}
\]
where $\approx$ is the equivalence relation identifying $\gamma$ and $\gamma^{-1}$ for all $\gamma \in \pi_1(M)$.
\end{lemma}

We write $|\gamma| := \p(\gamma)$ for the quotient map $\p \colon \smfrac{\pm\pi_1(M)}{\sim}  \sra  |\pi_1(M)|$. Again $\p$ has fibres of order 1 or 2 and we decompose $|\pi_1(M)|$ as a disjoint union $B_1\sqcup B_2$ according to this distinction as before. Choose a section $s\colon  |\pi_1(M)| \to \smfrac{\pm\pi_1(M)}{\sim}$ of $\p$.
As before, for $b \in B_2$ we denote the elements of the fibre by $\p^{-1}(b) = \{s(b),-s(b)\}$.

\begin{lemma}\label{lem:Gammaf}
The abelian group $\Gamma_{f}$ is a direct sum
$\Gamma_{f} = FA \oplus V$ of a free abelian group $FA$ on the set  $s(B_2)\subseteq \smfrac{\pm\pi_1(M)}{\sim} \subseteq \Gamma_{f}$ and a $\Z/2$ vector space $V$ with basis $s(B_1)=\p^{-1}(B_1)\subseteq \smfrac{\pm\pi_1(M)}{\sim}$.

Reading off the coefficients in this decomposition gives homomorphisms $c_{s(b)}\colon \Gamma_{f}\to\Z$ for $b\in B_2,$ and $c_b\colon \Gamma_{f}\to\Z/2$ for $b\in B_1,$ leading to a decomposition of $\Gamma_{f}$ as a direct sum of copies of $\Z$ and $\Z/2$.
 \end{lemma}

\begin{proof}
The proof is analogous to that of \cref{lem:Gammafg}.
\end{proof}

\begin{remark}
  If $M$ and $\Sigma$ are oriented, then
  \[\Gamma_f \cong \Z[f_\bullet\pi_1(\Sigma)\backslash \pi_1(M) / f_\bullet\pi_1(\Sigma)]/\langle\gamma -  \gamma^{-1}\rangle = \Z\bigl[ \smfrac{\left(f_\bullet\pi_1(\Sigma)\backslash \pi_1(M) / f_\bullet\pi_1(\Sigma)\right)}{\approx}\bigr] \]
is free abelian.  In this case $B_1 = \emptyset$ and $B_2 = |\pi_1(M)| = \smfrac{\left(f_\bullet\pi_1(\Sigma)\backslash \pi_1(M) / f_\bullet\pi_1(\Sigma)\right)}{\approx}$ by \cref{lem:double coset f}.

Consider instead $f\colon \RP\looparrowright M$ with $f_\bullet(\mathbb{RP}^1)=1$, without any assumption on $M$.
Then \[\Gamma_{f} \cong (\Z/2)[\pi_1(M)]/\langle\gamma - w^M(\gamma) \cdot \gamma^{-1}\rangle  = (\Z/2)\bigl[\smfrac{\pi_1(M)}{\approx}\bigr].\]
As in \cref{remark:examples-Gamma-fg} this is 2-torsion due to the relations $\gamma =w^{\RP}(\mathbb{RP}^1)\cdot f_\bullet(\mathbb{RP}^1)\ast \gamma=-\gamma$. In this case $B_1 =|\pi_1(M)| = \smfrac{\pi_1(M)}{\approx}$ and $B_2 = \emptyset$.
\end{remark}

As before the subgroups $FA$ and $V$ of $\Gamma_{f}$ do not depend on the choice of $s$, only the basis of $FA$ does. As a consequence, the coefficient maps $c_{s(b)}$ only depend on $s$ up to sign and satisfy $c_{-s(b)} = - c_{s(b)}$. Given $a\in \smfrac{\pm\pi_1(M)}{\sim}$ we may again take $s(\p(a)):=a$ to get $c_a$ and in particular $c_\gamma$ for $\gamma\in\pi_1(M)$ independent of the choice of $s$ at other points. This gives the following definition.

\begin{definition}\label{def:coeff-at-gamma}
For $\gamma\in \pi_1(M)$, we write $\mu(f)_\gamma:= c_{\gamma}(\mu(f))$. This quantity lies in $\Z$ (respectively, $\Z/2$) when $|\gamma|$ lies in $B_2$ (respectively $B_1$), or equivalently when $[\gamma]$ has infinite order (respectively, order~2) in $\Gamma_f$. The values does not depend on the choice  of $s$  and satisfies  $c_{\gamma_1} = - c_{\gamma_2}$ whenever $[\gamma_1] = -[\gamma_2]$.
\end{definition}

We focus on $\mu(f)_1$, which plays an important r\^{o}le in the distinction between the homotopy class and regular homotopy class of $f$, as we will discuss in the next subsection. In the usual case, where $\Sigma$ is simply connected, $\mu(f)_1\in \Z$. However, in general $\mu(f)_1$ may lie in either $\Z$ or $\Z/2$. The following lemma gives the precise conditions determining the home of $\mu(f)_1$.

\begin{lemma}\label{lem:mu1}
 Let $f\colon \Sigma\looparrowright M$ be a based, generic immersion, with whisker $v$. Recall that the map $f_\bullet\colon \pi_1(\Sigma)\to \pi_1(M)$ is given by $\alpha\mapsto v\ast (f\circ \alpha) \ast v^{-1}$.
	
	If $w^\Sigma$ is trivial on $\ker(f_\bullet)$ and $w^M$ is trivial on $\Ima(f_\bullet)$, then $[1]\in \Gamma_f$ has infinite order and thus $\mu(f)_1\in\Z$. Otherwise, $[1]$ has order 2 and $\mu(f)_1\in \Z/2$.
\end{lemma}

\begin{proof}
	By definition, for $1\in\pi_1(M)$, we know that $[1]\in \Gamma_{f,f}$ has order 2 precisely if (i) there exists $\alpha, \beta\in \pi_1(\Sigma)$ such that $f_\bullet(\alpha)\ast f_\bullet(\beta)=f_\bullet(\alpha\ast \beta)=1$ and $ w^\Sigma(\alpha)w^\Sigma(\beta)w^M(f_\bullet(\beta))=w^\Sigma(\alpha\ast \beta)w^M(f_\bullet(\beta))=-1$, or (ii) there exists $\delta\sim 1$ where $\delta$ has order two in $\pi_1(M)$ and $w^M(\delta)=-1$.

	Suppose that $w^\Sigma$ is trivial on $\ker(f_\bullet)$ and $w^M$ is trivial on $\Ima(f_\bullet)$. Then the first case (i) cannot happen, since if $f_\bullet(\alpha\ast \beta)=1$, then $\alpha\ast \beta \in \ker(f_\bullet)$ so $w^\Sigma(\alpha\ast \beta)w^M(f_\bullet(\beta))=1\cdot 1=1$. Similarly, (ii) cannot happen: suppose $\delta$ is order two in $\pi_1(M)$ and $\delta\sim 1$. Then by definition, $\delta= f_\bullet(\alpha)\ast 1\ast f_\bullet(\beta)=f_\bullet(\alpha\ast \beta)$ in $\pi_1(M)$, for some $\alpha,\beta\in \pi_1(M)$. In particular, $\delta\in \Ima(f_\bullet)$, and so again $w^M(\delta)=1$ by hypothesis, contradicting (ii). Therefore, $[1]$ has infinite order as claimed.
	
	Now suppose there is some $\alpha\in\pi_1(\Sigma)$ with $w^M(f_\bullet(\alpha))=-1$. Then we have $f_\bullet(\alpha^{-1}) \ast 1\ast f_\bullet(\alpha)=1$ and $w^{\Sigma}(\alpha^{-1})w^{\Sigma}(\alpha)w^M(f_\bullet(\alpha))=-1$, so $[1]\sim -[1]$ and $[1]$ has order two.
	
	Finally suppose that there is some $\alpha\in \ker(f_\bullet)$ with $w^\Sigma(\alpha)=-1$. Then we have $f_\bullet(\alpha) \ast 1 \ast f_\bullet(1_\Sigma) = 1$ and $w^{\Sigma}(\alpha)w^{\Sigma}(1_\Sigma) w^M(f_\bullet(1_\Sigma))=-1$, where $1_\Sigma$ denotes the trivial element in $\pi_1(\Sigma)$. Then again we have $[1]\sim -[1]$ and $[1]$ has order two.
\end{proof}

As with \cref{prop:regular-homotopy-inv-lambda} the proof of the following proposition is virtually identical to the case of discs and spheres using~\cref{prop:hom-gen-immersion} (see e.g.~\cite{DET-book-detintro}), and we leave it for the interested reader.

\begin{proposition}\label{prop:regular-homotopy-inv-mu}
Let $f\colon \Sigma\looparrowright M$ be a based generic immersion. The self-intersection number $\mu(f)$ is preserved under regular homotopies that are ambient isotopies near $\partial\Sigma$.
\end{proposition}

In this and the previous subsection, we have considered intersection and self-intersection numbers of connected surfaces. By combining these invariants, we can define the conglomerate notion of self-intersection number for disconnected surfaces $F = \{f_i\}_{i=1}^m \colon \Sigma \to M$, as considered in \cref{def:mu(F)}:
\[\mu (F) := \sum_{i < j} \lambda(f_i,f_j)+\sum_i\mu(f_i)\in \bigoplus_{i < j} \Gamma_{f_i,f_j}\oplus \bigoplus_i\Gamma_{f_i}.\]
\cref{prop:regular-homotopy-inv-lambda,prop:regular-homotopy-inv-mu} imply that $\mu(F)$ preserved under regular homotopies of $F$ that are ambient isotopies near $\partial\Sigma$.

\subsection{Whitney discs}\label{subsection:Whitney-discs}

A Whitney move cancels a pair of double points of a generic immersion $F\colon \Sigma\imra M$ as in \cref{convention}, provided all the assumptions on the guiding Whitney disc are satisfied. In our setting where $\Sigma$ and $M$ need be neither simply connected nor orientable, this requires some care. We start with the notion of arcs $A$ and $A'$ pairing double points $p$ and $q$, and the corresponding notion of $(p,q, A, A')$ having \emph{opposite sign}.

\begin{definition}\label{defn:paired-arcs-opposite-sign}
Let $f\colon \Sigma\to M$ and $g\colon \Sigma' \to M$ be based maps that either intersect transversely, or $f=g$ and $f$ is a generic immersion.
We say that two points $p,q\in f\pitchfork g\subseteq M$ are \emph{paired by arcs} if we equip them with the extra data of an arc $A\colon [0,1]\to\Sigma$ from $f^{-1}(p)$ to $f^{-1}(q)$ and an arc $A'$ in $\Sigma'$ from $g^{-1}(q)$ to $g^{-1}(p)$. In the case that $f=g$ we require that each point in $f^{-1}(p)$ and in $f^{-1}(q)$ is the endpoint of precisely one of the arcs $A$ and $A'$, i.e.\ $A$ and $A'$ lie in distinct sheets at both $p$ and $q$.
\end{definition}

With the extra data of the arcs $A$ and $A'$, we can make sense of whether two intersection points that are paired by arcs have opposite sign.

\begin{definition}\label{defn:opposite-signs}
Let $f\colon \Sigma\to M$ and $g\colon \Sigma' \to M$ be based maps that either intersect transversely, or $f=g$ and $f$ is a generic immersion.
Two intersection points $p,q\in f\pitchfork g\subseteq M$ paired by arcs $A$ in $\Sigma$ and $A'$ in $\Sigma'$ have \emph{opposite sign} if the following holds.
    Fix local orientations of $\Sigma$ at $f^{-1}(p)$ and of $\Sigma'$ at $g^{-1}(p)$. This choice induces a local orientation of $M$ at $p$. Transport the local orientation of $\Sigma$ from $f^{-1}(p)$ to $f^{-1}(q)$ along $A$, and the local orientation of $\Sigma'$ from $g^{-1}(p)$ to $g^{-1}(q)$ along $A'$. This gives a local orientation of $M$ at the point $q$. Compare this with the local orientation on $M$ at $q$ induced by transporting the local orientation from $p$ to $q$ along the arc $f\circ A$. If these orientations disagree then the points $p,q$ are said to have opposite sign (with respect to $A,A'$) and otherwise they are said to have the same sign.
The dependence on the choice of arcs $A$ and $A'$ is sometimes neglected.
\end{definition}

Note that double points having the same sign could be ``paired'' by an embedded disc, but this does not mean that a Whitney move using this disc is possible, because the required section of the normal bundle of the disc is not available; in this case any rank one sub-bundle of the normal bundle of the disc, restricted to the boundary, that is tangent to one sheet of $\Sigma$ and normal to the other sheet, turns out to be a M\"{o}bius bundle. So one does not study such discs and assumes that a Whitney disc always pairs two double points of opposite sign.

In the setting of based transverse maps $f\neq g$, with $\Sigma$ and $\Sigma'$ connected, recall from \cref{subsection:intersection-numbers} that $\lambda(f,g)$ is a sum of terms $\varepsilon(p)\cdot\eta(p)$, one for each double point $p\in f\pitchfork g$, with $\eta(p)\in \pi_1(M)$ and $\varepsilon(p)\in\{\pm 1\}$. This sum is well defined in the abelian group $\Gamma_{f,g}$ and each signed group element $a\in \pm\pi_1(M)$ represents a unique element $[a]\in\Gamma_{f,g}$. The same proof as in the case of simply connected surfaces~\cite{DET-book-detintro}*{Proposition~11.10} yields the following result.

\begin{lemma}\label{lem:pairing}
Let $\Sigma$ and $\Sigma'$ be compact connected surfaces and let $f\colon \Sigma\to M$ and $g \colon \Sigma'\to M$ be based maps with transverse double points $p,q\in f\pitchfork g\subseteq M$. Then $[\varepsilon(p)\cdot \eta(p) + \varepsilon(q)\cdot \eta(q)]=0 \in \Gamma_{f,g}$ if and only if $p$ and $q$ can be paired by arcs $A\subseteq\Sigma$ and $A'\subseteq\Sigma'$ such that
\begin{enumerate}[(i)]
\item\label{lem:pairing-item-i} the closed loop $f \circ A\cup_{p,q} g \circ A'$ is null-homotopic in $M$, and
\item\label{lem:pairing-item-ii} the points $p$ and $q$ have opposite sign with respect to the arcs $A$ and $A'$.
\end{enumerate}
\end{lemma}

If \eqref{lem:pairing-item-i} and \eqref{lem:pairing-item-ii} are satisfied for $p$ and $q$, we say that $W \colon D^2\to M$ is a (map of a) \emph{Whitney disc pairing $p$ and $q$} if its boundary is the closed loop in \eqref{lem:pairing-item-i}, the union of its two \emph{Whitney arcs} $f \circ A$ and $g \circ A'$.
We leave it to the reader to formulate the analogous notion for a pair of transverse self-intersection points of $f \colon \Sigma\to M$.
This gives rise to the following corollary to \cref{lem:pairing}.

\begin{corollary}\label{cor:vanishing}
Let $\Sigma$ and $\Sigma'$ be compact connected surfaces and let $f\colon \Sigma\to M$ and $g \colon \Sigma'\to M$ be transverse based maps. Then $\lambda(f,g)=0$ if and only if all intersection points between $f$ and $g$ can be paired by maps of Whitney discs.

Moreover, a based generic immersion $f \colon \Sigma\imra M$ satisfies $\mu(f)=0$ if and only if all self-intersection points of $f$ can be paired by maps of Whitney discs.
\end{corollary}

Note that by the geometric Casson lemma (\cref{lem:geometric-casson-lemma}) the vanishing  of $\lambda(f,g)$ is equivalent to the existence of a regular homotopy of $f$ and $g$ that makes their images disjoint, at the cost of introducing self-intersections in $f$ and $g$. There is no analogue of this argument if $\mu(f)=0$ by the failure of the Whitney trick in dimension 4, as for example exhibited by the secondary embedding obstruction $\km(f)$ (see \cref{section:KM-invariant}).

By the geometric characterisation in \cref{cor:vanishing}, it is meaningful to refer to $f$ and $g$ having trivial intersection number, and to $f$ having trivial self-intersection number, without using a basing.

The analogue of the characterisation in the second part of \cref{cor:vanishing}  holds for generic immersions $F=\{f_i\}_{i=1}^m \colon \Sigma\imra M$ from \cref{convention}, i.e.\ for compact but possibly disconnected domains, if we use \cref{def:mu(F)} from the introduction for the self-intersection number $\mu(F)$.

\begin{corollary}\label{cor:vanishing-pairing}
Let $F\colon (\Sigma,\partial\Sigma)\looparrowright (M,\partial M)$ be as in \cref{convention}. 
Then $\mu(F)=0$  if and only if the double points of $F$ can be paired by maps of Whitney discs.
\end{corollary}

\begin{proof}
This is a direct consequence of \cref{prop:regular-homotopy-inv-lambda,prop:regular-homotopy-inv-mu,cor:vanishing}  because every double point of $F$ is either a self-intersection point of a component $f_i$ or an intersection point between distinct components $f_i$ and $f_j$ (where we can assume that $i<j$). Note that both cases represent self-intersection points of~$F$.
\end{proof}

Collections of Whitney discs as above may be assumed to be \emph{convenient} in the following sense (see e.g.~\citelist{\cite{FQ}*{Section~1.4}\cite{DET-book-detintro}}).

\begin{definition}\label{def:convenient-Whitney}
Let $F\colon (\Sigma,\partial \Sigma)\looparrowright (M, \partial M)$ be as in \cref{convention}. A \emph{convenient} collection of Whitney discs for $F$ is a collection of framed, generically immersed Whitney discs pairing all the double points of $F$, with interiors transverse to $F$,  and with disjointly embedded boundaries.  A collection of arcs in $F(\Sigma)$ is called a collection of \emph{Whitney arcs} if the union of the arcs is the boundary of a convenient collection of Whitney discs.
\end{definition}

By pushing double points of a convenient collection across the boundaries of  Whitney discs~\cite{DET-book-detintro}*{Figure~11.4}, we may further assume that all Whitney discs are pairwise disjoint and embedded. However, the (resulting and pre-existing) intersections between the original surface $F$ and the Whitney discs can in general not be removed, as detected by the secondary invariant $\km(F)$.

\subsection{Homotopy versus regular homotopy of generic immersions}\label{subsection:hom-vs-reg-hom}

Let $f\colon \Sigma\looparrowright M$ be a generic immersion.
Local orientations of $M$ and $\Sigma$ determine a local orientation of $\nu f$. Hence, given a framing of $f\vert_{\partial \Sigma}$, one can define a relative Euler class of the normal bundle $\nu f$ in $H^2(\Sigma,\partial \Sigma;\Z^{w_1(\nu f)})$.  If  $f^*(w_1(M))=w_1(\nu f) + w_1(T\Sigma)= 0$ then the local orientation of $\Sigma$ determines a Poincar\'{e} duality isomorphism from this twisted cohomology group to $\Z$, and we denote the resulting integer by $e(\nu f)$. Note that $e(\nu f)$ does not depend on the local orientation of $\Sigma$ but only on the local orientation of $M$.
If $f^*(w_1(M)) \neq 0$ then there is still a mod 2 normal Euler number, which we also denote by $e(\nu f) \in \Z/2$.

A useful interpretation of $e(\nu f)$ is as follows. A vector in $\R^2$ together with the framing  of $f\vert_{\partial \Sigma}$ determines a nonvanishing section of $\nu f$ on $f(\partial \Sigma)$. Extend this to a section of $\nu f$ over all of $f(\Sigma)$, transverse to the zero section. Then $e(\nu f)$ counts, with sign, the number of zeros of the section, in $\Z$ or $\Z/2$ as appropriate.

Next we give an extension of \cite{PRT20}*{Theorem 1.2} from the simply connected to the general setting, restricting ourselves to the case of connected $\Sigma$ for convenience.
We note that \cite{PRT20}*{Theorem 1.2} was based on \cite{FQ}*{Lemma~1.2~and~Proposition~1.6}, but that the latter proposition was not proven in \cite{FQ}.

 By the following theorem, in some cases, for example if $M$ is orientable, then $e(\nu f) \in \Z$ is an additional invariant of regular homotopy classes of immersions. It changes by $\pm 2$ during a cusp homotopy (see e.g.~\cite{CST-classical}*{Figure~19}) and hence can be infinitely many regular homotopy classes of immersions, that are all  homotopic as continuous maps.

In \cref{theorem:generic-immersions-bijection}, in the case that $\Sigma$ has nontrivial boundary, we fix a framing on the embedding $f\vert_{\partial \Sigma}$, in order to define the relative Euler number $e(\nu \wt f)$, for $\wt f$ any generic immersion homotopic to $f$.

\begin{theorem}\label{theorem:generic-immersions-bijection}
	Let $\Sigma$ be a compact, connected surface and let $M$ be a 4-manifold. Then the inclusion of the subspace of generic immersions $\Imm(\Sigma,M)$ in the space of all continuous maps induces a map
	\[
	\frac{\Imm(\Sigma,M)} { \{\text{regular homotopy}\}} \xrightarrow{\phantom{5}i\phantom{5}} [\Sigma,M]_{\partial}, \]
	where $[\Sigma,M]_{\partial}$ denotes the set of homotopy classes of continuous maps that restrict on $\partial\Sigma$ to embeddings disjoint from the image of the interior of $\Sigma$.
	\begin{enumerate}
		\item\label{thm:gib-i} $i$ is surjective.
		\item\label{thm:gib-ii} The fibres of $i$ are related by cusp homotopies. More precisely, suppose that $f$ and $g$ are homotopic generic immersions. Then we can add cusps to $f$ and $g$, to obtain $f'$ and $g'$ respectively, such that $f'$ and $g'$ are regularly homotopic.
		\item\label{thm:gib-iii} For every $f\in [\Sigma,M]_{\partial}$, there is a bijection
			\[i^{-1}(f)\cong \begin{cases}
			2\Z&\text{if }f^*(w_1(M))=0 \text{ and } w_2(\nu\wt f)=0;\\[1ex]
			2\Z+1&\text{if }f^*(w_1(M))=0 \text{ and } w_2(\nu\wt f)=1;\\[1ex]
			\Z/2&\text{otherwise;}
		\end{cases}\]
		where $\nu\wt f$ is a normal bundle for $\wt f$, a generic immersion in $i^{-1}(f)$.
		In the cases where $f^*(w_1(M))=0$, the bijection is given by \[\wt f\mapsto e(\nu \wt{f}).\]
		 Otherwise the bijection is given by
		\[\wt f\longmapsto \mu(\wt f)_1\in \Z/2.\]
		\item\label{thm:gib-iv} If $f^*(w_1(M))=0$ and $w_1(\Sigma)|_{\ker(f_\bullet)}=0$, for $\wt f$ a generic immersion in $i^{-1}(f)$, the quantities $\mu(\wt{f})_1$ and $e(\nu \wt{f})$ are related by the formula
		\[
		\lambda(\wt f,\wt f)_1=2\mu(\wt f)_1 +e(\nu\wt{f}) \in \Z
		\]
and so $\mu(\wt f)_1\in\Z$ also detects the regular homotopy class of $\wt f \in i^{-1}(f)$.
	\end{enumerate}
\end{theorem}

While we prefer the upcoming direct argument analysing singularities, \cref{theorem:generic-immersions-bijection} could in principle also be proven via Smale--Hirsch immersion theory, which has a version in the topological category.
The main novelty of the theorem is that we give precise conditions in terms of the Stiefel--Whitney classes to control how large the fibres of $i$ are, and which invariants detect them.

\begin{proof}[Proof of \cref{theorem:generic-immersions-bijection}]
By \cref{prop:hom-gen-immersion}\,\eqref{item:hom-to-gen-imm-1}, the map is surjective. That is, every homotopy class contains a generic immersion. This proves \eqref{thm:gib-i}.

For~\eqref{thm:gib-ii}, note that if $f$ and $g$ are homotopic generic immersions, then by \cref{prop:hom-gen-immersion}\,\eqref{item:hom-to-gen-imm-2} there exists a generic homotopy $H$ between them, which by definition is a sequence of ambient isotopies, finger moves, Whitney moves, and cusp homotopies. We can modify $H$ such that there are real numbers $t_1 < t_2 \in [0,1]$ such that the singularities of $H$ in $[0,t_1]$ only consist of cusp homotopies that create double points, the singularities in $[t_1,t_2]$ only consist of finger moves and Whitney moves, and those in $[t_2,1]$ only consist of cusp homotopies that remove double points. The statement then follows by taking $f':=H_{t_1}$ and $g':=H_{t_2}$.

To achieve this modification, note that we can bring all the creating cusp singularities forward, so that they occur earlier, and we can delay all the removing cusps.  To arrange for a creating cusp to be rearranged earlier than a finger or Whitney move,  choose an arc in the image of $H$ starting from $H_t(\Sigma)$, for some $t \in (0,t_1)$, and ending at the cusp, which intersects each level in a point and is disjoint from all Whitney arcs and double points.  The homotopy can then be altered in a neighbourhood of this arc so that the cusp singularity occurs at time $t$.
Delaying a removing cusp is the same procedure but with the direction of time reversed. This completes the proof of~\eqref{thm:gib-ii}.
		
The proof of~\eqref{thm:gib-iii} splits naturally into two cases.

\setcounter{case}{0}
\begin{case}
$f^*(w_1(M))=0$.
\end{case}

As noted in \cref{sec:basic}, the sign of an intersection point is not always well defined. Nevertheless, in the case that $f^*(w_1(M))=0$ the sign of a cusp homotopy is well defined. The key point is that a cusp not only specifies a double point $p$ but also an arc between the preimages of $p$. In the case that $f^*(w_1(M))=0$, using this path, the sign of the double point $p$ is well defined, independent of the choice of path transporting the local orientation at the basepoint to the double point. Thus in this setting we define the \emph{sign} of a cusp to be the sign of the double point it creates or removes. We will use the terminology of \emph{creating cusps} for cusps that create a double point and \emph{removing cusps} for those that remove a double point.

Since $f^*(w_1(M))=0$, $e(\nu \wt f)$ is defined in $\Z$ for any generic immersion $\wt f$ homotopic to $f$. Recall that $w_2(\nu\wt{f})\equiv e(\nu\wt{f})\mod{2}$. Since regularly homotopic generic immersions have equal Euler numbers, the map in the theorem statement is well defined on equivalence classes in the domain of $i$.  Note that a cusp homotopy changes $e(\nu \wt f)$ by $2$ or $-2$, depending on the sign of the cusp and whether it is a creating or a removing cusp.
So every element of $2\Z$ or $2\Z+1$, depending on $w_2(\nu\wt{f})$, can be realised as the Euler number of a generic immersion in $i^{-1}(f)$.

To complete the proof when $f^*(w_1(M))=0$, it remains to show injectivity. We will show that given a generic homotopy between generic immersions with equal Euler numbers, we can modify the homotopy to cancel cusps, until we are left with a regular homotopy.
		
	First, note that when we have a removing cusp, and later in the homotopy we have a creating cusp with the same sign, then we can cancel these two cusps along a level-preserving path in the homotopy as indicated in \cref{fig:cuspcanceling}.
	
	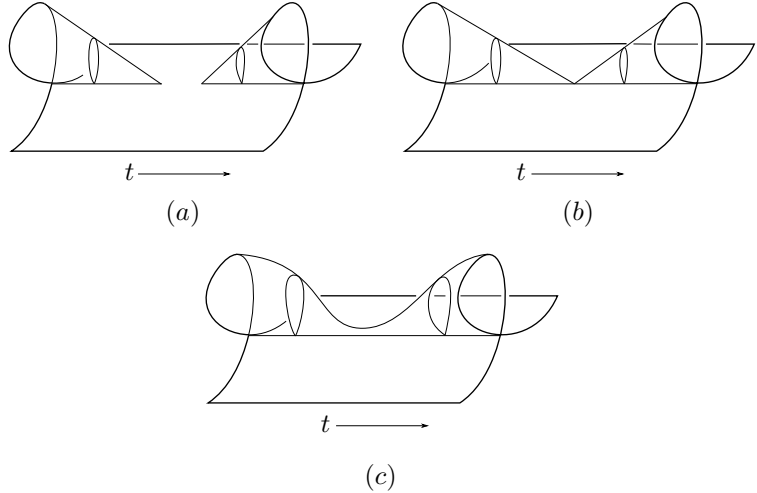
\begin{figure}[htb]
    \centering
    \begin{tikzpicture}	
            \node[anchor=south west,inner sep=0] at (0,-3.35){\includegraphics{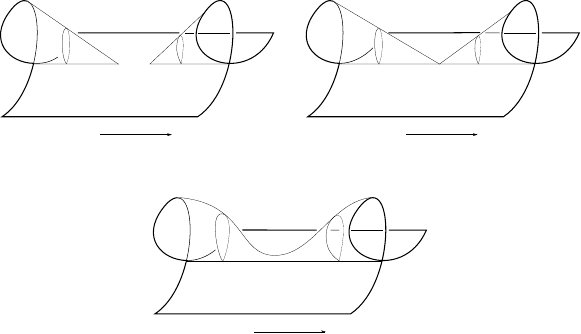}};
			\node at (1.6,0.05) {$t$};
			\node at (6.75,0.05) {$t$};
			\node at (4.175,-3.3) {$t$};
			\node at (2.3,-0.5) {$(a)$};
			\node at (7.5,-0.5) {$(b)$};
			\node at (4.9,-4) {$(c)$};
        \end{tikzpicture}
    \caption{A schematic picture showing how a removing cusp singularity and creating cusp singularity with the same sign can be cancelled. In each of (a), (b), and (c), a homotopy is traced out in the direction of $t$. At every time $t$, except the times of the cusp singularities, we depict an arc of a generic immersion homotopic to $f$. (a) Two cusp singularities are shown: a removing cusp occurring first, followed by a creating cusp of the same sign. (b)  Modify the homotopy, delaying the removing cusp until it coincides with the creating cusp. This involves choosing an arc in $\Sigma$ joining the two cusp points. (c) A further local modification removes the two cusps.}
\label{fig:cuspcanceling}
\end{figure}

However this is not sufficient. We also have to show that we can also cancel cusps in the following two situations.
\begin{enumerate}[(a)]
			\item\label{item:cusp-cancelling-a} Two creating cusps of opposite sign, or two removing cusps of opposite sign.
			\item\label{item:cusp-cancelling-b} A creating cusp paired with a later removing cusp, both of the same sign.
		\end{enumerate}

		Suppose that we have a generic homotopy $H$ between generic immersions with equal Euler numbers consisting of two creating cusp homotopies of opposite sign, as in \eqref{item:cusp-cancelling-a}. Create a self-homotopy $H_0$ of the starting immersion, i.e.\ the immersion at $t=0$, consisting of a trivial finger move together with two removing cusps for the double points created by the finger move, as shown in~\cref{fig:fingerthencusp}. Then concatenate $H_0$ with the original homotopy $H$. The new homotopy can be modified as in \cref{fig:cuspcanceling} to cancel the removing cusps in $H_0$ and the creating cusps in $H$, leaving only the finger move behind. An analogous argument shows how to cancel two removing cusps of opposite sign, this time concatenating at the end of $H$.
		
\begin{figure}[htb]
    \centering
    \begin{tikzpicture}	
            \node[anchor=south west,inner sep=0] at (0,0){\includegraphics{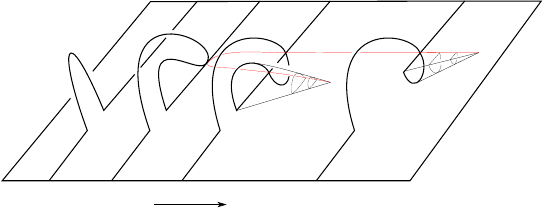}};
			\node at (2.5,0.05) {$t$};
        \end{tikzpicture}		
\caption{A schematic picture showing a self-homotopy consisting of a trivial self-finger move followed by two removing cusps. The homotopy is traced out in the direction of $t$. At every time $t$, except at the times of the cusp and finger move singularities, we depict an arc of a generic immersion homotopic to $f$. In red we show the arc of self-intersections of $f$ -- note that it starts at one cusp singularity and ends at the other. }
\label{fig:fingerthencusp}
\end{figure}		
		
		Similarly, for the situation in \eqref{item:cusp-cancelling-b}, suppose that we have a generic homotopy $H$ between generic immersions with equal Euler numbers consisting of a creating cusp and a later removing cusp of the same sign. Again we construct the self-homotopy $H_0$ and concatenate with $H$. In the result, we use the procedure from \cref{fig:cuspcanceling} to cancel the creating cusp in $H$ with one of the removing cusps in $H_0$. This entire operation has so far replaced a cusp with a cusp of opposite sign and direction. As before we can repeat the operation at the end of the homotopy to replace the removing cusp with a creating one, also with the opposite sign.  Thus when we have a creating cusp with a later removing cusp of the same sign, we can replace both by cusps of opposite sign and direction. Since now the removing cusp happens before the creating cusp, the two can be cancelled and we are done with case~\eqref{item:cusp-cancelling-b}. This completes the proof of \eqref{thm:gib-iii} in the case that $f^*(w_1(M))=0$.

\begin{case}
$f^*(w_1(M))\neq 0$.
\end{case}

Note that a cusp homotopy changes $\mu(\wt f)_1\in\Z/2$ by one. So both values of $\Z/2$ can be realised within the homotopy class. To show injectivity in this case,  we have to show that we can cancel cusps in a homotopy in arbitrary pairs. First use the trading argument above to get all the removing cusps before the creating cusps in the homotopy. Then for any pair of cusps, one removing and one creating,  choose some level-preserving path in the homotopy between  the first and the second cusp, and restrict to a small disc containing the path. If they have the same sign with respect to this disc, cancel the two cusps as before.

If they have opposite signs, change the choice of the arc to arrange that the union of the new arc and the old arc maps nontrivially under $w_1(M)$. Such an arc exists since $f^*(w_1(M))$ is nontrivial and $\Sigma$ is connected. With this new choice, the signs of the cusps in the disc become the same and we can again cancel the cusps. This completes the proof of both halves of \eqref{thm:gib-iii}.

Finally for~\eqref{thm:gib-iv} note that if $f^*(w_1(M))=0$ and $w_1(\Sigma)|_{\ker(f_\bullet)}=0$, then recall that by \cref{lem:mu1} that $\mu(\wt{f})_1$ is well defined in $\Z$. By the discussion above the statement of the theorem, $e(\nu \wt f)$ is also well defined in $\Z$. In this case the formula
\[
\lambda(\wt f,\wt f)_1=2\mu(\wt f)_1 +e(\nu\wt{f}) \in \Z
\]
holds by the proof of the corresponding fact for discs and spheres (see e.g.\ \cite{DET-book-detintro}*{Proposition~11.8}). Any cusp homotopy leaves $\lambda(\wt f, \wt f)_1$ unchanged, while it changes $\mu(\wt f)_1$ by $\pm 1$. By the formula, it changes $e(\nu \wt f)$ by $\mp 2$.
Thus if $\wt f$ and $\wt f'$ are generic immersions homotopic to $f$, then $e(\nu \wt f)=e(\nu \wt f')$ if and only if $\mu(\wt f)_1=\mu(\wt f')_1$. Hence \eqref{thm:gib-iv} follows from \eqref{thm:gib-iii}.
\end{proof}

\section{Secondary embedding obstructions} \label{section:KM-invariant}

The Whitney trick implies that every map $F\colon S^n\to M^{2n}$ is homotopic to an embedding, whenever $M$ is a simply connected $2n$-dimensional manifold and $n > 2$. In order to prove the failure of the Whitney trick in dimension 4, Kervaire--Milnor devised an obstruction in \cite{Kervaire-Milnor:1961-1} that gave counterexamples to the above statement for $n=2$. They showed that the homotopy class of $3\cdot\CP^1$ is not represented by an embedded sphere in $\CP^2$. In a smooth, oriented, closed 4-manifold $M$, consider the formula
\begin{equation}\label{eq:KM}
\theta(c):= \frac{c\cdot c-\sigma(M)}{ 8} \mod 2,
\end{equation}
where the $\Z/2$-reduction of $c\in H_2(M;\Z)$ is Poincar\'{e} dual to $w_2(M)$ and $\sigma(M)$ is the signature of the intersection form $(x,y)\mapsto x\cdot y$ on $H_2(M;\Z)$. In this setting, if $c$ is represented by an embedded sphere, then $\theta(c)=0$. Recall that for a unimodular form $\ell$ and a \emph{characteristic} element $c$, i.e.\ one satisfying $\ell(c,x) \equiv \ell(x,x) \mod 2$, the difference $\ell(c,c) - \sigma(\ell)$ is always divisible by 8. The condition on $c$ being dual to $w_2(M)$ is stronger than being characteristic for the intersection form since the mod~2 intersection condition holds for all $x\in H_2(M;\Z/2)$, not just for integral homology classes. For example, if $M$ is the Enriques surface (double covered by the $K3$ surface) then $\theta(0)\neq 0$, so $0$ cannot be dual to $w_2(M)$, even though the intersection form on $H_2(M;\Z)$ is even.

For the proof that $\theta$ is an embedding obstruction, Kervaire--Milnor add $(1-[F]\cdot [F])$ copies of $(\CP^2,\CP^1)$ to a proposed characteristic pair $(M,F\colon S^2\hra M)$, with $F$ assumed to be an embedding, to obtain an embedded sphere with self-intersection number $1$. Then they blow down that characteristic sphere to arrive at a spin manifold $M'$ with $\sigma(M') = \sigma(M) +(1- [F] \cdot [F]) - 1 = \sigma(M) - [F] \cdot [F]$. Rochlin's theorem~\cite{Rochlin}, that the signature of a smooth, closed, spin 4-manifold is divisible by 16, is equivalent to the original condition $\theta([F])=0$ in $M$.

The Kervaire--Milnor result also has consequences for spin manifolds, where it says that any (characteristic) homology class $c = 2 b$ that is represented by an embedded sphere must satisfy $b \cdot b \equiv 0\, \mod 4$. For example, $2\Delta\in \pi_2(S^2 \times S^2)$ is not represented by an embedding for $\Delta$ the diagonal 2-sphere.

For about a decade, it remained an open problem to find a combinatorial formula for $\theta(c)$ in terms of geometric representatives for $c$.

\subsection{Combinatorial formulae: Rochlin's Arf invariant}\label{sec:Rochlin}

Rochlin picked an embedded representative $F\colon \Sigma \hra M$ for $c\in H_2(M;\Z)$ as above and assumed that $H_1(M;\Z/2)$ vanishes  \cite{Rochlin:71}. Any simple closed curve $r$ in (the image of) $F$ then bounds a compact surface $R$ in $M$. The reader should think of $R$ as an ``unoriented cap'' and check that it has a relative Euler number, just like a Whitney disc or an ordinary cap.   Rochlin then asserted that setting $q_F(r) :=  \lvert \Int{R}\pitchfork F\rvert$, for $R$ with vanishing relative Euler number, defines a quadratic enhancement
\[
q_F\colon  H_1(\Sigma;\Z/2) \to \Z/2,
\]
that refines the mod~2 intersection form on $\Sigma$. Independence from the choice of $R$ follows from $F$ being dual to $w_2(M)$, in this setting using intersections of $F$ with all classes of the form $[R\cup R']\in H_2(M;\Z/2)$. Rochlin stated that the Arf invariant $\Arf(q_F)$ is equal to $\theta(c)=\theta([F])$.
A nice consequence of this equality is that $\Arf(q_F)=\theta(c)$ vanishes whenever $c$ can be represented by an embedded sphere, because $q_F$ is then defined on the zero vector space.

\subsection{Combinatorial formulae: Freedman--Kirby's characteristic bordism}\label{sec:FK}

Using the same definitions, Freedman and Kirby proved Rochlin's claims from above in \cite{freedman-kirby}, on their way to a geometric proof of Rochlin's original theorem. They worked with an arbitrary smooth, closed, oriented 4-manifold $M$, but before computing $q_F$ they performed surgery on circles in $M$ to arrange that $H_1(M;\Z/2)=0$; alternatively, they could have made $M$ simply connected and used discs for $R$, i.e.\ ordinary caps. They showed that $\Arf(q_F)$ is invariant under ``characteristic bordism'', implying independence from the choice of surgeries, as well as establishing the equality $\Arf(q_F)=\theta([F])$ by checking it on the generators of $\Omega^{\operatorname{char}}_4$. A different proof of $\Arf(q_F)=\theta([F])$ was given in \cite{matsumoto86}.

On a historical note, Freedman--Kirby wrote that they learnt these results from Casson and that they only heard of Rochlin's results after finishing their paper.
The Rochlin method was extended to non-orientable characteristic surfaces in closed 4-manifolds by Guillou--Marin~\cite{Guillou-Marin} and Kirby--Taylor~\cite{Kirby-Taylor}.

\subsection{Combinatorial formulae: Matsumoto's \texorpdfstring{$t$-invariant}{t-invariant}}\label{sec:Matsumoto}

 In \cite{matsumoto78}, published in the same proceedings as~\cite{freedman-kirby}, Matsumoto started with a spherical class $c\in \pi_2(M)$ and represented it by a generic immersion $F\colon S^2 \imra M$ with $2g$ algebraically cancelling double points. He assumed that $H_1(M;\Z)=0$, using this condition to find ``Whitney surfaces'', i.e.\ oriented surfaces $R_1,\dots, R_g$ bounded by pairs of Whitney arcs in $F$. Again there is a relative Euler number and we may assume that every $R_i$ has vanishing relative Euler number. Matsumoto proved that if $[F]\in H_2(M;\Z)$ is characteristic then
\begin{equation}\label{eq:Matsumoto}
\Arf(q_F) = \sum_{i=1}^g  \lvert \Int{R_i}\pitchfork F\rvert =: t(F)\in\Z/2
\end{equation}
by adding $g$ tubes based at pairs of double points of $F$ to turn it into an embedding of a surface $\Sigma$ of genus $g$, where $q_F$ is the quadratic enhancement defined above. The new surface has pairs of framed caps $(D_i,R_i)$ where $D_i$ is a meridional disc of the $i$-th tube and hence has one interior intersection with $F$,  so $q_F(\partial D_i)=1$. Since the boundaries of these caps form a hyperbolic basis of $H_1(\Sigma;\Z/2)$, the result follows from the usual formula
\[
\Arf(q_F)=\sum_{i=1}^g q_F(\partial D_i)\cdot q_F(\partial R_i) = \sum_{i=1}^g q_F(\partial R_i)=\sum_{i=1}^g  \lvert \Int{R_i}\pitchfork F\rvert.
\]

\subsection{Summary of the secondary embedding obstructions from the 1970s}\label{sec:KM-summary}

Given $[F]\in\pi_2(M)$ such that its Hurewicz image in $H_2(M;\Z/2)$ is Poincar\'e dual to $w_2(M)$, the above results show that
\[
\theta([F]) = \Arf(q_F) = t(F)\in\Z/2
\]
is an obstruction to representing $[F]$ by an embedding $F\colon S^2\hra M$. Note that $\theta$ only depends on the homology class  $h([F])\in H_2(M;\Z)$ by definition, whereas that is not clear for the other two invariants. 

An attractive aspect of Matsumoto's $t(F)$ is that it can be computed combinatorially from a generic immersion $F\colon S^2 \imra M$. One argues directly that $t(F)$ is an obstruction to representing $[F]\in\pi_2(M)$ by an embedded sphere and independence of the choice of $R_i$ comes from $[F]$ being characteristic.

Matsumoto's formula was extended in a number of ways. For example, in recent work of three of the current authors and Land on the stable diffeomorphism classification of spin 4-manifolds~\cites{KLPT,KPT-tau,KPT-long}, a version of Matsumoto's $t$-invariant was used to compute the relevant Arf invariant. We describe further extensions presently.

It follows from topological transversality \cite{FQ}*{Section~9.5} that in a smooth, closed, oriented 4-manifold $M$, the quantity $\theta(c)$ is also an obstruction to representing an element $c$ as before by a topological i.e.\ locally flat embedding $F\colon S^2\hra M$. If $M$ is not smooth, one adds the Kirby--Siebenmann invariant and then the formula
\[
\theta_{\operatorname{TOP}}(c):= \theta(c) + \ks(M)
\]
defines such an obstruction; see \cite{CST-universal}*{Introduction} for details. For example, it follows that the generator of $\pi_2(*\CP^2)$ is not represented by an embedding. Historically speaking, these applications were not known at the time of publication of~\cites{Rochlin:71,freedman-kirby,matsumoto78}.

In the following, we will return to considering topological manifolds and obstructions to topological embeddings.

\subsection{Secondary obstructions to embedding genus zero surfaces with dual spheres} \label{sec:unions}

If $\Sigma$ is a union of discs or spheres and $F\colon (\Sigma,\partial \Sigma)\looparrowright (M, \partial M)$ has algebraically dual spheres, then Freedman--Quinn~\cite{FQ} gave a version of Matsumoto's $t$-invariant in \cite{FQ}*{Definition~10.8A}, calling it the Kervaire--Milnor invariant. Rather than restricting $H_1(M;\Z)$ as in the discussions above, they assumed  that $\mu (F)=0$, i.e.\ that all double points of $F$ can be paired by Whitney discs.
They used  the same formula as in \eqref{eq:Matsumoto}, but counted intersections with $F^\twist$, restricting to the Whitney discs in a convenient collection $\W^\twist$ that pair double points of $F^\twist$. They claimed that this mod~2 count, $\tw(F^\twist,\W^\twist)$ from \cref{def:t}, is a secondary obstruction to representing $F$ by an embedding. However, this is only true if $F^\twist$ is $r$-characteristic (\cref{def:r-char}), as Stong's correction \cite{Stong} showed. Stong noticed that the choice of sheets for double points whose group elements have order 2 is related to immersed $\RP$s in $M$. If $F$ is dual to $w_2(M)$ then $F$ is also $r$-characteristic, but not vice versa, so this obstruction is more generally defined than $\theta([F])$.

The embedding theorem for unions of discs and spheres~\cite{FQ}*{Theorem~10.5}, as corrected by Stong, says, in our notation, that for good fundamental group $\pi_1(M)$, such an $F$ is homotopic to a topological embedding if and only if there exists a convenient collection of Whitney discs $\W^\twist$ for the double points of $F^\twist$ such that $t(F^\twist,\W^\twist)=0$. We give more details about the Freedman--Quinn--Stong embedding result in \cref{sec:intro-km-section}.

 \subsection{Secondary obstructions to embedding unions of spheres} \label{sec:tau}

Matsumoto's invariant $t(F)$ from \eqref{eq:Matsumoto} was extended to a secondary embedding obstruction in~\cite{Schneiderman-Teichner} for $F=\{f_i\}_{i=1}^m$, not assuming dual spheres, where each $f_i\colon  S^2\imra M$ is a generic immersion and assuming $\mu (F)=0$ and that $M$ is oriented.
By  counting interior intersections of $F$ with a convenient collection $\W$ of Whitney discs pairing the double points of $F$, and remembering group elements, signs, and components of $F$, Schneiderman and the fourth author defined an intersection count $\tau(F,\W)\in \TT(\pi_1(M),m)$.
Here $\TT(\pi_1(M),m)$ is the abelian group given by the direct sum of $m+\binom{m}{2}+\binom{m}{3}$ copies of $\Z[\pi_1(M) \times \pi_1(M)]$. To obtain a secondary embedding obstruction,~\cite{Schneiderman-Teichner}*{Section~8} gave a list of relations such that the subgroup $\RR(M,F)\leq\TT(\pi_1(M),m)$ generated by these relations has the property that
\[
\tau(f_1,\dots,f_m)=\tau(F):=[\tau(F,\W)]\in \TT(\pi_1(M),m)/\RR(M,F)
\]
 does not depend on the choice of convenient collection $\W$. In our current language, the main result of that paper is that  $\tau(F)=0$ if and only if $\km(F)=0$ as in \cref{def:km=0}. In the absence of dual spheres, $\tau(F)=0$ does not imply that $F$ is homotopic to an embedding. For example, there are obstructions from higher order Whitney towers.

If $F$ is $r$-characteristic then the augmentation map  $\EE\colon \TT(\pi_1(M),m)\to \Z/2$, summing all possible coefficients, takes $\RR(M,F)$ to zero and $\tau(F)$ to Matsumoto's $t(f_1\#\cdots\#f_m)$, for an arbitrary choice of interior connected sum of the $\{f_i\}_{i=1}^m$.
Moreover, if $F$ has algebraic dual spheres then $\EE$ induces an isomorphism of $\TT(\pi_1(M),m)/\RR(M,F)$ with either $\Z/2$ or $0$, depending on whether $F^\twist$ is $r$-characteristic or not. This gives the relationship to \cref{sec:unions}.

\subsection{Secondary embedding obstructions for arbitrary compact surfaces}\label{sec:km-sec3}

It is likely possible to extend the invariant $\tau$ from \cref{sec:tau} to arbitrary immersed compact surfaces, not just spheres. However determining the analogue of $\RR(M,F)$ would be a formidable task. In this paper we take the first step, namely by defining the right notion of $b$-characteristic surfaces for which Matsumoto's invariant extends from spheres to a secondary embedding obstruction for arbitrary compact surfaces. We also generalise the work of Freedman-Quinn--Stong to all compact surfaces in the presence of algebraically dual spheres.

Recall from \cref{def:km=0} that for $F\colon (\Sigma,\partial \Sigma)\looparrowright (M, \partial M)$ as in \cref{convention}, by definition the \emph{Kervaire--Milnor invariant} $\km(F)\in\Z/2$ vanishes if and only if after finitely many finger moves on the interior of~$F$, taking $F$ to some $F'$, there is a convenient collection of
Whitney discs, with interiors disjoint from $F'$, pairing all the double points of $F'$.

The finger moves in this definition are relevant because finger moves can add relations to the  fundamental group $\pi_1(M\smallsetminus F)$, making it easier to find (Whitney) discs in the complement of $F$.

We could have allowed arbitrary regular homotopies, from $F$ to $F'$, in the definition of $\km$. However, this is not needed as the following result shows. Note that a non-regular homotopy can change $\km(F)$, see \cref{cor:RP-euler-no-restrictions}.

\begin{proposition}\label{lemma:km-reg-homotopy-invariance}
Let $\Sigma$ and $M$ be as in \cref{convention}. If $F_1, F_2\colon \Sigma\imra M$ are regularly homotopic generic immersions then $\km(F_1) = \km(F_2) \in \Z/2$.
\end{proposition}

\begin{proof}
To show that $\km(F_1) = \km(F_2)$, by symmetry it suffices to show that $\km(F_1)=0$ implies $\km(F_2)=0$.
Suppose that $\km(F_1)=0$, and let $F_1'$ be obtained from $F_1$ by finger moves such that the intersections of $F_1'$ can be paired up by Whitney discs $\{W_i\}$ as in \cref{def:km=0}.
    Since $F_1'$ and $F_2$ are regularly homotopic, there is a generic immersion $F_3$ such that $F_3$ can be obtained from both $F_1'$ and $F_2$ by finger moves and ambient isotopies.   Since $F_3$ is obtained from $F_1'$ by finger moves and ambient isotopies and the finger moves can be assumed to be disjoint from $\{W_i\}$, all the double points of $F_3$ can also be paired up by Whitney discs with interiors disjoint from $F_3$, as in \cref{def:km=0}.  Since $F_3$ is obtained from $F_2$ by finger moves, taking $F_2' := F_3$ it follows that $\km(F_2)=0$.
\end{proof}

\cref{def:km=0} is optimised for the proof of \cref{thm:SET}, as we will see shortly, but is difficult to use in practice. In particular, while one may fortuitously detect specific finger moves and Whitney discs to show $\km(F)=0$, without a combinatorial description it appears, for a given $F$, to be hard to prove that the required finger moves from $F$ to some $F'$, together with Whitney discs for $F'$, do not exist.
We provide precisely such a combinatorial reformulation in \cref{thm:embedding-obstruction}, generalizing Matsumoto's invariant to our formula for $\tw(F)$ for $b$-characteristic $F$. In the proof of \cref{thm:main} we will show that in the presence of dual spheres this agrees with \cref{def:km=0}.

\section{The proof of the surface embedding theorem}\label{sec:thm1.1}

The surface embedding theorem (\cref{thm:SET}) can be deduced using the proof of \cite{FQ}*{Theorem~10.5\,(1)}, combined with an observation in~\cite{PRT20}*{Theorem~A, Lemma~6.5} for the condition on the homotopy class of $\ol{G}$, using our definition of the Kervaire--Milnor invariant (\cref{def:km=0}).  Since the surface embedding theorem does not follow directly from the statement of \cite{FQ}*{Theorem~10.5\,(1)}, as previously discussed, and also since the latter requires a correction by Stong~\cite{Stong}, it can be hard for the uninitiated to piece together a correct proof. Therefore we provide one in this section.

Further, the statement of \cite{FQ}*{Theorem~10.5\,(1)} is itself quite complicated, and our version focused on the surface embedding problem may be useful for those looking to apply the technology of Freedman--Quinn without delving into the details.

\subsection{Ingredients}

The statement of the surface embedding theorem uses the notions of algebraically and geometrically dual spheres. We recall the definitions.

\begin{definition}\label{defn:alg-dual-spheres}
Let $F\colon (\Sigma,\partial \Sigma)\looparrowright (M, \partial M)$ be as in \cref{convention}, with components $\{f_i\}_{i=1}^m$.
\begin{enumerate}
    \item  A collection $G=\{g_i\colon S^2\looparrowright M\}_{i=1}^m$ of generic immersions is said to be \emph{algebraically dual} to $F$ if $F \sqcup G$ is a generic immersion and $\lambda(f_i,g_j)=[\delta_{ij}] \in \Gamma_{f_i,g_j}$ for all $i,j$, for some choice of basings for $F$ and $G$.
	\item 	A  collection $\ol{G}=\{\ol{g}_i\colon S^2 \imra M\}_{i=1}^m$ of generic immersions is \emph{geometrically dual} to $F$ if $F \sqcup G$ is a generic immersion and the geometric count of intersections satisfies $|f_i \pitchfork \ol{g}_j|=\delta_{ij}$ for all $i,j$.
  \end{enumerate}
\end{definition}

We will need the following lemma, the idea behind which is due to Casson~\citelist{\cite{Casson} \cite{F}*{Section~3}}.  The formulation we give here is from \cite{PRT20}*{Lemma~5.1}.

\begin{lemma}[Geometric Casson lemma]\label{lem:geometric-casson-lemma}
Let $F$ and $G$ be transversely intersecting generic immersions of compact surfaces in a connected 4-manifold $M$.
Assume that the intersection points $\{p,q\}\subseteq F\pitchfork G$ are paired by a Whitney disc $W$. Then there is a regular homotopy from $F \cup G$ to $\ol F \cup \ol{G}$ such that  $\ol F \pitchfork \ol{G} = (F\pitchfork G) \sm \{p,q\}$. That is, the two paired intersections have been removed.
The regular homotopy may create many new self-intersections of $F$ and $G$; however, these are algebraically cancelling.
\end{lemma}

The proof of the surface embedding theorem also relies on Freedman's disc embedding theorem, whose statement we recall.

\begin{theorem}[Disc embedding theorem \cites{F,FQ,PRT20}; see also \cite{Freedman-notes}]\label{thm:DET}
Let $M$ be a connected topological $4$-manifold with good fundamental group, and let
\[
F= \{f_i\}_{i=1}^m \colon (D^2\sqcup\cdots \sqcup D^2, S^1\sqcup \cdots \sqcup S^1)  \looparrowright (M,\partial M)
\]
be a generic immersion of finitely many discs.
Assume that $F$ has framed algebraically dual spheres $G=\{[g_i]\}_{i=1}^m \subseteq \pi_2(M)$ such that $\lambda(g_i,g_j)=0=\mu(g_i)$ for all $i\neq j$. Then there is a flat embedding $\ol{F} = \{\ol{f}_i\}_{i=1}^m \colon (\sqcup D^2, \sqcup S^1)  \hookrightarrow (M,\partial M)$, which is equipped with geometrically dual spheres $\ol{G}=\{\ol{g}_i\}_{i=1}^m$, such that $\ol{F}$ and $F$ have the same framed boundary and $[\ol{g}_i]= [g_i]\in \pi_2(M)$ for all $i$.
\end{theorem}

We will also freely use standard constructions such as symmetric contraction, boundary twisting, and interior twisting (i.e.\ adding local cusps). See~\citelist{\cite{FQ}*{Chapters 1-2}\cite{Freedman-book-basicgeo}} for further details.

\subsection{Proof of the surface embedding theorem}

We recall the statement for the convenience of the reader.

\begin{reptheorem}{thm:SET}[Surface embedding theorem]
Let $F=\{f_i\}_{i=1}^m \colon (\Sigma,\partial \Sigma)\looparrowright (M, \partial M)$ be as in \cref{convention}. Suppose that $\pi_1(M)$ is good and that $F$ has algebraically dual spheres $G=\{[g_i]\}_{i=1}^m \subseteq \pi_2(M)$. Then the following statements are equivalent.
\begin{enumerate}[(i)]
\item \label{item:i-SET}  The  self-intersection number $\mu (F)$ and the Kervaire--Milnor invariant $\km(F)\in\Z/2$ vanish.
\item \label{item:ii-SET} There is an embedding $\ol{F}=\{\ol{f}_i\}_{i=1}^m \colon (\Sigma, \partial\Sigma)  \hookrightarrow (M,\partial M)$, regularly homotopic to $F$ relative to $\partial \Sigma$, with geometrically dual spheres $\ol{G}=\{\ol{g}_i\colon S^2\looparrowright M\}_{i=1}^m$ such that $[\ol{g}_i]= [g_i]\in \pi_2(M)$ for all $i$.
\end{enumerate}
\end{reptheorem}

\begin{proof}[Proof of \cref{thm:SET}]
The direction \eqref{item:ii-SET} $\Rightarrow$ \eqref{item:i-SET} follows from the fact that the intersection and self-intersection numbers, as well as the Kervaire--Milnor invariant, are invariant under regular homotopy (relative to the boundary) by \cref{prop:regular-homotopy-inv-lambda,prop:regular-homotopy-inv-mu,lemma:km-reg-homotopy-invariance} respectively.

The proof of the direction \eqref{item:i-SET} $\Rightarrow$ \eqref{item:ii-SET} reduces to the disc embedding theorem (\cref{thm:DET}) as follows. The argument is similar to the proof of~\cite{FQ}*{Corollary~5.1B} (see also the proof of~\cite{PRT20}*{Theorem~8.1}).

Apply the geometric Casson \cref{lem:geometric-casson-lemma} to upgrade $G=\{g_i\}$ from algebraically to geometrically dual spheres $G'=\{g_i'\}$, changing $F$ to $F'$ by a regular homotopy in the process. The intersection and self-intersection numbers, and the Kervaire--Milnor invariant, vanish for $F$.  So they also vanish for $F'$, since all three quantities are preserved under regular homotopy relative to the boundary by \cref{prop:regular-homotopy-inv-lambda,prop:regular-homotopy-inv-mu,lemma:km-reg-homotopy-invariance}.

Then by the definition of the Kervaire--Milnor invariant (\cref{def:km=0}), after further finger moves changing $F'$ to some $F''$, we can find a convenient collection of Whitney discs $\W = \{W_\ell\}$ for $F''$ whose interiors are disjoint from $F''$. Moreover $F''$ and $G'$ are still geometrically dual, since the finger moves may be assumed to miss $G'$.

We shall apply the disc embedding theorem (\cref{thm:DET}) to the collection of generically immersed  discs $\W$ in the 4-manifold $M\setminus \nu F''$, so we verify that the hypotheses are satisfied.
The Whitney discs $\W$ have framed algebraically dual spheres as follows. The Clifford tori at the double points of $F''$ are geometrically dual to $\W$. Symmetrically contract half of these tori, one per Whitney disc, using meridional discs for $F''$ tubed into the geometrically dual spheres $G'=\{g_i'\}$. The resulting spheres are only algebraically dual to $\W$ since the components of $\W$ and $G'$ may intersect arbitrarily; however, they have vanishing intersection and self-intersection numbers since they were produced by symmetric contraction.  They are also framed, as we argue briefly now.  If a sphere $g_i$ in $G'$ is not framed,  then the symmetric contraction uses incorrectly framed caps. However in the symmetric contraction process each cap is used twice, with opposite orientations, and so any framing discrepancies cancel out. Since $F''$ has geometrically dual spheres, the fundamental group $\pi_1(M\setminus \nu F'')\cong\pi_1(M)$ and is thus good. This verifies  the hypotheses of the disc embedding theorem (\cref{thm:DET}) for $\W$, as desired.

Apply the disc embedding theorem to the Whitney discs $\W$ in $M\setminus \nu F''$ to obtain disjointly embedded, flat, framed Whitney discs $\{\ol{W}_\ell\}$ for the double points of $F''$, with interiors still disjoint from $F''$, along with a collection of geometrically dual spheres for the $\{\ol{W}_\ell\}$ in $M\setminus \nu F''$.

Tube any intersections of $G'$ with $\{\ol{W}_\ell\}$ into the geometrically dual spheres for $\{\ol{W}_\ell\}$, giving a new collection of spheres $\ol{G}=\{\ol{g_i}\}$ disjoint from $\{\ol{W}_\ell\}$. Now we have that the interiors of $\{\ol{W}_\ell\}$ lie in the complement of $F''\cup \ol{G}$, and moreover $F''$ and $\ol{G}$ are geometrically dual. Perform Whitney moves on $F''$ along $\{\ol{W}_\ell\}$ to arrive at an embedding $\ol{F}$ as claimed. By construction, $\ol{F}$ and $\ol{G}$ are geometrically dual. That $[\ol{g}_i]= [g_i]\in \pi_2(M)$ for each $i$ follows from \cite{PRT20}*{Lemma~6.5}.
\end{proof}

\subsection{The \texorpdfstring{$\pi_1$}{pi1}-negligible surface embedding theorem}

Recall that a map $F \colon X\to Y$ is called \emph{$\pi_1$-negligible} if the inclusion $Y\setminus F(X) \subseteq Y$ induces an isomorphism on $\pi_1$ for all basepoints. Here is a reformulation of the surface embedding theorem.

\begin{corollary}[The $\pi_1$-negligible surface embedding theorem]\label{thm:surface-embedding-negligible}
Let $F\colon (\Sigma,\partial \Sigma)\looparrowright (M, \partial M)$ be as in \cref{convention}. Suppose that $F$ is $\pi_1$-negligible and that $\pi_1(M)$ is good. Then $\mu (F)=0$ and the  Kervaire--Milnor invariant of $F$ vanish if and only if there exists a $\pi_1$-negligible embedding $\overline{F} \colon (\Sigma,\partial\Sigma) \hookrightarrow (M, \partial M)$ regularly homotopic to $F$, relative to the boundary.
\end{corollary}

This corollary follows from the surface embedding theorem and the fact  that a generic immersion $F \colon \Sigma \imra M$ is $\pi_1$-negligible if and only if $F$ admits geometrically dual spheres, which can be seen as follows. For the forwards direction, the meridional circles are null-homotopic in $M$, so by $\pi_1$-negligibility they are null-homotopic in $M \sm F(\Sigma)$. The union of null-homotopies with meridional discs gives geometrically dual spheres. For the reverse direction, first note that by general position the homomorphism $\pi_1(M\setminus F(\Sigma)) \ra \pi_1(M)$
is surjective. The kernel is normally generated by a collection consisting of one meridional circle for each connected component of $\Sigma$. Since geometrically dual spheres provide null homotopies for these meridians, the assertion follows. By the geometric Casson lemma (\cref{lem:geometric-casson-lemma}), the map $F$ in the statement of the surface embedding theorem (\cref{thm:SET}) is regularly homotopic to a $\pi_1$-negligible map, due to the existence of the algebraically dual spheres $G$. Indeed, this is the first step of the proof of the surface embedding theorem. Note that \cref{thm:SET} also controls the homotopy class of the dual spheres, and so is slightly stronger than \cref{thm:surface-embedding-negligible}.

\section{Band characteristic maps and the combinatorial formula}\label{sec:intro-km-section}

In this section we define $b$-characteristic surfaces (\cref{def:b-char}) and motivate the combinatorial formula for the Kervaire--Milnor invariant (\cref{def:t}). We postpone many of the proofs to \cref{sec:km}. We hope this will help the reader to assimilate the overall structure more easily. We work towards the definition of $b$-characteristic surfaces by first defining the related notions of $s$-characteristic and $r$-characteristic surfaces, mirroring the historical development. These latter definitions are simpler to state and serve to motivate the more complicated definition of $b$-characteristic surfaces.

A sphere $g\colon S^2\looparrowright M$ in a topological $4$-manifold $M$ is said to be \emph{twisted} if the Euler number of the normal bundle is odd. We say $g$ is \emph{framed} if the normal bundle is trivial.  To coincide with the usual meaning of framed, one can also implicitly choose a trivialisation, although we will not make use of such a choice.
Observe that if the normal bundle of $g$ has even Euler number then it is homotopic to a generically immersed sphere with trivial normal bundle, via adding local cusps.

\begin{definition}\label{def:Ft}
Let $F=\{f_1,\dots,f_m\}\colon (\Sigma,\partial \Sigma)\looparrowright (M, \partial M)$ be as in \cref{convention}. We define $\Sigma^\twist\subseteq\Sigma$ so that for $i\in\{1,\dots,m\}$, the component $\Sigma_i \subseteq \Sigma^{\twist}$ if and only if there is no framed immersed sphere $g_i$ with $\lambda(f_j,g_i) = \delta_{ij}$ for all $j=1,\dots,m$. Then we use
	\[F^\twist=\{f^\twist_i\colon \Sigma_i^\twist\ra M\}\]
	to denote the restriction of $F$ to $\Sigma^\twist$. Note that if an $f_i$ does not admit an algebraically dual sphere at all, then it belongs to $F^\twist$.
\end{definition}

Recall that $x \in H_2(M,\partial M;\Z/2)$ is said to be \emph{characteristic} if $x \cdot a = a \cdot a \in\Z/2$ for every $a \in H_2(M;\Z/2)$, where $- \cdot -$ denotes the intersection pairing $H_2(M;\Z/2)\times H_2(M,\partial M;\Z/2)\to \Z/2$.  The next definition gives a weaker notion.

\begin{definition}\label{def:s-char}
Let $F\colon (\Sigma,\partial \Sigma)\looparrowright (M, \partial M)$ be as in \cref{convention}. 	The map $F$ is called \emph{spherically characteristic} (or \emph{$s$-characteristic} for short) if
$F \cdot a = a \cdot a \in\Z/2$ for all $a\in\pi_2(M)$, considered as an element of $H_2(M;\Z/2)$.
\end{definition}

We will show in \cref{lem:r-b}  that $b$-characteristic maps are $s$-characteristic.
\begin{lemma}\label{lem:F=F-twist}
Let $F\colon (\Sigma,\partial \Sigma)\looparrowright (M, \partial M)$ be as in \cref{convention}.
\begin{enumerate}[(i)]
\item If $F$ is $s$-characteristic, then $F=F^\twist$.
\item If $F$ has algebraically dual spheres, then $F^\twist$ is $s$-characteristic or empty.
\end{enumerate}
\end{lemma}

\begin{proof}
To prove (i), suppose $F$ is not equal to $F^\twist$ then there exists a component $\Sigma_i$ of $\Sigma$ with a framed dual sphere $g_i$, i.e.\ with $\lambda(f_j, g_i)=\delta_{ij}$ for all $j\neq i$. This leads to the contradiction
\[
1 =  f_i \cdot g_i =  f_i\cdot g_i + \sum_{j\neq i} f_j \cdot g_i =  F \cdot g_i =  g_i\cdot g_i =  0 \in\Z/2,
\]
where the second to last equality follows from $F$ being $s$-characteristic.

To prove (ii), suppose that $F$ has algebraically dual spheres. Note that the dual spheres for $F^\twist\subseteq F$ are necessarily twisted. Assume that $F^\twist$ neither  $s$-characteristic nor empty. Then there exists $a\in\pi_2(M)$ such that $F^\twist\cdot a\not= a\cdot a\in\Z/2$. By tubing into a dual sphere to a component of $F^\twist$ if necessary, we can assume that $a$ is untwisted, i.e.~$a\cdot a$ is zero and that $F^\twist\cdot a= 1$. Choose some component $f_j$ of $F^\twist$ such that $a\pitchfork f_j$ is nonempty. Except for one of the intersections between $f_j$ and $a$, tube all the intersections of $a$ and $F^\twist$ into the corresponding dual spheres to $F^\twist$. Call the resulting sphere $a'_j$. Since $F^\twist \cdot a=1$, we tubed into an even number of dual spheres, so $a'_j\cdot a'_j= a\cdot a= 0\in\Z/2$. Via adding local cusps, we may assume that $a'_j$ is framed. We also have that $\lambda(f_i,a'_j) = \delta_{ij}$ for all $i,j$. This contradicts the definition of $F^\twist$.
\end{proof}

Recall that a \emph{convenient} collection of Whitney discs $\W$ for the intersections within $F$ consists of framed, generically immersed Whitney discs with interiors transverse to $F$, and with disjointly embedded boundaries. Recall the invariant $t$ from \cref{def:t} appearing in \cref{thm:main}, where $\tw(F,\W)$ is the mod $2$ count of transverse intersections between $F$ and the interiors of the Whitney discs in $\W$.

If $F$ is not $s$-characteristic, then we can change $\tw(F,\W)$ as follows. Given $a\in \pi_2(M)$ with $F \cdot a$ odd but $a \cdot a$ even, one can tube a framed Whitney disc $W$ for $F$ into a framed representative $\tilde a\colon S^2\imra M$, keeping the new Whitney disc framed but adding an odd number of interior intersections with $F$. If $F\cdot a$ is even and $a\cdot a$ is odd one can tube $W$ into a representative $\tilde a$ and also add an odd number of boundary twists to keep the new Whitney disc framed but again adding an odd number of interior intersections with $F$. Using \cref{lem:F=F-twist}, this is one reason for the appearance of $F^\twist$ in the following statements.

 The following lemma is also used in the proof of \cref{thm:main}, and shows that the vanishing of $t$ for a given collection of Whitney discs implies the vanishing of $\km$.

\begin{lemma}\label{lem:tau}
Let $F\colon (\Sigma,\partial \Sigma)\looparrowright (M, \partial M)$ be as in \cref{convention}.
Suppose that $F$ admits algebraically dual spheres, and that all double points of $F$ are paired by a convenient collection $\W$ of Whitney discs. Let $\W^\twist \subseteq \W$ denote the sub-collection of Whitney discs for the intersections within $F^\twist$, where $F^\twist$ is as in \cref{def:Ft}.
If $\tw(F^\twist,\W^\twist)=0$, then $\km(F)=0$.
\end{lemma}

We wish to find practically verifiable conditions on $F$ that guarantee that $\tw(F^\twist,\W^\twist)$ is independent of the collection of Whitney discs $\W^{\twist}$. More precisely, the value of $\tw(F^\twist,\W^\twist)$ should be independent of the pairing of double points, the Whitney arcs joining the paired double points (which includes the choice of sheets at each double point) and finally the Whitney discs.
In the case that each $\Sigma_i$ is simply connected, $\tw(F^\twist,\W^\twist)$ agrees with \cite{FQ}*{Definition~10.8A}. However, Freedman-Quinn claim in their Lemma 10.8B that, for simply connected $\Sigma$, the quantity $\tw(F^{\twist},\W^{\twist})$ only depends on $F^\twist$, and not on the Whitney discs, as long as $F^\twist$ is $s$-characteristic. This is not true in general, as pointed out and corrected by Stong~\cite{Stong}.
Further, again with $\pi_1(\Sigma_i) = 1$ for all $i$, Stong established that the value of $t$ does not depend on the choice of $\W^\twist$ using the notion of \emph{$r$-characteristic} discs and spheres.  Here is our generalisation of his notion.

\begin{definition}\label{def:r-char}
Let $\Sigma$ and $M$ be as in \cref{convention}. A map $F\colon (\Sigma,\partial\Sigma)\to (M,\partial M)$ is called \emph{$\RP$-characteristic} (or simply \emph{$r$-characteristic}) if $F\cdot R= R\cdot R \in\Z/2$
	for every map $R\colon \mathbb{RP}^2\to M$ satisfying $R^*w_1(M)=0$.
\end{definition}

\begin{remark}\label{rem:r-char-implies-s-char}
A map $c \colon \RP \to S^2$ of odd degree (e.g.\ a collapse map) composed with elements of $\pi_2(M)$ can be used to show that $r$-characteristic maps are $s$-characteristic. Indeed, given $F$ and $a \in \pi_2(M)$ we obtain $a \circ c \colon \RP \to M$, and we have $0 =  F \cdot (a \circ c) + (a\circ c) \cdot (a \circ c) = F \cdot a + a \cdot a \in\Z/2$, where the second equality uses that $c$ has odd degree.
\end{remark}

Stong's key observation~\cite{Stong} was that in some instances involving elements of order two in~$\pi_1(M)$, one can change the choice of sheets at two double points of $F^\twist$, and hence the Whitney arcs and corresponding Whitney disc, with a resulting change in the value of $t(F^\twist,\W^{\twist})$ by one. The restriction to $r$-characteristic maps removes this source of indeterminacy. To summarise, Stong showed the following theorem.

\begin{theorem}[Stong~{\cite{Stong}}]\label{theorem:Stong}
Let $F \colon (\Sigma,\partial \Sigma) \imra (M,\partial M)$ be as in \cref{convention}, with $\Sigma$ a union of discs or spheres. Suppose $\mu (F)=0$ and that $F$ admits algebraically dual spheres. If $F^\twist$ is not $r$-characteristic then $\km(F)=0$, and if $F^\twist$ is $r$-characteristic then $\km(F)=\tw(F^\twist,\W^\twist)$ for any choice of Whitney discs $\W^\twist$ for the intersections within $F^\twist$.
\end{theorem}

\begin{remark}\label{rem:freedman-quinn-stong-win}
Combining \cref{theorem:generic-immersions-bijection} and \cref{theorem:Stong} with \cref{thm:SET} gives the complete answer to the embedding problem for spheres and discs with algebraically dual spheres for good fundamental groups, due to Freedman, Quinn, and Stong.
\cref{theorem:generic-immersions-bijection} allows one to fix the regular homotopy class of generic immersions within the homotopy class to be that with $\mu(-)_1=0$, \cref{theorem:Stong} computes the Kervaire--Milnor invariant, and then one applies \cref{thm:SET} to conclude whether or not there is a  regular homotopy to an embedding.
\end{remark}

Our contribution in the present paper extends this solution to the case that the components of $\Sigma$ are not all simply connected. In this case there is a further source of indeterminacy coming from the choice of Whitney arcs on $\Sigma$. For this reason we need a stronger restriction on $F^\twist$.

\begin{definition}\label{def:band}
	A \emph{band} refers to either of the two $D^1$-bundles over $S^1$, i.e.\ a band is either an annulus or a M\"obius band.
Let $F\colon (\Sigma,\partial \Sigma)\looparrowright (M, \partial M)$ be as in \cref{convention}. Let $\mathcal{M}_F$ be the mapping cylinder of $F$.  Write $\bands (F)\subseteq H_2(M,\Sigma;\Z/2) := H_2(\mathcal{M}_F,\Sigma;\Z/2)$ for the subset of elements of the relative homology group that can be represented by a square
\[
\begin{tikzcd}
\partial B \ar[r, "h"]\ar[d, "\iota", hook] & \Sigma \ar[d, "F"]  \\
B \ar[r, "g"] & M
\end{tikzcd}
\]
where $B$ is a band and $\iota\colon \partial B\hookrightarrow B$ is the inclusion, such that
\begin{equation}\label{item:defn-band-w1-condn}
\langle w_1(M), g(C)\rangle + \langle w_1(\Sigma), h(\partial B)\rangle =0 \in \Z/2,
\end{equation}
 where $C$ is the core curve of $B$.
\end{definition}

\begin{figure}[htb]
	\centering
\begin{tikzpicture}
        \node[anchor=south west,inner sep=0] at (0,0){	    \includegraphics[width=3.5cm]{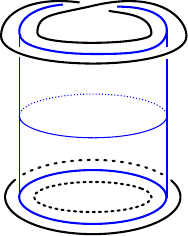}};
	  \node at (-0.23,0.15) {$F(\Sigma)$};
	\node at (0.1, 2.25) {$C$};
		\node at (3.35, 2.75) {$B$};
	\node at (-0.23,4.3) {$F(\Sigma)$};
	\end{tikzpicture}
	\caption{An annular band $B$ (blue) is shown with boundary on $F$ (black). One of the boundary components of $B$ is nonorientable on $F$ and one is orientable, so $\langle w_1(\Sigma), h(\partial B)\rangle =1$. Therefore, in order for this to be an element of $\bands(F)$, we must have $\langle w_1(M), g(C)\rangle=1$, where $M$ is the ambient $4$-manifold and $C$ is the core curve of the annulus, shown in blue.}
	\label{fig:band-example}
\end{figure}
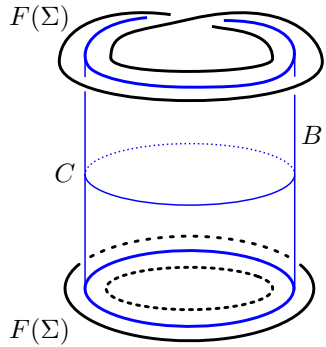

See \cref{fig:band-example} for an example of a band.
Note that every element of $\bands(F)$ can be represented by a generic immersion of pairs $(B,\partial B) \imra (M,\Sigma)$ (\cref{defn:generic-immersion-of-pairs}).
Writing $H_2(M,\Sigma;\Z/2)$ in place of $H_2(\mathcal{M}_F,\Sigma;\Z/2)$ is a slight but standard abuse of notation. The pair $(g,h)$ induces a relative homology class since the map
\[g \sqcup (h \times \Id_{[0,1]}) \colon B \sqcup (\partial B \times [0,1]) \ra M \sqcup (\Sigma \times [0,1])\]
descends to a map $\mathcal{M}_\iota \to \mathcal{M}_F$. The mapping cylinder $\mathcal{M}_\iota$ is homeomorphic to $B$, and so we obtain a map $(B,\partial B) \to (\mathcal{M}_F,\Sigma)$.
The image of the relative fundamental class $[B,\partial B]$ in  $H_2(\mathcal{M}_F,\Sigma;\Z/2)$ is an element of $\bands (F)$. From now on, since $\mathcal{M}_F \simeq M$, to simplify the notation we will not mention the mapping cylinder and refer to $\bands (F)\subseteq H_2(M,\Sigma;\Z/2)$.

We will see in \cref{lem:representable} that given a generic immersion $F\colon (\Sigma,\partial \Sigma)\looparrowright (M, \partial M)$ as in \cref{convention}, every element of $H_2(M,\Sigma;\Z/2)$ can be represented by an immersion of some compact surface $S$ into $M$, with interior transverse to $F$, and with boundary generically immersed in $F(\Sigma)$ away from the double points. The subset $\bands(F)$ consists of those homology classes for which $S$ can be chosen to be a band, satisfying condition~\eqref{item:defn-band-w1-condn}.

We use the notation
\[
\partial \colon H_2(M,\Sigma;\Z/2)\ra H_1(\Sigma;\Z/2)
\]
for the connecting homomorphism from the long exact sequence of the pair.  A class represented by a compact surface $S$ is mapped to its boundary $\partial S$ under the map $\partial$.
Then $\partial\bands(F)\subseteq H_1(\Sigma;\Z/2)$ consists of (the homology classes of) those closed $1$-manifolds immersed in $\Sigma$ whose images under $F$ bound bands in~$M$ satisfying~\eqref{item:defn-band-w1-condn}.

\begin{figure}[htb]
	\centering
\begin{tikzpicture}
\node[anchor=south west,inner sep=0] at (0,0){	    \includegraphics{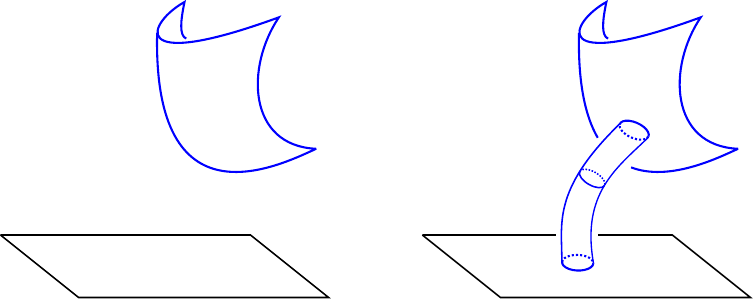}};
	  \node at (0.25,0.25) {$F(\Sigma)$};
	\node at (2.5, 3) {$S$};
		\node at (9.7, 3) {$B$};
	\node at (3,-0.35) {(a)};
		\node at (10.35,-0.35) {(b)};
	\end{tikzpicture}
	\caption{(a) An immersed surface $S$ (blue) in the ambient manifold $M$, and the image (black) of a generic immersion $F\colon (\Sigma,\partial \Sigma)\looparrowright (M, \partial M)$. (b) A thin tube is added, with one boundary component on $\Sigma$ and one on $S$. The surface $B$ (blue) is obtained by cutting out a disc on $S$ and gluing in the tube. Note that compared to $S$, the surface $B$ has a new boundary component lying on $\Sigma$.}
	\label{fig:tubing-band}
\end{figure}
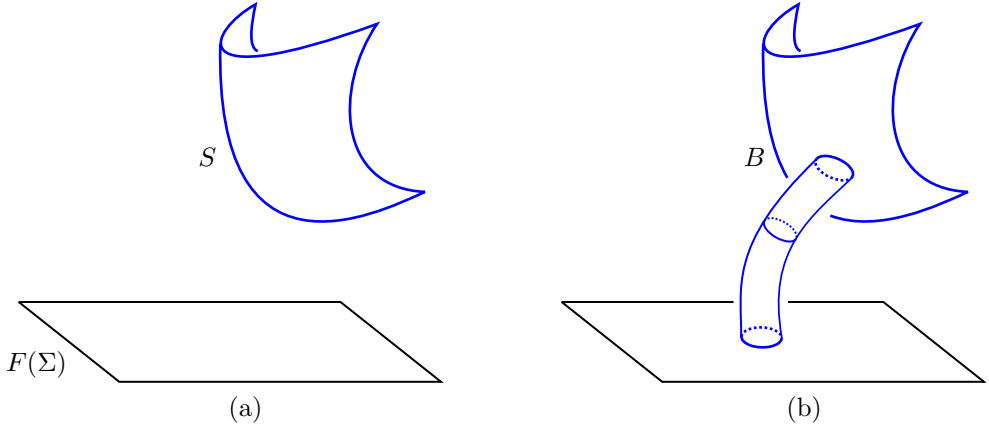

\begin{construction}\label{ex:tubing-band}
Given a generic immersion $F\colon (\Sigma,\partial \Sigma)\looparrowright (M, \partial M)$ as in \cref{convention} suppose we have a generically immersed surface $S$ in $M$ with boundary on $\Sigma$, i.e.\ admitting maps satisfying
\[
\begin{tikzcd}
\partial S \ar[r, "h"]\ar[d, "\iota", hook] & \Sigma \ar[d, "F"]  \\
S \ar[r, "g"] & M,
\end{tikzcd}
\]
where possibly $\partial S$ is empty.
Then the tubing procedure shown in \cref{fig:tubing-band} can be used to create a band, as follows. If $S$ is a disc, the procedure gives an annulus $B$ with boundary lying on $\Sigma$. This annulus satisfies \eqref{item:defn-band-w1-condn}, and therefore lies in $\bands(F)$, if and only if $\langle w_1(\Sigma),h(\partial S)\rangle=0\in\Z/2$, since the core of $B$ is null-homotopic in $M$ and and the newly created boundary component of $B$ is null-homotopic in $\Sigma$.

In the case that $S$ is a sphere, we can perform the tubing procedure of \cref{fig:tubing-band} to $S$ twice. In this case both boundary components of the annulus created are null-homotopic on $\Sigma$, so we always produce an element of $\bands(F)$.

Finally if $S$ is an $\RP$, the tubing procedure creates a M\"obius band with boundary on $\Sigma$, which lies in $\bands(F)$ if and only if $\langle w_1(M),g(\mathbb{RP}^1)\rangle=0\in \Z/2$, where $\mathbb{RP}^1\subseteq \RP$.
\end{construction}

When defining $b$-characteristic surfaces, we will restrict to the case that the $\Z/2$-valued intersection form $\lambda_{\Sigma}$ is trivial on $\partial \bands(F)$. We can restrict in this way because of the following lemma.

\begin{lemma}
	\label{lem:dependence_on_A}
Let $F\colon (\Sigma,\partial \Sigma)\looparrowright (M, \partial M)$ be as in \cref{convention}, with $\mu (F)=0$. If the $\Z/2$-valued intersection form $\lambda_{\Sigma}$ on $H_1(\Sigma;\Z/2)$ is nontrivial on $\partial\bands(F)$, then we can change $F$ by a regular homotopy to $F'$ such that there are convenient collections of Whitney discs $\W$ and $\W'$ for the double points of $F$ and $F'$ respectively, such that $t(F,\W)\neq t(F',\W')$.

Moreover, if $F$ has dual spheres and the $\Z/2$-valued intersection form $\lambda_{\Sigma^\twist}$ on $H_1(\Sigma^\twist;\Z/2)$ is nontrivial on $\partial\bands(F^\twist)$ then $\km(F)=0$.
\end{lemma}

In the case that $\lambda_{\Sigma}|_{\partial \bands(F)}$ is trivial, we define an invariant $\Theta$ on the set $\bands(F)$. We will need the following notions.

\begin{enumerate}
  \item For $Z$ a closed $1$-manifold generically immersed in $\Sigma$, the self-intersection number  $\mu_{\Sigma}(Z)\in\Z/2$ of $Z$ counts the number of double points, which we assume without loss of generality to be disjoint from the double points of $F$. As usual, this is not invariant under homotopies of $Z$ in $\Sigma$, only under regular homotopies.
  \item Let $S$ be a compact surface, with a generic immersion of pairs $(S,\partial S) \imra (M,\Sigma)$.
Suppose that $w_1(\Sigma)$ is trivial on each component of~$\partial S$, e.g.\ if~$\Sigma$ is orientable. Then the normal bundle of $\partial S$ in $\Sigma$ is trivial and we can pick a nowhere vanishing section (if $S$ is closed, this is an empty choice). Extend this section to the normal bundle of $S$ in $M$ (such a normal bundle exists by \cref{rem:not-so-generic}) such that the extension is transverse to the zero section. Then we define the \emph{Euler number} $e(S)$ to be the number of zeros of this section modulo 2. Observe that this is analogous to how one measures the twisting of a Whitney disc with respect to the Whitney framing. For $S$ a closed surface this coincides with the usual definition of the Euler number.
\end{enumerate}

Here is the definition of $\Theta(S)$ in the case that $w_1(\Sigma)$ is trivial on every component of $\partial S$.

\begin{definition}\label{def:thetaA}
Let $F\colon (\Sigma,\partial \Sigma)\looparrowright (M, \partial M)$ be as in \cref{convention}, with $\mu (F)=0$. Let $A$ be a choice of Whitney arcs pairing the double points of $F$. Let $S$ be a compact surface in $M$ with a generic immersion of pairs $(S,\partial S) \imra (M,\Sigma)$, such that $\partial S$ is transverse to $A$ and $w_1(\Sigma)$ is trivial on every component of $\partial S$. Define
	\begin{equation}\label{eq:Theta-definition}
	\Theta_A(S):= \mu_{\Sigma}(\partial S)+ \lvert \partial S\pitchfork A \rvert +\lvert \Int S \pitchfork F \rvert + e(S) \mod 2.
	\end{equation}
\end{definition}

For closed $S$ we have $\Theta_A(S)\equiv \lvert \Int S\pitchfork F \rvert + e(S)\mod{2}$, and thus $\Theta_A(S)$ vanishes for all closed $S$ if and only if $F$ is characteristic in the traditional sense. In the proof of \cref{thm:main}, we will only use the definition of $\Theta_A$ for bands. But the case of general surfaces will be useful for our proof that, in the cases relevant to us, $\Theta_A$ does not depend on the choice of $A$ (see \cref{lem:well-defined} below).

\begin{remark}
\cref{def:thetaA} suffices in the case of orientable $\Sigma$. The reader only interested in this case may safely skip ahead to \cref{lem:well-defined}.
\end{remark}

If a component of $\partial S$ is orientation-reversing in $\Sigma$, then its normal bundle in $\Sigma$ is nontrivial and hence we may not use it to choose a nowhere vanishing section of the normal bundle of $S$ on its boundary as before to define the Euler number. However, bands with such boundaries may exist in the ambient $4$-manifold and must be considered.
In case $w_1(\Sigma)$ is nontrivial on precisely one component of $\partial S$, e.g.\ when $\partial S$ is connected, we know $\lambda_{\Sigma}(\partial S,\partial S)=1$.
Then, by \cref{lem:dependence_on_A}, if a band with such boundary exists, then $\km(F)=0$, and there is no need to define $\Theta$. In particular, note that this means that we need not consider the case of a M\"obius band whose boundary consists of a curve on which $w_1(\Sigma)$ is nontrivial. There is one final case of relevant bands left to consider, which we do next.

\begin{definition}\label{def:thetaA-nonorientable}
Let $F\colon (\Sigma,\partial \Sigma)\looparrowright (M, \partial M)$ be as in \cref{convention}, with $\mu (F)=0$.
    Let $A$ be a choice of Whitney arcs pairing the double points of $F$. Let $B$ be an annulus with a generic immersion of pairs $(B,\partial B) \imra (M,\Sigma)$, such that $\partial B$ transverse to $A$ and $w_1(\Sigma)$ is nontrivial on both components of~$\partial B$.
    Pick an embedded arc $D$ in $B$ connecting the two boundary components, disjoint from the self-intersections of $B$ and the intersections of $B$ with $F$.  Let $\wh B$ be the result of cutting $B$ open along $D$. There is a canonical quotient map $\gamma \colon \wh B \to B$.

    Let $\nu_B^M$ denote the normal bundle of $B$ in $M$. Pick a nowhere vanishing section of $\gamma^*\nu_B^M$ on $\partial \wh B$ as follows. On each part of $\partial \wh B$ that maps to $\partial B$, use the normal bundle of $\partial B$ in $F$ to define the section locally. For this we require that on $\partial D$ the two vectors for the two components agree up to multiplication by $\pm 1$. On the part of $\partial \wh B$ that maps to $D$ pick a section so that on every pair of points that map to the same point in $D$ the vectors agree up to multiplication by $\pm 1$. See the middle picture of \cref{fig:non-orientable-section}.
    We define
\begin{equation}\label{eq:theta-nonorientable-definition}
	\Theta_A(B,D):= \mu_{\Sigma}(\partial B)+ \lvert \partial B\pitchfork A \rvert + \lvert \Int B\pitchfork F \rvert +e(\wh B) \mod 2.
\end{equation}	

\end{definition}

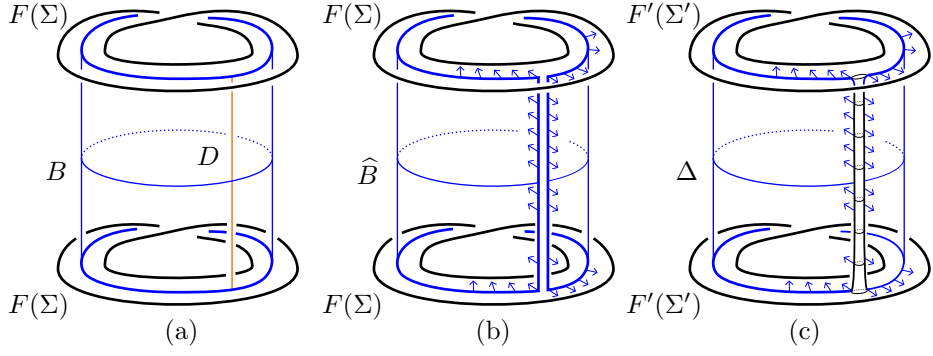
\begin{figure}[htb]
	\centering
\begin{tikzpicture}
        \node[anchor=south west,inner sep=0] at (0,0){	    \includegraphics[width=11.5cm]{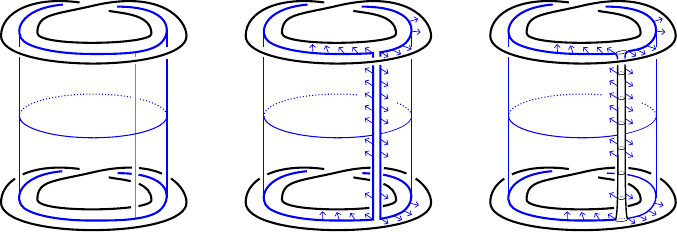}};
	   \node at (-0.23,0) {$F(\Sigma)$};
		\node at (-0.23,3.85) {$F(\Sigma)$};
		\node at (0, 1.8) {$B$};
		\node at (2, 2) {$D$};
	   \node at (3.9,0) {$F(\Sigma)$};
		\node at (3.9,3.85) {$F(\Sigma)$};
		\node at (4.1, 1.8) {$\wh{B}$};
	   \node at (8,0) {$F'(\Sigma')$};
		\node at (8,3.85) {$F'(\Sigma')$};
		\node at (8.3, 1.8) {$\Delta$};
		\node at (1.65,-0.35) {(a)};
		\node at (5.75,-0.35) {(b)};
		\node at (9.85,-0.35) {(c)};
	\end{tikzpicture}
	\caption{(a) A band $B$ (blue) is shown for $F$ (black). Both boundary components of $B$ are nonorientable on $\Sigma$. A vertical arc $D$ is shown in orange. (b) By cutting along the arc $D$ we obtain a disc $\wh{B}$. We show how to choose a nowhere vanishing section of the normal bundle of $\partial \wh{B}$ in $M$. (c) Adding a tube to $F$ guided by $D$ splits the band into a disc $\Delta$, and changes $\Sigma$ to a surface $\Sigma'$. We show how to choose a section of the normal bundle of $\partial \Delta$ in $F'$. }
	\label{fig:non-orientable-section}
\end{figure}

\begin{remark}\label{rem:absent-finger}
An alternative definition of $\Theta_A(B,D)$ would use the arc $D$ to add a tube to $F(\Sigma)$, in such a way that the tube intersects the band $B$ in two parallel copies of $D$. More precisely, we perform an ambient surgery on $F(\Sigma)$. This requires choosing a $2$-dimensional sub-bundle of the normal bundle of $D$ in $M$ -- the tube itself consists of the circle bundle for this sub-bundle, considered within a tubular neighbourhood of $D$. We build the required sub-bundle by first choosing a section lying in the normal bundle of $D$ in $B$, denoted $\nu^B_D$. The second section can be chosen freely.  Let $F'$ denote the immersion constructed by the tubing procedure. Observe that the domain of $F'$, denoted $\Sigma'$, is obtained from the abstract surface $\Sigma$ by adding a $1$-handle. Depending on the choice of the second section above, this may be an orientation-reversing or orientation-preserving $1$-handle, but this will not matter for us.

Adding the tube changes the band $B$ to a disc $\Delta$, by removing a thin strip neighbourhood of $D$. The disc $\Delta$ has boundary lying on $\Sigma'$. Observe that $\partial\Delta$ is orientation-preserving on $\Sigma'$, since it is the result of banding together two orientation-reversing curves. Since no new intersection points were added, the collection $A$ is a collection of Whitney arcs for the intersection points of $F'$. As a result, we may define
\[
\Theta_A(B,D):=\Theta_A(\Delta),\]
 where the latter is computed as in \cref{def:thetaA}. In order to see that this agrees with \cref{def:thetaA-nonorientable}, we need only check that the definition of the Euler number terms agree, assuming we choose the tube to be thin enough to miss any double points. For this compare Figures~\ref{fig:non-orientable-section}\,(b) and (c) to see that the sections at the boundary in both cases are the same and hence also the Euler numbers coincide.
 When comparing the pictures, the choice of the second section of the $2$-dimensional sub-bundle which determines the tube corresponds to the choice of the section of $\gamma^*\nu_B^M$ on the part of $\partial \wh B$ that maps to~$D$.
\end{remark}

Note that we did not prove that $e(\wh B)$ is independent of the choice of $D$. This will follow from the upcoming proof of \eqref{item-ii-lem-welldefined} in \cref{lem:well-defined} below, which states that $\Theta_A(B,D)$ depends only on the homology class of $B$. The following lemma, whose proof is again deferred to \cref{sec:km}, shows that $\Theta$ is well defined in all required cases. As a reminder, the case of orientable $\Sigma$ does not require the notion of $\Theta_A(B,D)$ from \cref{def:thetaA-nonorientable}, so parts \eqref{item-ii-lem-welldefined} and \eqref{item-iii-lem-welldefined} of the following lemma may be skipped by anyone only interested in that case.

\begin{lemma}
	\label{lem:well-defined}
Let $F\colon (\Sigma,\partial \Sigma)\looparrowright (M, \partial M)$ be as in \cref{convention}, with $\mu (F)=0$. Let $A$ be a choice of Whitney arcs pairing the double points of $F$.
	\begin{enumerate}[(i)]
	    \item\label{item-i-lem-welldefined} 	
Let $S$ be a compact surface, with a generic immersion of pairs $(S,\partial S) \imra (M,\Sigma)$, such that $\partial S$ is transverse to $A$ and $w_1(\Sigma)$ is trivial on every component of $\partial S$.
Then $\Theta_A(S) \in \Z/2$ depends only on the homology class of $S$ in $H_2(M,\Sigma;\Z/2)$.
	\item\label{item-ii-lem-welldefined} 	Let $B$ be an annulus, with a generic immersion of pairs $(B,\partial B) \imra (M,\Sigma)$, such that $\partial B$ is transverse to $A$ and $w_1(\Sigma)$ is nontrivial on both components of $\partial B$.  Pick an embedded arc $D$ in $B$ connecting the components of $\partial B$ and disjoint from all double points. Then $\Theta_A(B,D) \in \Z/2$ depends only on the homology class of $B$ in $H_2(M,\Sigma;\Z/2)$.  In particular, $\Theta_A(B,D)$ does not depend on $D$, so we write~$\Theta_A(B)$.
	\item\label{item-iii-lem-welldefined} Let $S$ be a surface as in \eqref{item-i-lem-welldefined} and let $B$ be an annulus as in \eqref{item-ii-lem-welldefined} such that $[S]=[B] \in H_2(M,\Sigma;\Z/2)$. Then $\Theta_A(S) = \Theta_A(B) \in \Z/2$.
	\item\label{item-iv-lem-welldefined} 	If  $\lambda_{\Sigma}|_{\partial \bands(F)}=0$, the restriction of $\Theta_A$
	to $\bands(F)$ is independent of the choice of $A$, giving a well defined map $\Theta\colon \bands(F)\to \Z/2$.
	\end{enumerate}
\end{lemma}

Finally we are ready to define the required generalisation of $r$-characteristic maps, called \emph{$b$-characteristic} maps.

\begin{definition}\label{def:b-char}
Let $F\colon (\Sigma,\partial \Sigma)\looparrowright (M, \partial M)$ be as in \cref{convention}, with $\mu (F)=0$. We say~$F$ is \emph{band characteristic} (or  \emph{$b$-characteristic} for short) if $\lambda_{\Sigma}|_{\partial \bands(F)}=0$ and $\Theta\colon \bands(F)\to \Z/2$ is trivial.
\end{definition}

\begin{lemma}\label{lem:r-b}
Every $b$-characteristic map is $r$-characteristic. Moreover, the two notions agree for unions of discs or spheres.
\end{lemma}

\begin{proof}
Let $F\colon (\Sigma,\partial \Sigma)\looparrowright (M, \partial M)$ be as in \cref{convention}, with $\mu (F)=0$. Let $R \colon \RP \to M$ be a generic immersion which is transverse to $F$ and so that $R^*w_1(M)=0$.
We apply \cref{ex:tubing-band}.  In other words, take a small disc on $F(\Sigma)$ away from
$R \pitchfork F$, and tube into the image of $R$.  This creates a M\"obius band $B$ with boundary on $\Sigma$. Here $\partial B$ is homotopically trivial in $\Sigma$, so $\langle w_1(\Sigma),\partial B \rangle =0$. The core $C$ of $B$ corresponds to $\mathbb{RP}^1$ within the original immersed $\RP$. Therefore, since $R^*w_1(M)=0$, we have that $\langle w_1(M), C\rangle=0$. So $B\in \bands(F)$.
Note that $\Theta\colon \bands(F)\to \Z/2$ is well defined by \cref{lem:well-defined}\,\eqref{item-iv-lem-welldefined} since $F$ is $b$-characteristic. Further $\Theta(B) = \Theta(R) =  F \cdot R + R \cdot R \in \Z/2$ by \cref{lem:well-defined}\,\eqref{item-i-lem-welldefined}
since $[B]=[R] \in H_2(M,\Sigma;\Z/2)$. But this vanishes since $F$ is $b$-characteristic. Hence $F$ is $r$-characteristic.

For the second sentence, suppose that $\Sigma$ is a union of discs or spheres and is $r$-characteristic.  Let $B \in \bands(F)$ be a band. Since $\Sigma$ is simply connected, the boundary of $B$ is null homotopic in $F(\Sigma)$. Therefore~$B$ can be closed up using a codimension zero submanifold of $\Sigma$ to either a sphere or an $\RP$ immersed in~$M$.
The resulting closed surface $R$ again satisfies $\Theta(B) = F \cdot R + R \cdot R \in \Z/2$ by \cref{lem:well-defined}\,\eqref{item-i-lem-welldefined}. Here $\lambda_{\Sigma}|_{\partial \bands(F)}=0$ since $\Sigma$ is simply connected and so $\Theta\colon \bands(F)\to \Z/2$ is again well defined by \cref{lem:well-defined}\,\eqref{item-iv-lem-welldefined}.
Once again, since null homotopic circles on $\Sigma$ must be orientation-preserving, $\langle w_1(\Sigma),\partial B \rangle =0$ and so \eqref{item:defn-band-w1-condn} implies that $R^*w_1(M)=0$. So $\Theta(B) =0$ since $F$ is $r$-characteristic. Thus $F$ is $b$-characteristic. It follows that the notions of $b$-characteristic and $r$-characteristic coincide as claimed.
\end{proof}

Recall that if $F^\twist$ is $b$-characteristic, then \cref{thm:main} states that $\km(F) = t(F^\twist,\W^\twist)$, so we have a combinatorial description of $\km(F)$.  Moreover, since $\Theta$ only depends on the homology class of a band in $H_2(M,\Sigma;\Z/2)$, we can in principle determine whether or not $F$ is $b$-characteristic by computing $\Theta$ on finitely many homology classes. Having said that, as mentioned in the introduction, in practice deciding precisely which homology classes can be represented by maps of bands may be tricky.

\begin{lemma}\label{lem:bchar-reg-htpy}
Let $F\colon (\Sigma,\partial \Sigma)\looparrowright (M, \partial M)$ be as in \cref{convention}, with $\mu (F)=0$. If $G$ is regularly homotopic to $F$ and  $F$ is $b$-characteristic then $G$ is $b$-characteristic.
\end{lemma}

\begin{proof}	
	By definition, a regular homotopy can be decomposed into a sequence of ambient isotopies, finger moves, and Whitney moves. None of these affect which classes of $H_1(\Sigma;\Z/2)$ bound a band. In particular, $\lambda_\Sigma\vert_{\partial\bands(F)} =\lambda_\Sigma\vert_{\partial\bands(G)}$.
Assume that $G$ is not $b$-characteristic. Then either $\lambda_\Sigma\vert_{\partial\bands(G)}$ is nontrivial or there is a  band in $\bands(G)$ on which $\Theta$ is nonvanishing. In the former case $\lambda_\Sigma\vert_{\partial\bands(F)} =\lambda_\Sigma\vert_{\partial\bands(G)}$ implies that $F$ is not $b$-characteristic.

For the latter case, let $B$ in $\bands(G)$ be such that $\Theta(B)=1$. It suffices to show that such a band still exists after an ambient isotopy, a finger move, or a Whitney move. This is obvious for ambient isotopy. Recall that $\Theta$ only depends on the homology class of the band by \cref{lem:well-defined}. Hence we can assume that the boundary of $B$ is away from the singularity of the finger move. Then we can still consider the band $B$ as a band for the surface after the finger move and $\Theta$ is unchanged. The argument in the case of a Whitney move is similar. We first let $B$ undergo a homotopy to arrange that it is disjoint from the boundary of the Whitney disc $W$ along which the Whitney move is performed. Then we can again consider the same band $B$ for the new surface. The Whitney move leaves all terms in the definition of $\Theta$ except $|\Int B \pitchfork F|$ unchanged. Since the Whitney move uses two copies of the Whitney disc, the change in $|\Int B\pitchfork F|$ is twice $|\Int B\pitchfork W|$. As $\Theta$ takes values in $\Z/2$, $\Theta(B)$ is unchanged as claimed.  Thus we have a band $B$ in $\bands(F)$ with $\Theta(B)=1$, and so again $F$ is not $b$-characteristic. We have shown the contrapositive of the desired statement.
\end{proof}

\begin{remark}
As a counterpoint to \cref{lem:bchar-reg-htpy}, there exist maps that are homotopic to each other, but where one is $b$-characteristic and the other is not. For example, let $\Sigma$ be the Klein bottle. Then an embedding $f\colon \Sigma \hookrightarrow \R^4$ must have normal Euler number $e(\nu f)\in \{-4,0, 4\}$ by \cite{Massey-nonorientable}. It can be verified, as we do presently, that the embeddings with $e(\nu f)=0$ are precisely those which are $b$-characteristic. Hence the $b$-characteristic notion is not invariant under homotopy.

To see that $f$ is $b$ characteristic if and only if $e(\nu f)=0$, think of  $\Sigma \cong \RP \# \RP$, with a corresponding isomorphism $H_1(\Sigma;\Z/2) \cong \Z/2 \oplus \Z/2$. There is a standard embedding of $\RP$ in $\R^4$ with normal Euler number $\pm 2$, and there are essentially three ways to take connected sums of these embeddings, realising the three options $e(\nu f)\in \{-4,0, 4\}$. With $e(\nu f)$ fixed, these embeddings are unique up to regular homotopy by \cref{theorem:generic-immersions-bijection}.

We explicitly construct the standard embeddings, as follows. Take two disjoint, unlinked, unknotted  M\"{o}bius bands $M_1$ and $M_2$ in $\R^3$, with an $\varepsilon_i \in \{\pm 1\}$ signed half-twist, for $i=1,2$.  Take the boundary connected sum $M_1 \natural M_2$ ambiently to obtain a punctured Klein bottle in $\R^3$ with boundary an unknot.  Cap this unknot off with a standard slice disc in  $\R^4_{\geq 0}$ to obtain a standard embedding $f$ with normal Euler number $-2(\varepsilon_1 + \varepsilon_2)$. We do not justify the sign, which depends on conventions that are not important for us.

By \cref{lem:bchar-reg-htpy}, it suffices to check whether the three standard embeddings above are $b$-characteristic, which we do next. First one computes that $\lambda_{\Sigma}|_{\partial \bands(f)}=0$ in all three cases, as follows. We have \[\bands(f)\subseteq H_2(\R^4,\Sigma;\Z/2)\xrightarrow[\partial]{\cong} H_1(\Sigma;\Z/2)\cong \Z/2\oplus \Z/2.\] By \eqref{item:defn-band-w1-condn} in order for $B\in \bands(f)$ we need $\langle w_1(\Sigma),h(\partial B)\rangle=0$. Hence $\partial B \in H_1(\Sigma;\Z/2)\cong \Z/2\oplus \Z/2$ is either $(0,0)$ or $(1,1)$.
To see that $(x,x)$, for $x \in \{0,1\}$, can be realised as the boundary of a band, pick a simple closed curve $Z$ on $\Sigma$ representing the homology class $(x,x)$ and a generically immersed disc $D$ bounded by $Z$ in $\R^4$. Then add a tube from $D$ to $\Sigma$ to turn $D$ into an annulus $B$, using  \cref{ex:tubing-band} (see \cref{fig:tubing-band}).
Thus $\partial\bands(f) = \{(0,0),(1,1)\}$, on which $\lambda_\Sigma$ vanishes.

Hence whether or not the given standard embedding of $\Sigma$ is $b$-characteristic is decided by $\Theta(B)$, where $B$ is a band with boundary $(1,1) \in H_1(\Sigma;\Z/2)$. For the standard embeddings constructed above, such a band can be constructed explicitly, as follows. Take the core curves of $M_1$ and $M_2$, and connect sum them inside $M_1 \natural M_2$. This gives an unknot representing $(1,1)$, which bounds a standard slice disc $D$ in $\R^4_{\leq 0}$. \cref{ex:tubing-band} converts $D$ to a band $B$.

Since $\Theta$ only depends on the class of a band in $H_2(M,\Sigma;\Z/2)$, we can use the band $B$ from the previous paragraph. We shall compute that $\Theta(B)=0$ for this band if and only if $e(\nu f)=0$, that is if the embedding of $\Sigma$ arises from the connected sum of the standard embeddings of $\RP$ with opposite normal Euler numbers. The curve $\partial B$ is orientation preserving on $\Sigma$, so we use \eqref{eq:Theta-definition} to compute $\Theta(B)$. Most of the terms in this definition are trivial in this case, since we are working with an embedding of $\Sigma$ and $\partial B$ is itself embedded. Also $D$ has interior disjoint from the image of $f$, and therefore so does~$B$. Only the relative Euler number term remains, which can be computed from the twists in $M_1$ and $M_2$.  It follows that $\Theta(B) \equiv (\varepsilon_1 + \varepsilon_2)/2 \in \Z/2$, which vanishes if and only if $\varepsilon_1 = - \varepsilon_2$, which in turn holds if and only if $e(\nu f)=0$.
\end{remark}

\section{Homotopy versus regular homotopy}\label{sec:homotopy-reg-homotopy}
In this short section, we describe \cref{const:change-t} and apply it to prove \cref{theorem:changing-t-intro}, which we used in \cref{subsection:homotopy-vs-reg-homotopy} to compare homotopy and regular homotopy of maps. Note that the results in this section require that the surface $\Sigma$ from \cref{convention} is nonorientable.

If we are interested in finding an embedding in a given homotopy class, rather than a regular homotopy class, we may use the construction below to replace a given map by a homotopic map for which the invariant $t$ is trivial. In particular, the construction is applicable in the cases from \cref{theorem:generic-immersions-bijection} in which there are infinitely many regular homotopy classes with $\mu(-)_1 =0$ in a given homotopy class.

\begin{construction}\label{const:change-t}
Let $F=\{f_i\}_{i=1}^m \colon (\Sigma,\partial \Sigma)\looparrowright (M, \partial M)$ be as in \cref{convention}, with $\mu (F)=0$.  Let $\W$ denote a convenient collection of Whitney discs for the double points of $F$. Suppose that there exists $f_i$ with  $w_1(\Sigma)|_{\ker (f_i)_\bullet}$ nontrivial.

Let $N\subseteq \Sigma_i$ be a M\"{o}bius band with $f_i|_N$ $\pi_1$-trivial. In a small disc in $N$ introduce four double points with the same sign by cusp homotopies and call the resulting immersion $f'_i$. Let $F'$ denote the map given by $\{f_j\}_{j\neq i} \cup \{f'_i\}$. Then there is a convenient collection of Whitney discs $\W'$ for all the double points of $F'$ such that
\[
t(F',\W')\equiv t(F,\W)+1 \mod{2}.
\]
\end{construction}
While we have created four double points with the same sign, we will use in the proof that the M\"{o}bius band is nonorientable to change the sign of two of the double points, in order to then be able to find new Whitney discs.

\begin{proof}[Proof of \cref{const:change-t}]
We will pair up the two new pairs of double points with new Whitney discs. Pick any pair of the four new double points and pair them by arcs in the small disc, as in \cref{defn:paired-arcs-opposite-sign}, such that the resulting circle is null-homotopic in $M$. For one of the arcs perform a connected sum in the interior with the core $\alpha$ of $N$. With this new pair of arcs, the double points have opposite sign, and by our choice of $N$ the resulting Whitney circle bounds a Whitney disc $W_1$ in $M$ with embedded boundary. By  boundary twisting, arrange that $W_1$ is framed, and by pushing off  ensure there is no intersection between the boundary of $W_1$ and the boundaries of the components of $\W$. For this we push the boundary arc of $W_1$ over the end of the boundary arc for the Whitney disc in $\W$. This way $t(F,\W)$ remains unchanged. Do the same for the remaining two new double points in $f_i'$, namely pair them by a Whitney disc $W_2$, which by definition is a parallel copy of $W_1$.

Since $\lambda_N(\alpha,\alpha)=1$, the boundaries of $W_1$ and $W_2$ intersect an odd number of times. To turn $\W\cup\{W_1,W_2\}$ into a convenient collection of Whitney discs we have to remove any intersections between their boundaries. For such an intersection, push the Whitney arc of $W_1$ over the end of the Whitney arc of $W_2$. This will in turn change the number of intersections between the interior of $W_1$ and $f'_i$ by one mod $2$; that is, $|\Int W_1\pitchfork  F|\equiv |\Int W_2\pitchfork  F|+1\mod 2$. Let $\W'$ be the resulting collection of Whitney discs. We have
	\[
t(F',\W')\equiv t(F,\W)+|\Int W_1\pitchfork F|+|\Int W_2\pitchfork F|\equiv t(F,\W)+1 \mod{2}.\qedhere\]
\end{proof}

With the above construction in hand, we can now prove \cref{theorem:changing-t-intro} from the introduction.

\begin{reptheorem}{theorem:changing-t-intro}
Let $F=\{f_i\}_{i=1}^m \colon (\Sigma,\partial\Sigma)\looparrowright (M,\partial M)$ be as in \cref{convention} with $\mu (F)=0$.  Suppose that there is at least one $f_i \in F^{\twist}$ with  $w_1(\Sigma)|_{\ker (f_i)_\bullet}$ nontrivial.  Then there exists a generic immersion $F'$ homotopic to $F$ with $\mu(F')=0$, and a convenient collection of Whitney discs $\W'$ such that $t((F')^{\twist},(\W')^{\twist})=0$.
Thus if $F'$ has algebraically dual spheres then $\km(F')=0$, and if moreover $\pi_1(M)$ is good then $F'$ is regularly homotopic, relative to~$\partial \Sigma$, to an embedding.
\end{reptheorem}

\begin{proof}
By the vanishing of the intersection and self-intersection numbers, there is a convenient collection of Whitney discs $W$ for $F$ and therefore for $F^\twist$. If $t(F^\twist, W^\twist)=0$ set $F'=F$. If $t(F^\twist, W^\twist)=1$, use \cref{const:change-t} to find a generic immersion $F'$ homotopic to $F$, with $t((F')^{\twist},(W')^{\twist})=0$. If $F'$ has algebraically dual spheres, then $\km(F')=0$ by \cref{thm:main} since either $F'$ is not $b$-characteristic or $\km(F')=t((F')^{\twist},(W')^{\twist})=0$. If in addition $\pi_1(M)$ is good, apply \cref{thm:SET} to see that $F'$ is regularly homotopic, relative to $\partial\Sigma$, to an embedding.
\end{proof}

We end this section by giving another pair of applications of \cref{const:change-t}.

\begin{proposition}\label{corollary:changing-t-2}
	Let $f\colon (\Sigma,\partial \Sigma)\looparrowright (M, \partial M)$ be a generic immersion as in \cref{convention}. Assume that $\Sigma$ is connected, $\mu(f)=0$, and that $w_1(\Sigma)|_{\ker f_\bullet}$ is nontrivial while $f^*(w_1(M))$ is trivial. If $f'$ is a generic immersion homotopic to $f$, both $f$ and $f'$ are $b$-characteristic, and $e(\nu f)-e(\nu f')=\pm 8$, then
	\[t(f')\equiv t(f)+1 \mod{2}.\]
\end{proposition}

\begin{proof}
Since $f^*(w_1(M))$ is trivial, regular homotopy classes of generic immersions homotopic to $f$ are detected by the Euler number of the normal bundle by \cref{theorem:generic-immersions-bijection}. Further, since $e(\nu f)-e(\nu f')=\pm 8$ we may add four cusps (of the same sign) to $f$ to obtain a map $f''$ which is regularly homotopic to $f'$.

The map $f$ is $b$-characteristic by assumption, while the map $f''$ is $b$-characteristic since $f'$ is, by \cref{lem:bchar-reg-htpy}. By \cref{lem:mu1}, we know that $\mu(f)_1\in \Z/2$, so by construction $\mu(f)=\mu(f'')=\mu(f')=0$. So the quantities $t(f)$ and $t(f'')$ are defined, and further $t(f'')=t(f')$. Apply \cref{const:change-t} to see that $t(f'')\equiv t(f)+1\mod{2}$.
\end{proof}

Applying~\cref{corollary:changing-t-2} to immersions of $\RP$ into $\R^4$ we obtain the following corollary, obstructing generic immersions of $\RP$ in $\R^4$ with $e(f)\neq \pm 2 \mod 16$ from being regularly homotopic to an embedding and thus partially recovers the result due to Massey \cite{Massey-nonorientable} that every embedding of $\RP$ in $\R^4$ must have Euler number $\pm 2$.  Massey stated the result for smooth embeddings, since he used the $G$-signature theorem. But the $G$-signature theorem was later extended to the topological category by \cite{Wall-surgery-book}*{Chapter~14B}, so Massey's result also holds for locally flat embeddings of~$\RP$ in~$\R^4$.

\begin{corollary}\label{cor:RP-euler-no-restrictions}
	Let $f\colon \RP\looparrowright \R^4$ be a generic immersion with $\mu(f)=0$.  Then $t(f)=0$ if and only if $e(\nu f)=\pm 2\mod 16$.
\end{corollary}

\begin{proof}
Recall that there exist embeddings $g_\pm\colon \RP\hookrightarrow\R^4$ with Euler number $\pm 2$. First we prove that $e(\nu f)\equiv 2\mod{4}$. By \cref{lem:mu1}, we know that $\mu(f)_1\in \Z/2$. Then $\mu(g_+)=0$, and so $\mu(g_+)_1=0$. Since $f$ is homotopic to $g_+$ and $\mu(f)=0$, it follows from \cref{theorem:generic-immersions-bijection} that $e(\nu f)\equiv e(\nu g_+)\equiv 2\mod{4}$. (The same argument would have applied with $g_-$.)

Note that any generic immersion of $\RP$ into $\R^4$, and in particular the map $f$, is $b$-characteristic since $H_2(\R^4,\RP;\Z/2)\cong \Z/2$ and the nontrivial element does not satisfy condition~\eqref{item:defn-band-w1-condn} in \cref{def:band}.

We have $t(g_{\pm})=0$ since $t$ vanishes for embeddings. Since $e(\nu f)\equiv 2\mod{4}$, it differs from one of $\pm 2$ by a multiple of $8$. Let $k\in \Z$ be such that $e(\nu f)=\pm 2+8k=e(\nu g_\pm)+8k$. By \cref{corollary:changing-t-2},
\[t(f)\equiv t(g_\pm)+k\equiv k \mod 2.\]
 Thus $t(f)=0$ if and only if $k$ is even, which is the case precisely when $e(\nu f)$ differs from $e(\nu g_\pm)=\pm 2$ by a multiple of $16$.
\end{proof}

\section{\texorpdfstring{Proofs of statements from \cref{sec:intro-km-section}}{Proofs of statements from Section 5}}\label{sec:km}

In this section, we provide the proofs we skipped in \cref{sec:intro-km-section}. 
The following transfer move will be useful for arranging that algebraically cancelling intersection points occur on the same Whitney disc.

\begin{construction}[Transfer move]\label{lem:transfer-move}
Let $\Sigma$ and $M$ be as in \cref{convention} and let $H\colon (\Sigma,\partial\Sigma) \to (M,\partial M)$ be a generic immersion, with components $\{h_i\colon \Sigma_i\to M\}$.
Assume the double points within $H$ are paired by a convenient collection $\W$ of Whitney discs.

Let $W_1$ and $W_2$ be components of $\W$ with $\Int {W}_1\pitchfork H\neq \emptyset \neq \Int {W}_2\pitchfork H$.  We can perform three finger moves on $H$, so that the resulting generic immersion $H'$ has six new double points, paired by three framed, embedded Whitney discs $\{V, U_1, U_2\}$, each of which has two intersections with $H'$, and such that the boundaries of $\{V, U_1, U_2\}$ are mutually disjoint and embedded. Moreover, the collection $\mathcal{W}':=\mathcal{W}\cup \{V, U_1, U_2\}$ is a convenient collection of Whitney discs for $H'$ and we have
\begin{align*}
\lvert \Int {W}_1\pitchfork H'\rvert&=  \lvert \Int {W}_1\pitchfork H\rvert-1\text{ and}\\
\lvert \Int {W}_2\pitchfork H'\rvert&=  \lvert \Int {W}_2\pitchfork H\rvert-1.
\end{align*}
\end{construction}

\begin{figure}[htbp]
    \centering
\begin{tikzpicture}	
        \node[anchor=south west,inner sep=0] at (0,0){	    \includegraphics[scale=1]{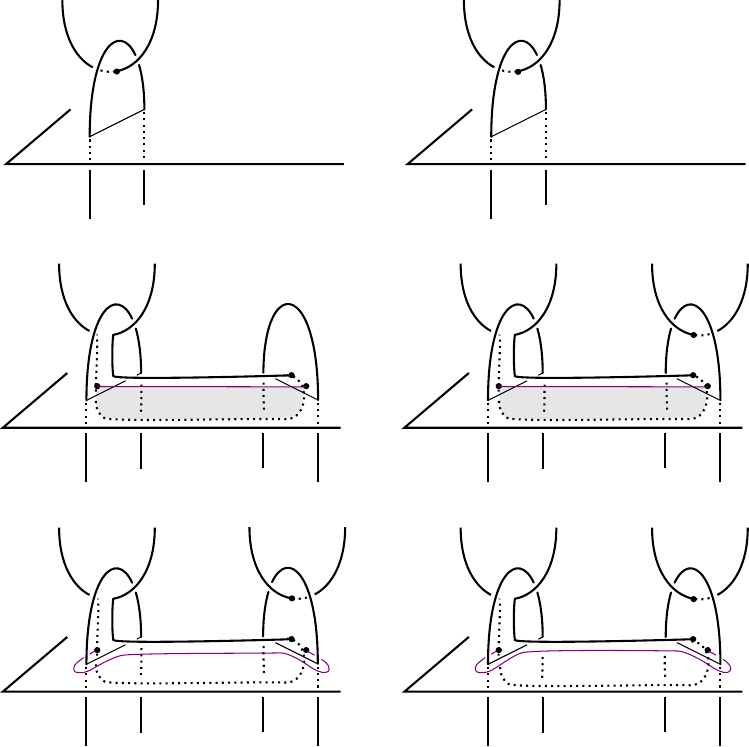}};
		\node at (6.4,9) {(i)};
		\node at (0.25,10.4) {$h_a$};
		\node at (1.3,12.5) {$h_e$};
		\node at (2.75,11.25) {$W_1$};
		\node at (1.8,9.25) {$h_b$};
		\node at (7.1,10.4) {$h_c$};
		\node at (9.55,11.25) {$W_2$};
		\node at (8.1,12.5) {$h_f$};
		\node at (8.6,9.25) {$h_d$};
				
		\node at (6.4,4.5) {(ii)};
		\node at (0.25,6) {$h_a$};
		\node at (1.25,8) {$h_e$};
		\node at (2.7,6.7) {$W_1$};
		\node at (1.75,4.7) {$h_b$};
		\node at (5.15,4.7) {$h_c$};
		\node at (5.55,6.7) {$V$};
		\node at (7.1,6) {$h_c$};
		\node at (9.5,6.7) {$W_2$};
		\node at (8.1,8) {$h_f$};
		\node at (8.55,4.7) {$h_d$};
		\node at (3.5,5.8) {$U_1$};
		\node at (10.35,5.8) {$U_2$};
		\node at (11.9,4.7) {$h_a$}; 	
		\node at (12.35,8) {$h_e$};		
		\node at (12.35,6.7) {$V$};
		
		\node at (6.4,0.1) {(iii)};
		\node at (0.25,1.5) {$h_a$};		
		\node at (1.25,3.5) {$h_e$};
		\node at (2.7,2.25) {$W_1$};
		\node at (1.75,0.25) {$h_b$};
		\node at (5.15,0.25) {$h_c$};
		\node at (5.55,2.25) {$V$};
		\node at (7.1,1.5) {$h_c$};
		\node at (9.5,2.25) {$W_2$};
		\node at (8.1,3.5) {$h_f$};
		\node at (8.55,0.25) {$h_d$};
		\node at (11.9,0.25) {$h_a$};
		\node at (12.35,3.5) {$h_e$};	
		\node at (5.6,3.5) {$h_f$};
		\node at (12.35,2) {$V$}; 						
\end{tikzpicture}
    \caption{The transfer move. (i) Whitney discs $W_1$ and $W_2$ pairing intersections between $h_a$ and $h_b$, and between $h_c$ and $h_d$ respectively. (ii) A finger move between $h_a$ and $h_c$ creates a new pair of intersections, paired by a Whitney disc $V$, shown on both panels. Depicted on the left panel, an intersection between $W_1$ and $h_e$ is pushed down into $h_a$ and then one of the resulting intersection points is pushed across to $V$. In the right panel, we see this new intersection between $V$ and $h_e$. Further, an intersection between $W_2$ and $h_f$ is pushed down to $h_c$ and one of the resulting intersection points is pushed over to $V$. These last three moves form a regular homotopy of $H$, with result $H'$. Each $W_i$ has one fewer intersection with $H'$ than with $H$, at the expense of creating two new intersections within $H'$, paired by Whitney discs $U_1$ and $U_2$, both shaded grey. Additionally, $V$ has two intersections with $H'$. The result of a boundary push off operation making the Whitney arcs for the $U_i$ (purple) disjoint from the Whitney arcs for $W_i$ and $V$ is shown in (iii). This operation creates two intersections of each $U_i$ with $H'$.}
\label{fig:transfer-move}
\end{figure}
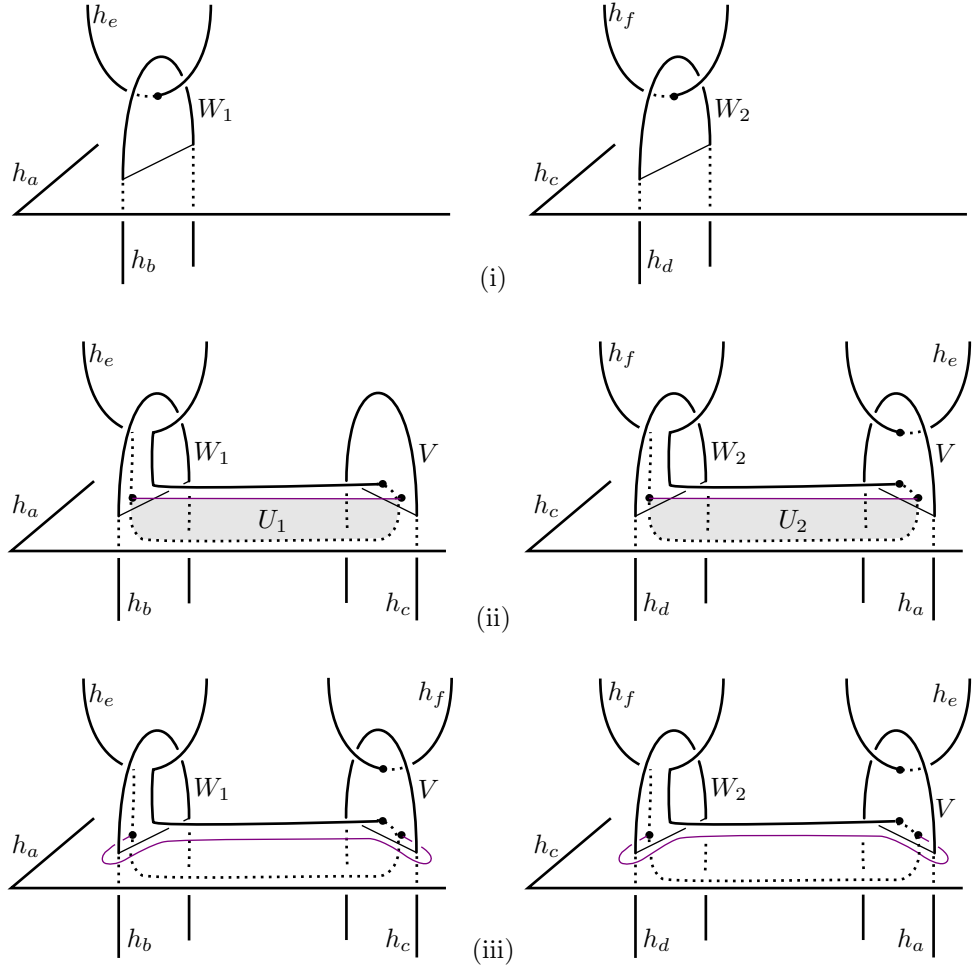

\begin{proof}[Proof of \cref{lem:transfer-move}]
Suppose that $W_1$ pairs intersections of $h_a$ and $h_b$ while $W_2$ pairs intersections of $h_c$ and $h_d$, where repetition within $a,b,c,d$ is allowed. Perform a finger move between $h_a$ and $h_c$, creating two new double points paired by a corresponding framed, embedded Whitney disc $V$. Note that the interior of $V$ is disjoint from the image of $H$. The operation depicted in \cref{fig:transfer-move} gives a further regular homotopy, involving a finger move pushing $h_e$ through $h_a$, and a finger move pushing $h_f$ through $h_c$. We call the outcome of all three finger moves~$H'$. The procedure creates six new intersections within $H'$ compared with $H$.  The four intersections created by the $h_e$ -- $h_a$ and $h_f$ -- $h_c$ finger moves are paired by Whitney discs~$U_1$ and~$U_2$. A preliminary version of these are shown in the middle panel of \cref{fig:transfer-move}; the final versions are those arising after the boundary push off operations indicated by the bottom panel.
Overall, the move transfers an intersection of $H$ with $W_1$, as well as an intersection of $H$ with $W_2$, on to $V$, so that $|\Int{V} \pitchfork H' | =2$.  By construction, each $U_i$ intersects $H'$ twice.
\end{proof}

Now we prove \cref{lem:tau}, whose statement we recall.

\begin{replemma}{lem:tau}
Let $F\colon (\Sigma,\partial \Sigma)\looparrowright (M, \partial M)$ be as in \cref{convention}. Suppose that $F$ admits algebraically dual spheres, and that all double points of $F$ are paired by a convenient collection $\W$ of Whitney discs. Let $\W^\twist \subseteq \W$ denote the sub-collection of Whitney discs for the intersections within $F^\twist$, where $F^\twist$ is as in \cref{def:Ft}.
If $\tw(F^\twist,\W^\twist)=0$, then $\km(F)=0$.
\end{replemma}

\begin{proof}
By applying the geometric Casson lemma (\cref{lem:geometric-casson-lemma}) and \cref{prop:regular-homotopy-inv-lambda,prop:regular-homotopy-inv-mu,lemma:km-reg-homotopy-invariance}, we may arrange by a regular homotopy that $F$ and $G$ become geometrically dual.
By definition
\begin{equation}\label{eq:count-zero}
 \tw(F^\twist,\W^\twist) = \Sum_{\ell,i}\ \lvert \Int{W^\twist_\ell}\pitchfork f^\twist_i\rvert  =0\in \Z/2.
\end{equation}
We modify the collection of Whitney discs, as follows, so that each has an even number of  intersections with $F^\twist$. Since the count in \eqref{eq:count-zero} is zero, the number of Whitney discs with odd intersection with $F^\twist$ is even, so we may pair them up (arbitrarily).
For each such pair, apply \cref{lem:transfer-move}. This changes $F$ by finger moves to some $F'$, whose double points are paired by a convenient collection of Whitney discs $\W' := \{W'_{\ell}\}$, such that each element of $\W'$ has an  even number of intersections with $(F')^\twist$. Note that the new Whitney discs created by the application of \cref{lem:transfer-move} have been added to the collection.

For each intersection of some $W'_\ell$ with $(F')^\twist$, tube $W_{\ell}'$ into the corresponding geometrically dual sphere. Note that each sphere  being tubed into is necessarily twisted, but since we tube an even number of times, the total change in the framing of $W'_\ell$ is even.  Do this for each element of $\W'$. The resulting family of Whitney discs may still intersect $F'$, but not $(F')^\twist$. For each such intersection with $F'$, again tube into the appropriate geometrically dual sphere. Now the spheres are not twisted, so the framing of the Whitney discs changes by an even number. Arrange for  all the Whitney discs to be framed by adding local cusps in the interior.  We may do this because the framing coefficient of each of the Whitney discs is even. We have now produced the desired convenient collection of Whitney discs for the intersections within $F'$, whose interiors lie in the complement of $F'$. This shows that $\km(F')=0$, and therefore $\km(F)=0$, as desired.
\end{proof}

Before giving the proof of \cref{lem:dependence_on_A}, we explain the key new construction in this paper, which we already mentioned in~\cref{sec:intro-bchar}. Briefly, given a band $B$ with boundary lying on an immersed surface, a finger move along a fibre of the band produces two new double points paired by a Whitney disc arising from $B$.

\begin{construction}[Band fibre finger move]\label{construction:finger-move}
Let $F\colon (\Sigma,\partial \Sigma)\looparrowright (M, \partial M)$ be as in \cref{convention}. Suppose that $\mu (F)=0$ and that the self-intersections of $F$ are paired by a convenient collection $\W = \{W_\ell\}$ of Whitney discs with boundary arcs $A$.

Consider $B\in \bands(F)$ as a $D^1$-bundle. Then we can do a finger move along a fibre of $B$, with endpoints missing $A$, as depicted in \cref{fig:band-gives-whitney-disc}.  We call this fibre the \emph{finger arc}, and denote it by~$D$.  We assume that $D$ misses all double points of $B$ and all intersections between $\Int B$ and $F$.

A finger move depends on a choice of $2$-dimensional sub-bundle of the normal bundle to the finger arc (the proof of~\cref{lem:fiber-move} will give further details).  We use a sub-bundle that lies in the tangent bundle $TF$ at one end of the arc, contains $\nu_{D}^B$ along $D$, and intersects $TF$ in the line bundle  $\nu_D^B$ at the other end of the arc. Here $\nu_D^B$ is the normal bundle of $D$ in $B$.

Call the immersion resulting for the above finger move $F'$. We will check in the next paragraph that the remainder of $B$, i.e.\ the complement in $B$ of a tubular neighbourhood of $D$, gives a Whitney disc for the new pair of double points. Make the boundary embedded and disjoint from $A$, by boundary push off operations, and then boundary twist if necessary, to obtain a framed Whitney disc $W_B$ for the new double points.  Then $\W' := \W \cup \{W_B\}$ is a convenient collection of Whitney discs  pairing the double points of $F'$. Note that both $F'$ and $W_B$ depend on the choice of finger arc $D$ and the $2$-dimensional sub-bundle of its normal bundle mentioned above.

Now, as promised, we check that $W_B$ is a Whitney disc. The finger move creates a trivial Whitney disc and we refer to the corresponding Whitney arcs as the trivial arcs. The double points are also paired by the arcs $A_1, A_2\subseteq \partial B$ where $A_1\cup A_2=\partial W_B$.  The existence of the disc $W_B$ implies that the group elements of the double points agree with respect to the arcs $A_1$ and $A_2$. It remains only to to see that the double points have opposite signs with respect to the arcs $A_1$ and $A_2$ (see~\cref{defn:opposite-signs}). The case that both $M$ and $\Sigma$ are orientable is straightforward. The general case follows from the $w_1$ condition in the definition of a band,~\eqref{item:defn-band-w1-condn}, as we now check.

First we consider the case that $B$ is an annulus. Let $\partial_1 B$ and $\partial_2B$ denote the two components of $\partial B$. Then $A_1$ and $A_2$ differ from the trivial arcs joining the new double points by $\partial_1 B$ and $\partial_2B$ respectively. By \cref{defn:opposite-signs} the double points have opposite sign precisely when $\langle w_1(\Sigma), \partial B\rangle +\langle w_1(M),\partial_1B\rangle=0$, which matches~\eqref{item:defn-band-w1-condn} since $\partial_1B$ is homotopic in $M$ to the core of $B$.

Now suppose that $B$ is a M\"obius band. Then the union of $A_1$ and $A_2$ and the trivial pair of arcs is the circle $\partial B\subseteq \Sigma$. Moreover, the union of the image of $A_1$ and either one of the trivial arcs forms a circle in $M$ homotopic to the core $C$ of $B$. As before, the double points have opposite sign precisely when $\langle w_1(\Sigma), \partial B\rangle +\langle w_1(M),C\rangle=0$, which again matches~\eqref{item:defn-band-w1-condn}.
\end{construction}

The following lemma explains how \cref{construction:finger-move} changes the value of $t$. In the proof we will also carefully explain how to make suitable choices of $2$-dimensional sub-bundles, as required for the finger move in \cref{construction:finger-move}.

\begin{lemma}
	\label{lem:fiber-move}
Let $F\colon (\Sigma,\partial \Sigma)\looparrowright (M, \partial M)$ be as in \cref{convention}.
\begin{enumerate}[(i)]
	    \item\label{item-lem-fibre-move-i} Suppose that $F'$ and $\W'$ are obtained from $F$ and $\W$ by a single application of \cref{construction:finger-move} with respect to a band $B\in \bands(F)$, where $w_1(\Sigma)$ restricted to every component of $\partial B$ is trivial. Let $A$ denote the Whitney arcs corresponding to $\mathcal{W}$. Then
	    \[t(F',\W') = t(F,\W) + \Theta_A(B) \in \Z/2.\]
	    \item\label{item-lem-fibre-move-ii}
	   Suppose that $F'$ and $\W'$ are obtained from $F$ and $\W$ by a single application of \cref{construction:finger-move} with respect to a band $B\in \bands(F)$ and an arc $D\subseteq B$, where $B$ is an annulus with $w_1(\Sigma)$ nontrivial on both boundary components and $D \subseteq B$ connects the two boundary components. Let $A$ denote the Whitney arcs corresponding to $\mathcal{W}$. Then
	\[t(F',\W') = t(F,\W) + \Theta_A(B,D) \in \Z/2.\]
	\end{enumerate}
\end{lemma}

For the proof it will be advantageous to refrain from applying boundary twists and removing intersections involving $\partial W_B$, and to instead use the following alternative definition of $t(F,\W)$, using a slightly weaker restriction on collections of Whitney discs, as in~\cite{Stong}, cf.~\cite{FQ}*{Section~10.8A}.

\begin{definition}\label{def:weak-Whitney}
Let $F\colon (\Sigma,\partial \Sigma)\looparrowright (M, \partial M)$ be as in \cref{convention}. A \emph{weak} collection of Whitney discs for $F$ is a collection $\mathcal{W}$ of Whitney discs pairing all the double points of $F$, with generically immersed interiors transverse to $F$, and with Whitney arcs whose interiors are generically immersed in $F(\Sigma)$ minus the double points of~$F$.
\end{definition}

In particular, compared to the definition of a convenient collection of Whitney discs (\cref{def:convenient-Whitney}), we allow the boundaries of Whitney discs to be generically immersed on $\Sigma$ and to intersect one another transversely.  We also allow the Whitney discs to be twisted, i.e.\ for the framing of the normal bundle restricted to the boundary to disagree with the Whitney framing.  Each of the discs in a weak collection of Whitney discs admits a normal bundle. The proof is the same as for a generic immersion of pairs, but with a preliminary step than one has to first fix the normal bundle in neighbourhoods of the two double points being paired.

\begin{definition}\label{def:t_alt}
Given a weak collection of Whitney discs $\mathcal{W}:=\{W_\ell\}$ for the double points of a generic immersion $F$ as in \cref{convention}, fix an ordering on the indexing set for $\W$ and define
    \[t_{\Alt}(F,\W) := \sum_{\ell} \Big( \mu_\Sigma(\partial W_\ell)  + \sum_{k > \ell} \lvert \partial W_\ell \pitchfork \partial W_k \rvert+  \sum_i \lvert \Int W_\ell \pitchfork f_i \rvert   +  e(W_\ell) \Big)  \in \Z/2,\]
where $e(W_\ell)$ is the relative Euler number of the normal bundle with respect to the Whitney framing on~$\partial W_\ell$.
    \end{definition}

Note that if $\mathcal{W}$ is a convenient collection of Whitney discs for $F$, then $t_{\Alt}(F,\mathcal{W})=t(F,\mathcal{W})$ (cf.~\cref{def:t}). In particular, since a convenient collection of Whitney discs comprises framed Whitney discs and has embedded and disjoint Whitney arcs, a majority of terms in the definition of $t_{\Alt}$ vanish.
The following lemma shows that $t_{\Alt}$ can be used as a proxy for $t$ in general.

\begin{lemma}\label{lem:alt-enough}
    Given a weak collection of Whitney discs $\mathcal{W}:=\{W_\ell\}_{\ell=1}^n$ for the double points of a generic immersion $F$ as in \cref{convention},
there exists a convenient collection of Whitney discs $\W'$ such that $t(F,\W') = t_{\Alt}(F,\W)$.
\end{lemma}

\begin{proof}
For each Whitney disc $W_\ell$ with $e(W_\ell) \neq 0$, add boundary twists to obtain $\ol{W}_\ell$ with $e(\ol{W}_\ell) =0$.
Each boundary twist changes the Euler number by $\pm 1$ and introduces an intersection of the Whitney disc with~$F$.
We have
    \[\sum_i \lvert \Int \ol{W}_\ell \pitchfork f_i \rvert \equiv  \sum_i \lvert \Int W_\ell \pitchfork f_i \rvert + e(W_\ell) \mod{2},\]
    and also
     $\mu_\Sigma(\partial \ol{W}_\ell) = \mu_\Sigma(\partial W_\ell)$, and $\lvert \partial \ol{W}_\ell \pitchfork \partial \ol{W}_k \rvert = \lvert \partial W_\ell \pitchfork \partial W_k \rvert$, for each $k\neq \ell$.

   Next, we will remove intersections between Whitney arcs as well as self-intersections of Whitney arcs, at the expense of adding intersections between $F$ and the Whitney discs. We will use the procedure described in~\cite{Freedman-book-basicgeo}*{Section~15.2.3}. For an intersection between $\partial \ol{W}_\ell$ and $\partial \ol{W}_k$, where possibly $k=\ell$, the procedure pushes the intersection off one of the endpoints of one of the Whitney arcs of $\partial \ol{W}_\ell$, i.e.\ a double point of~$F$, moving a neighbourhood of $\partial \ol{W}_k$ and creating an intersection between~$\ol{W}_k$ and~$F$. This new intersection point is created in a small neighbourhood of the double point of $F$ chosen for the pushing off procedure. If several Whitney arcs intersect the given Whitney arc of $\partial \ol{W}_\ell$,  push off in order of proximity to the endpoint. This avoids extraneous intersections between Whitney arcs being created. Perform this pushing off procedure on both arcs of $\partial \ol{W}_\ell$.  For each of the two arcs in $\partial \ol{W}_\ell$, push towards one of the two double points of $F$ paired by $\ol{W}_\ell$; choose these double points so that we use one double point for each arc. This ensures that the new intersections between Whitney discs and $F$ arise in disjoint neighbourhoods in the ambient manifold.

Apply the move described in the previous paragraph to the Whitney arcs of $\{\ol{W}_\ell\}$ in order, beginning with $\ell=n$. In other words, in the $i$th step, we push off the intersections of $\partial \ol{W}_k$ with $\partial \ol{W}_{n-i+1}$, for $k\leq n-i+1$.
After the $n$th step, we produce a convenient collection $\W':=\{W'_\ell\}$, where each $W'_\ell$ is the result of applying the above procedure to $\ol{W}_\ell$.
    This yields, for each $\ell$,
    \begin{align*}
        \sum_i \lvert \Int W'_\ell \pitchfork f_i\rvert &\equiv
        \sum_i \lvert \Int \ol{W}_\ell \pitchfork f_i \rvert +  \mu_\Sigma(\partial \ol{W}_\ell) + \sum_{k>\ell} \lvert \partial \ol{W}_\ell \pitchfork \partial \ol{W}_k \rvert \mod{2}\\
        &\equiv  \sum_i \lvert \Int W_\ell \pitchfork f_i\rvert + e(W_\ell) +  \mu_\Sigma(\partial W_\ell) + \sum_{k > \ell} \lvert\partial W_\ell \pitchfork \partial W_k\rvert \mod{2}.
    \end{align*}
In the above expression, the term $\sum_{k>\ell} \lvert \partial \ol{W}_\ell \pitchfork \partial \ol{W}_k \rvert$ arises since the arcs in $\partial \ol{W}_\ell$ are moved, to create a new intersection point of $\ol{W}_\ell$ with $F$, precisely once for each intersection of $\partial \ol{W}_\ell$ with $\bigcup_{k >\ell} \partial \ol{W}_k$.

Sum over $\ell$ to obtain that
$t(F,\W') = t_{\Alt}(F,\W)\in \Z/2$ as claimed.
\end{proof}

\begin{proof}[Proof of~\cref{lem:fiber-move}]
By~\cref{lem:alt-enough}, it will suffice to show that in case \eqref{item-lem-fibre-move-i},
\[t_{\Alt}(F',\W') = t_{\Alt}(F,\W) + \Theta_A(B) \in \Z/2,\] and in case~\eqref{item-lem-fibre-move-ii},
	\[t_{\Alt}(F',\W') = t_{\Alt}(F,\W) + \Theta_A(B,D) \in \Z/2.\]
The proof splits into three cases.

\setcounter{case}{0}
\begin{case}\label{case:easy-case}
The band $B$ is an annulus as in \eqref{item-lem-fibre-move-i}.
\end{case}

Recall that
\[
	\Theta_A(B):= \mu_{\Sigma}(\partial B)+ \lvert \partial B\pitchfork A \rvert +\lvert \Int B\pitchfork F \rvert + e(B) \mod 2.
	\]
By the construction of $F'$ and $\mathcal{W}'$, we have
\[t_{\Alt}(F',\W') \equiv t(F,\mathcal{W}) +\mu_\Sigma(\partial W_B) + \lvert \partial W_B \pitchfork A\rvert +  \lvert \Int W_B \pitchfork F \rvert + e(W_B)  \mod{2}.
\]
Every self-intersection of $\partial B$ and each intersection $\partial B \pitchfork A$ will contribute one intersection of $\partial W_B$ and between $\partial W_B$ and $A$ respectively, while each intersection $\Int B\pitchfork F$ will contribute one intersection between $\Int W_B$ and $F$. Thus it remains to show that the framing $e(B)$ appearing in the definition of $\Theta_A(B)$ agrees with the framing $e(W_B)$. For this it will be helpful to pick the finger arc and the $2$-dimensional sub-bundle for its normal bundle needed for the finger move more carefully, which we do next.

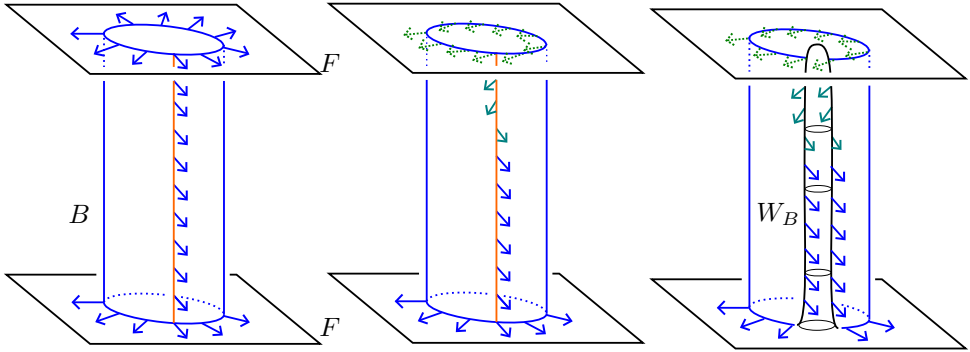
\begin{figure}[htb]
    \centering
    \begin{tikzpicture}
        \node[anchor=south west,inner sep=0] at (0,0){	    \includegraphics{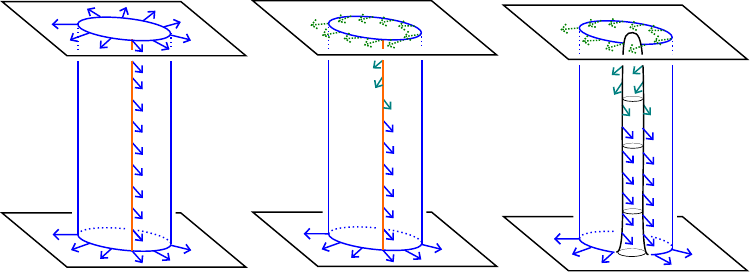}};
		\node at (4.3,0.3) {$F$};
		\node at (1,1.8) {$B$};
		\node at (4.3, 3.8) {$F$};
		\node at (10.2, 1.8) {$W_B$};
	\end{tikzpicture}
    \caption{Left: we see a model annulus $B$ (blue) connecting two sheets of $F$ (black), and a finger arc $D=\{\pt\} \times D^1 \subseteq S^1 \times D^1 \cong B$ (orange).  We also see the surface framing on $\partial B$ and the section $s$ along the finger arc of the normal bundle of $B$ in $M$. Recall the section is defined over all of $B$, but we only show it on a subset. Middle: in the top half of $B$, we rotate the section so that it lies in $(\nu_{\Sigma}^{M}|_{\partial_i B} \cap \nu_B^M|_{\partial_i B})$, i.e.\ in the time direction, on the top boundary component (dotted green). The modified section is called $s'$. Right: after performing the finger move, $s'$ gives the Whitney framing for the new Whitney disc $W_B$.}
\label{fig:suggestive}
\end{figure}

Let $\partial_i B$ denote the connected components of $\partial B\subseteq \Sigma$. Consider the following decomposition of the tangent bundle of $M$ restricted to $\partial_i B$:
\begin{equation}\label{eq:tangent-decomposition}
TM|_{\partial_i B} \cong T(\partial_i B) \oplus \nu_{\partial_i B}^{\Sigma} \oplus \nu_{\partial_i B}^B \oplus (\nu_{\Sigma}^{M}|_{\partial_i B} \cap \nu_B^M|_{\partial_i B}).
\end{equation}
As shown in \cref{fig:suggestive}, choose a section $s$ of the normal bundle of $B$ that is nonvanishing on the finger arc.  In both boundary components $\partial_i B$, we assume that this section lies in $\nu_{\partial_i B}^{\Sigma}$. Now rotate the section near the top boundary component, as shown in the middle picture of \cref{fig:suggestive}, to obtain a section $s'$, so that on the top component $s'$ lies in $(\nu_{\Sigma}^{M}|_{\partial_i B} \cap \nu_B^M|_{\partial_i B})$.  For the finger move, by definition, we use the 2-dimensional sub-bundle of $\nu^M_D$ determined by $s'$ and $T(\partial_i B)$, as shown in the right-most figure of \cref{fig:suggestive}.

Now consider the Whitney disc $W_B$ obtained from $B$ after performing the finger move along $D$ using the above $2$-dimensional sub-bundle of its normal bundle. By definition, $e(W_B)$ equals the number of zeros of $s'|_{W_B}$, since on the boundary Whitney arcs it is normal to one sheet and tangent to the other.
On the other hand, the number of zeros of $s'|_{W_B}$ equals the number of zeros of $s$, since $s'$ was obtained from $s$ by a rotation, and neither section vanishes near the finger arc.
Finally by definition $e(B)$ counts the zeros of $s$. Therefore we see that $e(B)=e(W_B)$.

\begin{case}
The band $B$ is an annulus such that $w_1(\Sigma)$ restricted to both components of $\partial B$ is nontrivial, as in \eqref{item-lem-fibre-move-ii}.
\end{case}

Assume that a finger arc $D$ has been chosen. To define $\Theta_A(B,D)$, we pick a section $s$ as in \cref{def:thetaA-nonorientable} on $\wh{B}$, which is by definition the result of cutting $B$ open along $D$.
Recall that
\[
\Theta_A(B,D):= \mu_{\Sigma}(\partial B)+ \lvert \partial B\pitchfork A \rvert + \lvert \Int B\pitchfork F \rvert +e(\wh B) \mod 2.\]
As in \cref{case:easy-case}, we need only check that the term $e(\wh{B})$ in $\Theta_A(B,D)$ agrees with the term $e(W_B)$ in $t_{\Alt}$. Again as in \cref{case:easy-case}, we rotate the section near the top boundary component to obtain a section $s'$, so that on the top component $s'$ lies in $(\nu_{\Sigma}^{M}|_{\partial_i B} \cap \nu_B^M|_{\partial_i B})$. For the finger move, we use the 2-dimensional sub-bundle determined by $s'$ and $T(\partial_i B)$. Note that, just like $s$, the section $s'$ is not defined on all of $B$, but only on $\wh{B}$; see \cref{fig:non-orientable-BFF}.
Nevertheless, on points that map to the same point in $D$, the section $s'$ agrees up to a sign and thus still determines a 1-dimensional sub-bundle. The section $s'$ restricts to a section of the normal bundle of the Whitney disc $W_B$ obtained from $B$. As in \cref{case:easy-case}, the quantities $e(B)$ and $e(W_B)$ coincide.

\begin{figure}[htb]
	\centering
  \begin{tikzpicture}	
        \node[anchor=south west,inner sep=0] at (0,0){\includegraphics{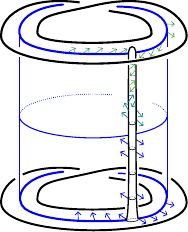}};
        \node[anchor=south west,inner sep=0] at (-5,0){\includegraphics{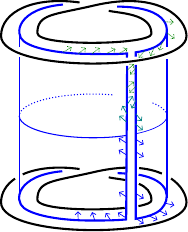}};
		\node at (-0.23,0.5) {$F'$};
		\node at (-0.23,3.35) {$F'$};
		\node at (-0.1, 1.8) {$W_B$};
		\node at (-5.23,0.5) {$F$};
		\node at (-5.23,3.35) {$F$};
		\node at (-5.1, 1.8) {$\wh{B}$};
	\end{tikzpicture}
	\caption{Left: the section $s'$ on $D$ (orange), and on a parallel copy of $D$; cf.\ \cref{fig:non-orientable-section}\,(b). Close to the top boundary the section extends into the time direction (teal and dotted green).  Right: the section $s'$ on the boundary of the new Whitney disc $W_B$.
	}
	\label{fig:non-orientable-BFF}
\end{figure}

\begin{case}
The band $B$ is a M\"obius band with $w_1(\Sigma)$ restricted to $\partial B$ trivial, as in \eqref{item-lem-fibre-move-i}.
\end{case}

As in \cref{case:easy-case}, we only need to show that the term $e(B)$ in $\Theta_A(B)$ agrees with the term $e(W_B)$ in $t_{\Alt}$.
Let $D$ denote a properly embedded arc on $B$ along which we wish to perform the finger move. Identify $B$ with the quotient of the square $S:= [-1,1]\times [-1,1]$ as usual, i.e.\ $(-1,x) \sim (1,-x)$ for all $x \in [-1,1]$, with $D$ corresponding to the arc $\{-1\}\times [-1,1]\equiv\{1\}\times [-1,1]$ (see \cref{fig:necklace}). Pull back the normal bundle $\nu_B^M$ of $B$ in $M$ to $S$ via the quotient map $\pi\colon S\to B$ and then pick a trivialisation $\pi^*(\nu_B^M)\cong S\times \R^2$ so that on the horizontal boundary $H:=[-1,1]\times \{-1,1\}$ we have that $\pi^*\nu^\Sigma_{\partial B}$ coincides with $H\times \R \times \{0\}\subseteq S\times \R\times \R$.

	\begin{figure}[htb]
    \centering
    \begin{tikzpicture}
        \node[anchor=south west,inner sep=0] at (0,0){\includegraphics{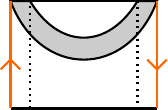}};
        \node[anchor=south west,inner sep=0] at (4,0){\includegraphics{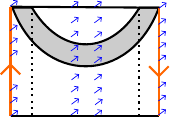}};
        \node[anchor=south west,inner sep=0] at (8,0){\includegraphics{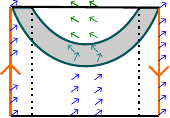}};
		\node at (1.45,-0.5) {(a)};
		\node at (5.45,-0.5) {(b)};
		\node at (9.45,-0.5) {(c)};
		\node at (-0.1,1) {$D$};
		\node at (3.1,1) {$D$};
		\node at (1.45,1.05) {$N$};
		\node at (1.5,0.25) {$S_-$};
		\node at (1.5,1.6) {$S_+$};		
	\end{tikzpicture}
    \caption{(a) The necklace region $N\subseteq S$ is shown in grey. The dotted black lines indicate the boundary of the region where the finger move occurs. The solid black lines indicate where the band is attached to the surface $F$. The arc $D$ is shown in orange. (b) The section $s$ is shown in blue. (c) The section $s'$ is indicated. Note that on $S_-$, the sections $s$ and $s'$ agree. On $S_+$, the section $s'$ (green) is obtained by rotating $s$ by 90 degrees. In the necklace region, the section rotates continuously (teal), interpolating between the values on $S_+$ and $S_-$. Note that the section on the arc $D$ has not changed, but it has been modified on part of the dotted lines.}\label{fig:necklace}
\end{figure}
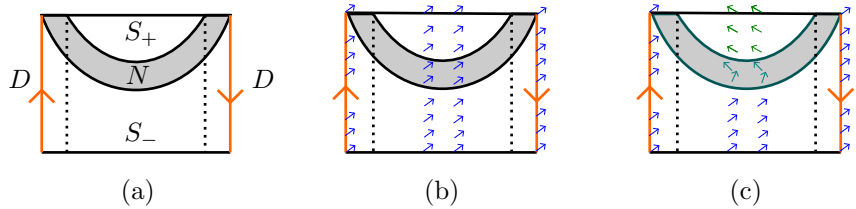

	Then we pick a section $s$ of $\nu_B^M$ such that $s\vert_{\partial B}$ lies in $\nu^\Sigma_{\partial B}$. Note that $s\vert_{\partial B}$ is nowhere vanishing but the section $s$ might have zeros.
Without loss of generality we assume that any zeros of $s$ do not lie in the strip $\left([-1,-1+\varepsilon]\cup [1-\varepsilon,1]\right) \times [-1,1]$ for some $\varepsilon \in (0,1/4)$.
	
	Next we modify the section $s$.  Choose a ``necklace'' region, i.e.\ a sub-square $N$ with two opposite edges coinciding with $[-1,-1+\varepsilon]\times \{1\}$ and $[1-\varepsilon,1]\times \{1\}$, and otherwise lying in the interior of $S:=[-1,1]^2$. We consider the pullback of $s$ to $S$, where we have a trivialisation. Modify this pullback so that it remains unchanged in the lower component $S_-$ of $S\sm N$, and is rotated by 90 degrees on the upper component $S_+$. On the region $N$, the section rotates continuously, interpolating between the values on $S_+$ and $S_-$. Push this forward to get a modified section $s'$ on $B$. Since the modification is produced by a continuous rotation, the number of zeros of this modification agrees with the number of zeros of the original $s$.
	
Recall that we wish to perform a finger move guided by the arc $D=\{-1\}\times [-1,1]$. Without loss of generality, we assume that the `width' of the finger move is $2\varepsilon$. More precisely, to perform a finger move we need a $2$-dimensional sub-bundle of $\nu_D^M$. We require that this contains $\nu_D^B$ to ensure that the finger move cuts $B$ open into a Whitney disc as desired. Fix an identification of the total space of the normal bundle  $\nu_D^M$ with $D\times \R^3$. We choose an embedding $\iota \colon \nu_{D}^M \hookrightarrow M$ restricting to the inclusion of $D$ on $D \times \{0\}$,  with the following properties.
\begin{enumerate}[(i)]
  \item We assume that the first $\R^1$ factor of $D \times \R^3$ corresponds to $\nu_D^B$, and that $\iota$ identifies $D\times \{\pm 1\}\times \{0\}\times \{0\}$ with the arcs $\{-1+\varepsilon, 1-\varepsilon\}\times [-1,1]$. This is what was meant by the width of the finger move.
  \item  We also require that $\iota$ identifies $D\times \{t\}\times \{1\}\times \{0\}$  with $s'(\iota(D \times \{t\} \times \{0\} \times \{0\}))$ for $t\geq 0$, and with $s'(\iota(D \times \{t\} \times \{0\} \times \{0\}))$ rotated by $90$ degrees for $t\leq -1$. (Here we also implicitly  identify $\nu_B^M$ with its image in $M$.)
\end{enumerate}

	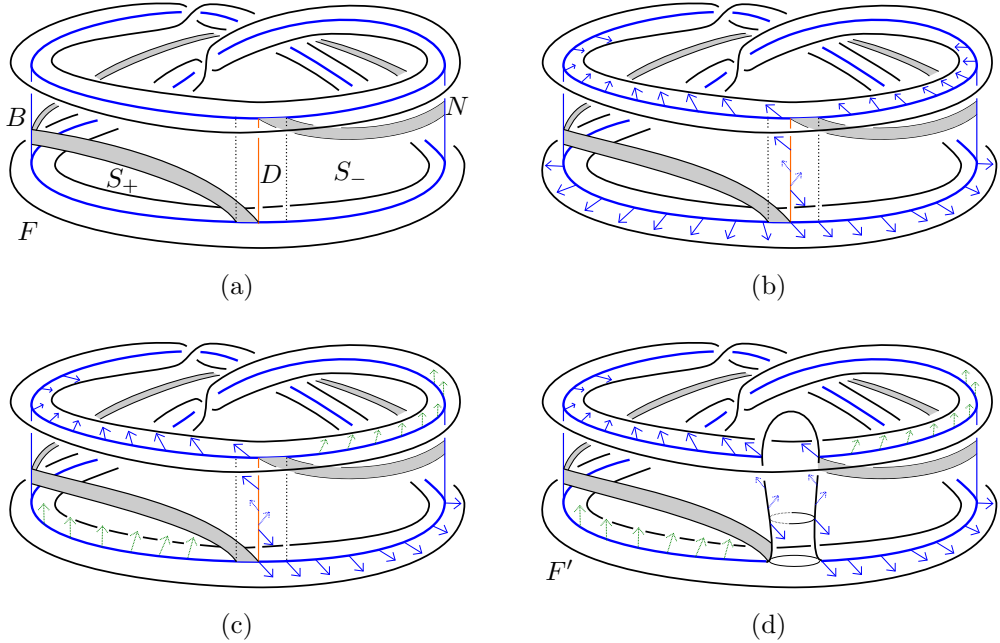
\begin{figure}[htb]
    \centering
    \begin{tikzpicture}
	    \node[anchor=south west,inner sep=0] at (0,4.5){\includegraphics{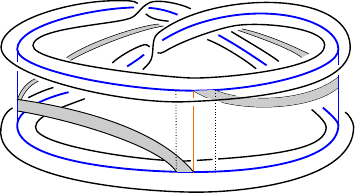}};
        \node[anchor=south west,inner sep=0] at (7,4.5){\includegraphics{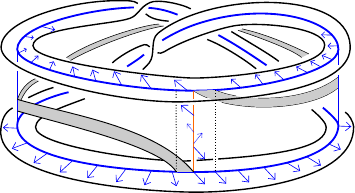}};
        \node[anchor=south west,inner sep=0] at (0,0){\includegraphics{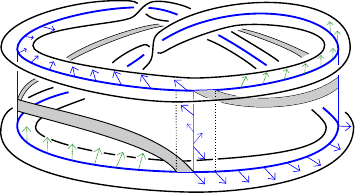}};
        \node[anchor=south west,inner sep=0] at (7,0){\includegraphics{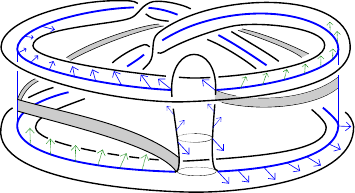}};
		\node at (3,4) {(a)};
		\node at (10,4) {(b)};
		\node at (3,-0.5) {(c)};
		\node at (10,-0.5) {(d)};						
		\node at (3.45,5.5) {$D$};
		\node at (4.5,5.5) {$S_-$};
		\node at (0.25,4.75) {$F$};
		\node at (7.25,0.25) {$F'$};
		\node at (1.5,5.4) {$S_+$};
		\node at (5.9,6.4) {$N$};
		\node at (0.1,6.25) {$B$};		
	\end{tikzpicture}
    \caption{(a) The surface $F$ is shown in black, and the M\"obius band $B$ in blue. Note this picture is entirely in $\R^3$. The necklace region $N$ is in grey, and splits $B$ into two components $S_+$ and $S_-$. The finger arc $D$ is in orange, and the width of the finger move is shown with dotted lines. (b) We show the section $s$ in blue. Note that while there is a rotation along $D$ there are no zeros of $s$ in the strip between the dotted lines. (c) The modified section $s'$. Note the section coincides with $s$ on $S_-$ and has been rotated (green) on $S_+$. (d) The section $s'$ on the Whitney disc $W_B$ formed after the band fibre finger move. By the construction of the 2-dimensional sub-bundle of the normal bundle of $D$ used to guide the finger move, the section $s'$ is tangent to $F'$ along the right edge of the finger (corresponding to the right dotted line in (c), where the finger contains part of a Whitney arc of $W_B$), and $s'$ is normal to $F'$  along the left edge of the finger.}\label{fig:necklace-on-display}
\end{figure}

Now do the finger move using $D\times S^1\times \{0\}$ according to this parametrisation, where $S^1$ is the unit circle in the $\R^2$ factor, along with a finger tip. The choices above imply that $s'|_{\partial W_B}$ is a Whitney framing, where $W_B$ denotes the new Whitney disc created according to \cref{construction:finger-move}. Specifically, let $F'$ denote the result of the finger move. By our choice of the $2$-dimensional sub-bundle for the finger move above, the section $s'$ is normal to $F'$ along half of $\partial W_B$, and tangent along the other half; see~\cref{fig:necklace-on-display}.

As previously mentioned, we need to check that the relative Euler number $e(B)$ in $\Theta_A(B)$ agrees with the twisting number $e(W_B)$ in $t_{\Alt}$.
The relative Euler number $e(B)$ is given by the number of zeros of the section $s$ on the interior of $B$. As mentioned before, this coincides with the number of zeros of the section $s'$.
 Since $s'|_{\partial W_B}$ is a Whitney framing, this further coincides with the twisting number $e(W_B)$ as desired, since we assumed there are no zeros of $s$ within the strip $\left([-1,-1+\varepsilon] \cup [1-\varepsilon,1]\right) \times [-1,1]\subseteq B$ used for the finger move.
\end{proof}

Next we prove \cref{lem:dependence_on_A}. Here is the statement for the convenience of the reader.

\begin{replemma}{lem:dependence_on_A}
Let $F\colon (\Sigma,\partial \Sigma)\looparrowright (M, \partial M)$ be as in \cref{convention}, with $\mu (F)=0$. If the $\Z/2$-valued intersection form $\lambda_{\Sigma}$ on $H_1(\Sigma;\Z/2)$ is nontrivial on $\partial\bands(F)$, then we can change $F$ by a regular homotopy to $F'$ such that there are convenient collections of Whitney discs $\W$ and $\W'$ for the double points of $F$ and $F'$ respectively, such that $t(F,\W)\neq t(F',\W')$.

Moreover, if $F$ has dual spheres and the $\Z/2$-valued intersection form $\lambda_{\Sigma^\twist}$ on $H_1(\Sigma^\twist;\Z/2)$ is nontrivial on $\partial\bands(F^\twist)$ then $\km(F)=0$.
\end{replemma}

\begin{proof}
We first prove the statement (without using dual spheres) about $t(F,\W)$ depending on the choice of $\W$ under our assumption.
By hypothesis $F$ is a generic immersion whose double points can be paired by a convenient collection $\W=\{W_\ell\}$ of Whitney discs (\cref{cor:vanishing-pairing}).
By hypothesis, $\lambda_{\Sigma }$ is nontrivial on $\partial\bands(F)$, meaning that there are bands $B_1$ and $B_2$ with boundary on $F(\Sigma )$ minus double points such that $\lambda_{\Sigma }(\partial B_1,\partial B_2)\neq 0\in \Z/2$. Here it is possible that $B_2$ is a parallel push-off of $B_1$. Using $B_i$ and \cref{construction:finger-move}, perform a finger move and obtain a new framed Whitney disc, calling the resulting convenient collection of Whitney discs $\W_i$, for $i=1,2$, and the resulting map $F_i$

If $t(F_i,\W_i)\neq t(F,\W)$ for some $i=1,2$, we can set $F'=F_i$ and $\W'=\W_i$. Otherwise, use \cref{lem:fiber-move} twice, for $B_1$ and $B_2$ simultaneously, and let $\W'$ denote the resulting convenient collection of Whitney discs for the resulting map $F'$. Then the change in $t(F,\W )$ is as before except that there is an additional contribution from the odd number of intersections between the boundary arcs for the new Whitney discs coming from $B_1$ and $B_2$. Specifically, removing these by pushing one Whitney arc off the end of the other (as part of \cref{construction:finger-move}) introduces an odd number of intersections between the Whitney discs and $F$. Therefore, $t(F',\W')\neq t(F,\W)$, as needed.

For the second statement, apply the above argument to the sub-collection $\W^\twist$ of $\W$ pairing the intersections within $F^\twist$. It follows that we may find $\W^\twist$ such that $t(F^\twist,\W^\twist)=0$, possibly after a regular homotopy of $F^\twist$.
 Then $\km(F)=0$ follows from \cref{lem:tau}.
\end{proof}

For the proof of \cref{lem:well-defined}, we will need the next four \cref{lem:representable,lem:quadratic,lem:embedded_boundary,lem:cutting}.

\begin{lemma}
	\label{lem:representable}
Let $F\colon (\Sigma,\partial \Sigma)\looparrowright (M, \partial M)$ be as in \cref{convention}. Every element of $H_2(M,\Sigma;\Z/2)$ can be represented by an immersion of some compact surface into $M$, with interior transverse to $F$, and with boundary generically immersed in $F(\Sigma)$ away from the double points.
\end{lemma}
\begin{proof}
Let $\mathcal{N}_k(M,\Sigma)$ denote the $k$-dimensional unoriented bordism group over $(M,\Sigma)$, and let $\mathcal{N}_k$ denote the $k$-dimensional unoriented bordism group over a point.
    Using topological transversality, it suffices to show that every element of $H_2(M,\Sigma;\Z/2)$ can be represented by a map $(S,\partial S)\to (M,\Sigma)$ for some surface $S$. To show this, it suffices to see that the edge homomorphism $\mathcal{N}_2(M,\Sigma)\to H_2(M,\Sigma;\Z/2)$ from the Atiyah--Hirzebruch spectral sequence is onto.

    Recall that the $\mathcal{N}_0$ is isomorphic to $\Z/2$ while the $\mathcal{N}_1$ vanishes. It follows that in the Atiyah--Hirzebruch spectral sequence with $E_2$-term $H_p(M,\Sigma ;\mathcal{N}_q)$ and  converging to $\mathcal{N}_{p+q}(M,\Sigma)$, there is no nontrivial differential going out of $H_2(M,\Sigma;\mathcal{N}_0)\cong H_2(M,\Sigma;\Z/2)$; such a differential would have codomain $H_0(M,\Sigma;\mathcal{N}_1)=0$. Thus the edge homomorphism $\mathcal{N}_2(M,\Sigma)\to H_2(M,\Sigma;\Z/2)$ is onto, as desired.
\end{proof}

\begin{lemma}
	\label{lem:quadratic}
Let $F\colon (\Sigma,\partial \Sigma)\looparrowright (M, \partial M)$ be as in \cref{convention}. Suppose that $\mu (F)=0$ and let $A$ be a choice of Whitney arcs pairing the double points of $F$.
Then the function $\Theta_A$ is quadratic with respect to the $\Z/2$-valued intersection form $\lambda_\Sigma$. That is, let $S$ and $S'$ be compact surfaces, with generic immersions of pairs $(S,\partial S) \imra (M,\Sigma)$ and $(S',\partial S') \imra (M,\Sigma)$ such that $\partial S$ and $\partial S'$ intersect $A$ and each other transversely, and are such that $w_1(\Sigma)$ is trivial on every component of $\partial S$ and $\partial S'$. Then we have
	\[
	\Theta_A(S\cup S')=\Theta_A(S)+\Theta_A(S')+\lambda_\Sigma(\partial S,\partial S').
	\]
\end{lemma}

\begin{proof}
	The term $e(S)$ in \cref{def:thetaA} is defined component-wise and the terms $\lvert \partial S\pitchfork A \rvert$ and $\lvert \Int S\pitchfork F \rvert$ are linear in $\partial S$ and $S$, respectively. Hence the only term that is not linear in $S$ is $\mu_F(\partial S)$. This term is also quadratic in the sense that
	\[
	\mu_\Sigma(\partial S\cup \partial S')=\mu_\Sigma(\partial S)+\mu_\Sigma(\partial S')+\lambda_\Sigma(\partial S,\partial S'),
	\]
	which proves the lemma.
\end{proof}

\begin{lemma}
	\label{lem:embedded_boundary}
Let $F\colon (\Sigma,\partial \Sigma)\looparrowright (M, \partial M)$ be as in \cref{convention}. Suppose that $\mu (F)=0$ and let $A$ be a choice of Whitney arcs pairing the double points of $F$. Let $S$ be a compact surface, with a generic immersion of pairs $(S,\partial S) \imra (M,\Sigma)$, such that $\partial S$ is transverse to $A$ and  $w_1(\Sigma)$ is trivial on every component of $\partial S$. Then there is another such surface $S'$ with a generic immersion of pairs $(S',\partial S') \imra (M,\Sigma)$ with $\partial S'$ transverse to $A$ such that:
	\begin{enumerate}
		\item $[S]=[S']\in H_2(M,\Sigma;\Z/2)$;
		\item $\Theta_A(S)=\Theta_A(S')$; and
		\item $\partial S'$ is embedded in $\Sigma$.
	\end{enumerate}
\end{lemma}

\begin{proof}
To start, pick a section $\gamma_S$ of the normal bundle $\nu_S^M$ which on $\partial S$ is nowhere vanishing and lies in $\nu_{\partial S}^F$ as in the definition of $\Theta_A(S)$.

\begin{figure}[htb]
    \centering
   \begin{tikzpicture}
        \node[anchor=south west,inner sep=0] at (0,0){	               \includegraphics{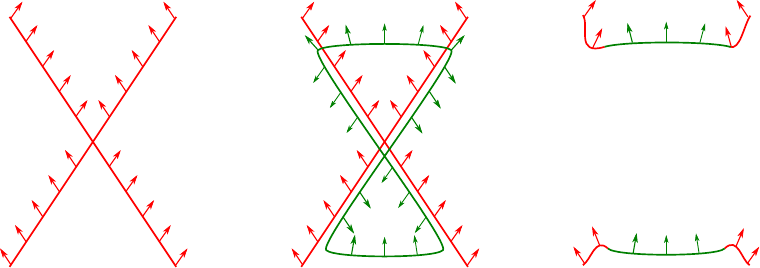}};
		\node at (0.5,0) {$\partial S$};
		\node at (5.3,0) {$\partial S$};
		\node at (6.4,-0.1) {$\partial D$};
		\node at (11.2,3.5) {$\partial S'$};
	\end{tikzpicture}
    \caption{Adding a disc $D$ to $S$ to remove a self-intersection of $\partial S$. Left: The neighbourhood of a self-intersection of $\partial S$ before adding the disc $D$. The section $\gamma_S$ is shown along $\partial S$. Middle: the boundary of the disc~$D$. The section $\gamma_D$ is shown along $\partial D$. Right: after the  modification cf.\ \cref{fig:pushoff}.}
\label{fig:remove-intersection}
\end{figure}

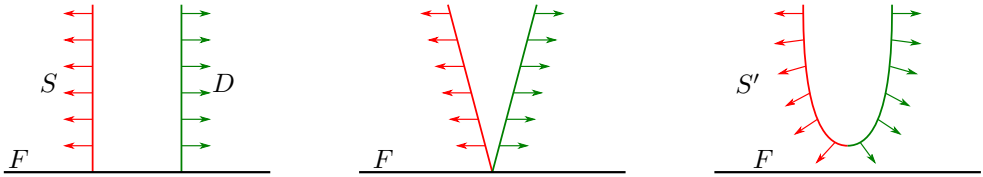
\begin{figure}[htb]
    \centering
  \begin{tikzpicture}
        \node[anchor=south west,inner sep=0] at (0,0){	            \includegraphics{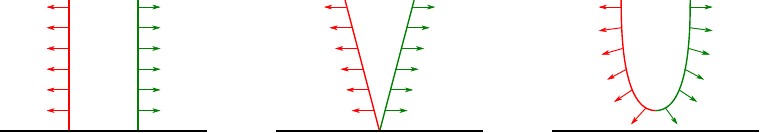}};
		\node at (0.2,0.2) {$F$};
		\node at (5,0.2) {$F$};
		\node at (10,0.2) {$F$};
		\node at (0.6,1.2) {$S$};
		\node at (2.9,1.2) {$D$};
		\node at (9.8,1.2) {$S'$};
	\end{tikzpicture}
    \caption{Glue $D$ to $S$ along the aligned parts of the boundaries and push this part of the boundary off $F$. }
\label{fig:pushoff}
\end{figure}

The idea of the proof is to remove all intersections of $\partial S$ by locally adding a twisted disc $D$ as indicated in \cref{fig:remove-intersection}. More precisely, we add these discs $D$ such that the interiors are disjoint from the interior of $F$ and the boundary is disjoint from $A$. Then pick a section $\gamma_D$ of $\nu_D^M$ such that, along the aligned (i.e.\ parallel) parts of the boundaries, $\gamma_D$ and $\gamma_S$ are opposite. Glue $D$ to $S$ along the aligned parts of the boundaries and push this part of the boundary off $F$ as indicated in \cref{fig:pushoff}. Each of these local twisted discs has mod 2 Euler number $1$, as can be seen from the nontrivial linking in \cref{fig:twisted-band}. Thus the resulting surface $S'$ has embedded boundary and the mod 2 Euler number of $S'$ differs from that of $S$ by the number of intersections of $\partial S$ modulo two, i.e.\  $\mu_\Sigma(\partial S)$. Since we have neither changed the number of intersections of the interior with $F$ nor the number of intersections of the boundary with $A$, we have $\Theta_A(S)=\Theta_A(S')\in \Z/2$. As the local discs are trivial in $H_2(M,\Sigma;\Z/2)$, we furthermore have $[S]=[S']\in H_2(M,\Sigma;\Z/2)$ as needed.
\end{proof}

\begin{figure}[htb]
    \centering
     \begin{tikzpicture}
        \node[anchor=south west,inner sep=0] at (0,0){\includegraphics[width=4cm]{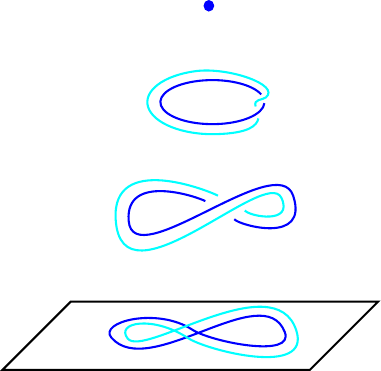}};
		\node at (3.4,0.5) {$F$};
	\end{tikzpicture}
    \caption{A twisted band with Euler number $+1$ in a movie description. Bottom: an immersed figure-eight curve (blue) is shown lying on the immersed surface $F(\Sigma)$ (black) away from the double points. A framing on the normal bundle on the boundary of the band is shown in light blue. Moving upward/forward in time, we see a simple closed curve shrinking to a point. The push-off corresponding to the framing induced by $\Sigma$ is shown in light blue. For the twisted band with Euler number $-1$, we use the other resolution.}
\label{fig:twisted-band}
\end{figure}

\begin{lemma}
	\label{lem:cutting}
Let $F\colon (\Sigma,\partial \Sigma)\looparrowright (M, \partial M)$ be as in \cref{convention}. 	Let $Z$ be a disjoint union of embedded circles in $\Sigma$.  Let $\Sigma \mid Z$ denote $\Sigma$ cut along $Z$, i.e.\ the completion of $\Sigma \sm Z$ to a compact manifold with boundary.
Let $\cF=\{\Sigma_i\}$ be the connected components of  $\Sigma \mid Z$ and suppose that $[Z]= 0 \in H_1(\Sigma;\Z/2)$.
	Then we can pick a subset $\cF'\subseteq \cF$ such that each component of $Z$ appears exactly once as a connected component of the boundary of precisely one $\Sigma_i\in\cF'$.
\end{lemma}

\begin{proof}
Without loss of generality, assume that $\Sigma$ is connected.
	Considering the entire collection $\cF=\{\Sigma_i\}$, every component of $Z$ would appear as the boundary of precisely two of the $\Sigma_i$, since otherwise $Z$ would be nontrivial in $H_1(\Sigma;\Z/2)$.  To see this note that $Z$ can contain homologically essential curves in $H_1(\Sigma;\Z/2)$, provided they cancel. However none of these can be orientation-reversing curves, since $Z$ is embedded.
	
	The idea of the proof is to take ``half'' of the components of $\cF$. Let $x \in \Sigma\mid  Z$ be an arbitrary basepoint away from $Z$.  For each $\Sigma_i$, define $p(\Sigma_i) \in \Z/2$ as follows. Pick a point $y\in \Int \Sigma_i$ and a path $w$ in $\Sigma$ from $x$ to $y$ which is transverse to $Z$. Define $p(\Sigma_i)$ as the mod 2 intersection number of~$w$ and~$Z$.
	
	We show that $p(\Sigma_i)$ is independent of the choices of $w$ and $y$.
	If $w'$ is another path from $x$ to $y$, then the concatenation $w^{-1}\cdot w'$ is a loop in $\Sigma$ and we have
	\[\lvert (w^{-1}\cdot w')\pitchfork Z\rvert
	=\lambda_\Sigma([w^{-1}\cdot w'],[Z])
	=\lambda_\Sigma([w^{-1}\cdot w'],0)=0.\]
	So $p(\Sigma_i)$ does not depend on the choice of $w$.
	Also, since each $\Sigma_i$ is connected, $p(\Sigma_i)$ does not depend on~$y$. To see this let $y'\in \Int\Sigma_i$, and choose a path $z$ from $y$ to $y'$ that lies in $\Int{\Sigma_i}$. Let $w'$ be a path from $x$ to $y'$, which is further transverse to $Z$. Then \[\lvert w \pitchfork Z \rvert = \lvert (w\cdot z) \pitchfork Z \rvert = \lvert w' \pitchfork Z \rvert.\]
	The first equation uses that $z \subseteq \Sigma_i$ and the second uses independence of the choice of $w$.
	Hence $p(\Sigma_i) \in \Z/2$ is well defined as desired.
	
	Now let $\cF'$ consist of all the components $\Sigma_i$ for which  $p(\Sigma_i)=0$. This subset is the one we seek, since for a fixed component $Z_j$ of $Z$, the two components of $\cF$ containing a cut-open copy of $Z_j$ have different values of $p$. This completes the proof of the lemma.
\end{proof}

We are now ready for the proof of \cref{lem:well-defined}.

\begin{replemma}{lem:well-defined}
Let $F\colon (\Sigma,\partial \Sigma)\looparrowright (M, \partial M)$ be as in \cref{convention}, with $\mu (F)=0$. Let $A$ be a choice of Whitney arcs pairing the double points of $F$.
	\begin{enumerate}[(i)]
	    \item\label{repitem-i-lem-welldefined} 	
Let $S$ be a compact surface, with a generic immersion of pairs $(S,\partial S) \imra (M,\Sigma)$, such that $\partial S$ is transverse to $A$ and $w_1(\Sigma)$ is trivial on every component of $\partial S$.
Then $\Theta_A(S) \in \Z/2$ depends only on the homology class of $S$ in $H_2(M,\Sigma;\Z/2)$.
	\item\label{repitem-ii-lem-welldefined} 	Let $B$ be an annulus, with a generic immersion of pairs $(B,\partial B) \imra (M,\Sigma)$, such that $\partial B$ is transverse to $A$ and $w_1(\Sigma)$ is nontrivial on both components of $\partial B$.  Pick an embedded arc $D$ in $B$ connecting the components of $\partial B$ and disjoint from all double points. Then $\Theta_A(B,D) \in \Z/2$ depends only on the homology class of $B$ in $H_2(M,\Sigma;\Z/2)$.  In particular, $\Theta_A(B,D)$ does not depend on $D$, so we write~$\Theta_A(B)$.
	\item\label{repitem-iii-lem-welldefined} Let $S$ be a surface as in \eqref{item-i-lem-welldefined} and let $B$ be an annulus as in \eqref{item-ii-lem-welldefined} such that $[S]=[B] \in H_2(M,\Sigma;\Z/2)$. Then $\Theta_A(S) = \Theta_A(B) \in \Z/2$.
	\item\label{repitem-iv-lem-welldefined} 	If  $\lambda_{\Sigma}|_{\partial \bands(F)}=0$, the restriction of $\Theta_A$
	to $\bands(F)$ is independent of the choice of $A$, giving a well defined map $\Theta\colon \bands(F)\to \Z/2$.
	\end{enumerate}
\end{replemma}

\begin{proof}
	To prove~\eqref{repitem-i-lem-welldefined}, assume that $S$ and $S'$ are immersed compact surfaces, with $w_1(\Sigma)$ trivial on each of the connected components of the boundaries, representing the same element in $H_2(M,\Sigma;\Z/2)$. Modulo isotopy we can assume that $S$ and $S'$ intersect transversely in their interiors in $M$, and their boundaries intersect transversely on $F$. In particular, their boundaries $\partial S$ and $\partial S'$ intersect in an even number of points. Hence $\Theta_A(S\cup S')=\Theta_A(S)+\Theta_A(S')$ by \cref{lem:quadratic}. Thus it suffices to show that $\Theta_A(S)=0$ for a compact surface $S$ such that $0=[S]\in H_2(M,\Sigma;\Z/2)$ and $w_1(\Sigma)$ is trivial on $\partial S$. In particular, we know by \cref{lem:representable} that every element of $H_2(M,\Sigma;\Z/2)$, in particular the trivial class, can be represented by an immersed surface $S$.

    By \cref{lem:embedded_boundary}, we can assume that $\partial S$ is embedded. As $0=[S]\in H_2(M,\Sigma;\Z/2)$, we also have that $0=[\partial S]\in H_1(\Sigma;\Z/2)$ since $S$ maps to $\partial S$ under the map $H_2(M,\Sigma;\Z/2)\to H_1(\Sigma;\Z/2)$.   Pick a set $\cF'$ of components of $\Sigma \mid \partial S$ as in \cref{lem:cutting}.  Gluing the $F_i\in \cF'$ to $S$ along the common boundary, we obtain a closed surface $N$. First note that $N$ represents the same class as $S$ in $H_2(M,\Sigma;\Z/2)$ since it only differs by a subset of $F(\Sigma)$. Hence $0=[N]\in H_2(M,\Sigma;\Z/2)$. As $N$ is closed it also defines an element in $H_2(M;\Z/2)$. Note that we have the pair sequence
	\[\cdots\ra H_2(\Sigma;\Z/2)\xrightarrow{\phantom{5}F\phantom{5}} H_2(M;\Z/2)\ra H_2(M,\Sigma;\Z/2)\ra \cdots\]
	Hence $N$ represents the same class in $H_2(M;\Z/2)$ as a subsurface $\Sigma'$ of $\Sigma$.
	Let $\lambda_M$ denote the $\Z/2$-valued intersection form on $H_2(M;\Z/2)$. By hypothesis, we have $\lambda_M(f_j,f_{j'})=0$ for any two connected components $f_j,f_{j'}$ of $F$. Thus
\begin{equation}\label{eqn:lambdaNF=lambdaNN}
  \lambda_M([N],[F])+\lambda_M([N],[N])=0.
\end{equation} We finish the proof of~\eqref{repitem-i-lem-welldefined} by showing that $\Theta_A(S)=\lambda_M([N],[F])+\lambda_M([N],[N])$.
	
	Recall that we were able to assume that $\partial S$ is embedded in $F(\Sigma)$ away from the double points and that $\lambda_M([N],[N])=e(\nu N)\mod{2}$. We claim that this in turn agrees with $e(S)+\sum_{F_i\in\cF'}e(F_i)$. Here we define $e(F_i)$ as follows. We used $F$ to define a nowhere vanishing section of $\nu_S^M|_{\partial S}$. Since $\nu_S^M|_{\partial S}$ is two dimensional, we can pick a linearly independent nonvanishing section. This can be equivalently used for the definition of $e(S)$. But this new section now can also be used to define $e(F_i)$. Combining these vector fields that are transverse to the zero section defines a vector field on the normal bundle of $N$, and hence computes the Euler number of the normal bundle of $N$. Thus we have shown
	\[\lambda_M([N],[N])=e(S)+\sum_{F_i\in\cF'}e(F_i).\]	
	Now consider $\lambda_M([N],[F])$. We can use the vector field used for defining $e(F_i)$ to make $N$ and $F$ transverse. Then $\lambda_M([N],[F])$ is given by the sum of $\lvert S\pitchfork F \rvert$, $\sum_{F_i\in\cF'} e(F_i)$, and the self-intersection points of $F$ contained in the $F_i\in\cF'$. As the self-intersection points of $F$ are paired by the Whitney arcs $A$, we have that modulo two the number of self-intersection points of $F$ contained inside $F_i$ agrees with $\lvert A\pitchfork \partial F_i \rvert$. Since the boundary of the $F_i$ is precisely $\partial S$, we have
	\[
	\lambda_M([N],[F])= \lvert \Int S\pitchfork F \rvert + \sum_{F_i\in\cF'}e(F_i) + \lvert A \pitchfork \partial S \rvert.
	\]	
Therefore
\begin{align*}
  \lambda_M([N],[N]) + &\lambda_M([N],[F])\\
  = &e(S) + \sum_{F_i\in\cF'} e(F_i) + \lvert \Int  S\pitchfork F \rvert  + \sum_{F_i\in\cF'}e(F_i) + \lvert A \pitchfork \partial S \rvert \\
  = &e(S)  + \lvert \Int  S\pitchfork F \rvert  +  \lvert A \pitchfork \partial S \rvert\\
   = &\Theta_A(S) \in \Z/2,\end{align*}
where the last equality holds because $\mu_S(\partial S)=0$.
Combine this with \eqref{eqn:lambdaNF=lambdaNN} to obtain
	\[\Theta_A(S)=\lambda_M([N],[N])+\lambda_M([N],[F])=0.\]
	This completes the proof of~\eqref{repitem-i-lem-welldefined}.
	
Before proving~\eqref{repitem-ii-lem-welldefined} and~\eqref{repitem-iii-lem-welldefined}, we introduce a general construction. Let $B$ be an annulus as in~\eqref{repitem-ii-lem-welldefined} with an embedded arc $D$ in $B$ connecting its two boundary components. As in \cref{rem:absent-finger}, add a tube to $F(\Sigma)$ along the arc $D$. Let $d, d'$ denote the two discs removed from $F$ when the tube is added. Adding the tube changes $F$ to some $F'$, an immersion of a surface $\Sigma'$, and changes $B$ to a disc $\Delta$. As before, observe that $\partial \Delta$ is an orientation-preserving curve in $\Sigma'$ and $A$ is now a collection of Whitney arcs pairing the double points of $F'$. By construction, we see that $\Theta_A(B,D)=\Theta_A(\Delta)$.

Moreover, suppose there is either some immersed compact surface $S$ in $M$ as in~\eqref{repitem-i-lem-welldefined} or some immersed annulus $B'$ as in~\eqref{repitem-ii-lem-welldefined}, with an embedded arc $D'$ on $B'$ connecting its two boundary components, where possibly $B=B'$. We may choose the tube in the above construction thin enough so that $\Theta_A(S)$ and $\Theta_A(B',D')$ remain unchanged. In particular, this means we assume, after a small local isotopy, that the discs $d$ and $d'$ do not intersect the boundaries of $S$ and $B'$, so both represent classes in $H_2(M,\Sigma';\Z/2)$.
We have the following claim.

\begin{claim}\label{lem:learning-to-lift}
 If $[B]=[S]\in H_2(M,\Sigma;\Z/2)$, then  either $[\Delta]=[S]\in H_2(M, \Sigma';\Z/2)$ or $[\Delta]=[S]+[d]\in H_2(M, \Sigma';\Z/2)$.
Similarly,  if $[B]=[B']\in H_2(M,\Sigma;\Z/2)$ then either $[\Delta]=[B']\in H_2(M, \Sigma';\Z/2)$ or $[\Delta]=[B']+[d]\in H_2(M, \Sigma';\Z/2)$.
\end{claim}

\begin{proof}
The exact sequence of the triple with $\Z/2$ coefficients yields
	\begin{align*}
	(\Z/2)^2 \cong H_2(\Sigma,\Sigma\setminus (\mathring{d}\cup \mathring{d'})) &\ra   H_2(M,\Sigma\setminus (\mathring{d}\cup \mathring{d'}))\xrightarrow{\phantom{5}j\phantom{5}}  H_2(M,\Sigma) \\ & \ra H_1(\Sigma,\Sigma\setminus (\mathring{d}\cup \mathring{d'})) = 0,
	\end{align*}
so  $j$ is surjective with kernel generated by the images of  $[d]$ and $[d']$ from the left hand group.

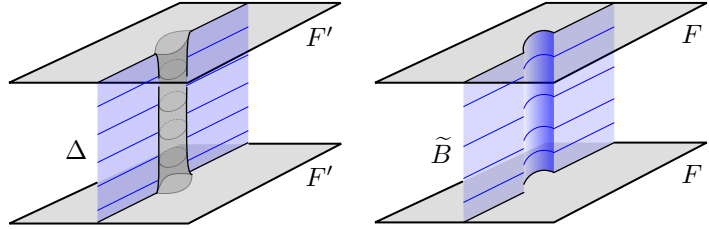
\begin{figure}[htb]
    \centering
     \begin{tikzpicture}
        \node[anchor=south west,inner sep=0] at (0,0){	        \includegraphics{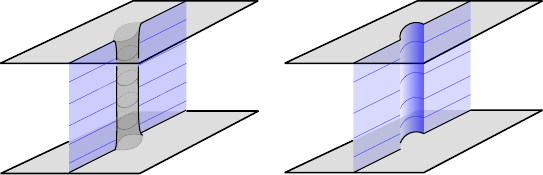}};
		\node at (4.1,0.65) {$F'$};
	   \node at (4.1,2.5) {$F'$};
	   \node at (0.9,1) {$\Delta$};
	   \node at (9,0.68) {$F$};
	   \node at (9,2.53) {$F$};
	   \node at (5.7,1) {$\wt{B}$};
	\end{tikzpicture}
    \caption{A strip, i.e.\ half of the tube, added to $\Delta$.}
\label{fig:absent-finger-paint}
\end{figure}

    Construct a lift $\wt{B}$ of $\Delta$ in $H_2(M,\Sigma\setminus (\mathring{d}\cup \mathring{d'}))$ by adding a strip along the added tube to $\Delta$, as shown in \cref{fig:absent-finger-paint}. Since $\wt{B}$, $S$, and $B'$ are mapped by $j$ to $B$, $S$, and $B'$ in $H_2(M,\Sigma)$ respectively, and the kernel is generated by $[d]$ and $[d']$, we see that the classes of $\wt{B}$, $S$, and $B'$ differ at most by the classes $[d]$ and $[d']$. The map $H_2(M,\Sigma\setminus (\mathring{d}\cup \mathring{d'}))\to H_2(M,\Sigma')$ identifies $[d]$ and $[d']$, so the claim follows.
\end{proof}

We continue now to prove~\eqref{repitem-ii-lem-welldefined}. Let $B$ and $B'$ be immersed annuli in $M$ as in the statement of ~\eqref{repitem-ii-lem-welldefined}. Choose arcs $D$ in $B$ and $D'$ in $B'$ connecting the boundary components of each, and assume that $[B]=[B']$. By the construction from the proof of \cref{lem:learning-to-lift} applied twice, once to $B$ and once to $B'$, we find discs $\Delta$ and $\Delta'$, coming from $B$ and $B'$ respectively, such that $\Theta_A(B,D)=\Theta_A(\Delta)$ and $\Theta_A(B',D')=\Theta_A(\Delta')$. From~\cref{lem:learning-to-lift}, applied twice with the r\^{o}les of $B$ and $B'$ reversed, we see that the classes $[\Delta]$ and $[\Delta']$ satisfy:
\[\big([\Delta] = [B]   \text{ or } [\Delta] = [B] +[d] \big) \text{ and } \big( [\Delta'] = [B']   \text{ or } [\Delta'] = [B'] +[d] \big). \]
Since also $[B] = [B']$, it follows that either $[\Delta]=[\Delta']$ or $[\Delta]=[\Delta']+[d]$ in $H_2(M,\Sigma';\Z/2)$ for $\Sigma'$ the surface obtained from applying the construction (twice) to $\Sigma$.

In the first case, $[\Delta]=[\Delta']$, the proof of~\eqref{repitem-ii-lem-welldefined} is completed by appealing to~\eqref{repitem-i-lem-welldefined}, which says that $\Theta_A(\Delta)=\Theta_A(\Delta')$, since both $\Delta$ and $\Delta'$ have $w_1(\Sigma')$ trivial on the boundary. In the second case, $[\Delta]=[\Delta']+[d]$, we also appeal to~\eqref{repitem-i-lem-welldefined}, but now for the pair of surfaces $\Delta$ and $\Delta' \cup d$. So we have that $\Theta_A(\Delta)=\Theta_A(\Delta' \cup d)$. It follows directly from the definition that $\Theta_A(\Delta'\cup d)=\Theta_A(\Delta')$. This completes the proof of~\eqref{repitem-ii-lem-welldefined}. In particular, we have proved that $\Theta_A(B,D)$ does not depend on the choice of arc $D$.

The proof of~\eqref{repitem-iii-lem-welldefined} is similar. Suppose we have an immersed annulus $B$ in $M$ as in the statement of~\eqref{repitem-ii-lem-welldefined}, as well as an immersed compact surface $S$ in $M$ as in the statement of~\eqref{repitem-i-lem-welldefined}. Choose an embedded arc $D\subseteq B$ connecting the boundary components. Assume that $[S]=[B]\in H_2(M,\Sigma;\Z/2)$. Apply the previous construction to $B$, yielding a disc $\Delta$ which by~\cref{lem:learning-to-lift} satisfies either $[\Delta]=[S]$ or $[\Delta]=[S]+[d]$ in the group $H_2(M,\Sigma';\Z/2)$ for the surface $\Sigma'$ obtained from applying the construction to $\Sigma$. Further, we know that $\Theta_A(B,D)=\Theta_A(\Delta)$. Now in the first case the proof is completed by appealing to~\eqref{repitem-ii-lem-welldefined}, which says that $\Theta_A(\Delta)=\Theta_A(S)$. In the second case, apply~\eqref{repitem-ii-lem-welldefined} to the pair $\Delta$ and $S \cup d$, to see that $\Theta_A(\Delta)=\Theta_A(S \cup d)$. It follows directly from the definition that $\Theta_A(S\cup d)=\Theta_A(S)$.
	
It remains to prove~\eqref{repitem-iv-lem-welldefined}. Let $B$ denote an element of $\bands(F)$. First note that only the term $\lvert \partial B\pitchfork A \rvert$ of $\Theta_A(B)$ depends on the Whitney arcs $A$. Let $A'$ denote another collection of Whitney arcs. The quantities $\Theta_A(B)$ and $\Theta_{A'}(B)$ differ by $\lvert \partial B\pitchfork A \rvert + \lvert \partial B\pitchfork A' \rvert$, regardless of whether $\Sigma$ is orientable or nonorientable.
	
	\setcounter{case}{0}
\begin{case}\label{case:same-pairing}
The collections of Whitney arcs $A$ and $A'$ correspond to the same choice of pairing up of the double points of $F$.
\end{case}
 For each pair of double points, we can pick Whitney discs $W_1$ and $W_2$ with boundary in $A$ and $A'$ respectively. By adding small strips to the union of $W_1$ and $W_2$ in the neighbourhood of the double points, we can see that the difference of $A$ and $A'$ is the boundary of some collection of bands $B'$. For more details about this construction, see the upcoming proof of~\cref{thm:main}.
	Then we have
	\[
	\lvert \partial B\pitchfork A\rvert + \lvert \partial B\pitchfork A' \rvert = \lvert \partial B\pitchfork (A\cup A')\rvert = \lvert \partial B\pitchfork \partial B'\rvert =\lambda_\Sigma(\partial B, \partial B') \mod 2,
	\]
	which vanishes by assumption.
	
	\begin{figure}[htb]
    \centering
    \begin{tikzpicture}
            \node[anchor=south west,inner sep=0] at (0,0){\includegraphics{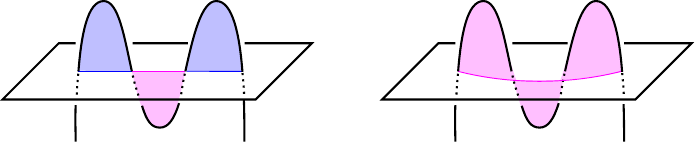}};
				\node at (1.8,1.6) {$W_1$};
				\node at (3.7,1.6) {$W_2$};
				\node at (2.7,0.95) {$V_1$};
				\node at (9.2,1.25) {$V_2$};
				\node at (1.1,1.2) {$p_1$};
				\node at (2.05,1) {$p_2$};
				\node at (3.35,1) {$q_1$};
				\node at (4.4,1.2) {$q_2$};
				\node at (7.5,1.2) {$p_1$};
				\node at (10.8,1.2) {$q_2$};
        \end{tikzpicture}
    \caption{Left: the Whitney discs $W_1$, $W_2$, and $V_1$, pairing up double points as $(p_1,p_2)$, $(q_1,p_2)$ and $(q_1,q_2)$, respectively. Right: the Whitney disc $V_2$ pairing up $(p_1,q_2)$ is obtained as a union of $W_1$, $W_2$, and $V_1$, by adding small bands at the points $p_2$ and $q_1$ to resolve the singularities, and pushing the interiors of the bands into the complement of $F$. Compare with~\cite{Stong}*{Figure~2}.}
\label{fig:long-whitney-arc}
\end{figure}
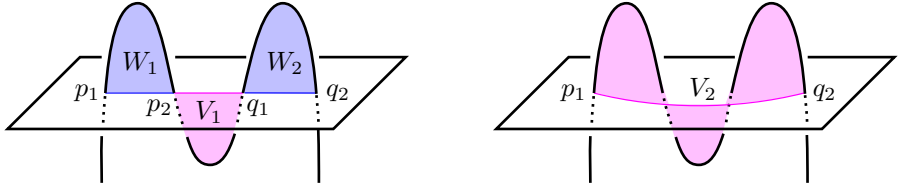

\begin{case}
The collections of Whitney arcs $A$ and $A'$ correspond to a different pairing up of the double points of $F$.
\end{case}
 From $A$ we can construct Whitney arcs $A''$ so that $A'$ and $A''$ correspond to the same pairing up of double points, as in ~\cref{fig:long-whitney-arc}. Here are the details.
	We will define the family $A''$ iteratively, starting with $A$. Let $p_1,p_2,q_1,q_2$ be double points of~$F$.
    Suppose that arcs in $A$ pair up $p_1$ and $p_2$, as well as $q_1$ and $q_2$, while arcs in $A'$ pair up $p_2$ and $q_1$. Pick Whitney discs $W_1$ and $W_2$ with boundary in $A$. Let $V_1$ be a Whitney disc for the points $p_2$ and $q_1$ with boundary away from $A$. Then, as indicated in \cref{fig:long-whitney-arc}, we may choose Whitney arcs, away from the other arcs in $A$, so that $p_2$ and $q_1$ are also paired by a Whitney disc $V_2$, obtained as a union of $W_1$, $W_2$, and $V_1$. Modify the family $A$ by removing $\partial W_1$ and $\partial W_2$, and adding in $\partial V_1$ and $\partial V_2$. Comparing this new family with $A'$, we see that we have reduced the number of mismatches in the pairing up of double points of $F$. Iterate this process and call the result $A''$.

    Looking more closely at the construction in the previous paragraph, observe that at each step, the family of arcs changes by adding in two parallel copies of the boundary of a Whitney disc $V_1$. Since intersection points are counted modulo~$2$, $\Theta_A$ and $\Theta_{A''}$ are equal. By \cref{case:same-pairing} we know that
    $\Theta_{A'}$ and $\Theta_{A''}$ are equal when restricted to $\bands(F)$. Thus, $\Theta_A$ and $\Theta_{A'}$ are equal when restricted to $\bands(F)$, as needed.
\end{proof}

\section{Proof of \texorpdfstring{\cref{thm:main,thm:embedding-obstruction}}{the main theorems}}\label{sec:proof-of-main-theorem}

First we prove \cref{thm:embedding-obstruction} from the introduction, which shows that for $b$-characteristic surfaces,  $\tw(F ,\W )\in \Z/2$ is well defined, i.e.\ independent of the Whitney discs $\W $. Note that the theorem has no assumption about the existence of algebraically dual spheres.

\begin{reptheorem}{thm:embedding-obstruction}
Let $F\colon (\Sigma,\partial \Sigma)\looparrowright (M, \partial M)$ be as in \cref{convention} with $\mu (F)=0$. Let $\W$ be a convenient collection of Whitney discs for the double points of $F$. Then
$F$ is $b$-characteristic if and only if for every $F'$ regularly homotopic to $F$ and convenient collection $\W'$ for the double points of $F'$, we have $\tw(F,\W)=\tw(F',\W')$.

For $b$-characteristic $F$, we denote the resulting regular homotopy invariant by $\tw(F) \in \Z/2$. Then if $\km(F)=0$, e.g.\ if $F$ is an embedding, then $\tw(F) =0$.
\end{reptheorem}

\begin{proof}
The final sentence, that $\km(F)=0$ implies $t(F)=0$ for $b$-characteristic $F$, is an immediate consequence of the definitions.

Now suppose that $F $ is not $b$-characteristic. Then by~\cref{lem:dependence_on_A} we can assume that $\lambda_{\Sigma }\vert_{\partial \bands(F )}$ is trivial, which implies that the function $\Theta$ is well defined on $\bands(F )$. Since we assume that $F $ is not $b$-characteristic, there exists $B\in \bands(F )$ so that $\Theta(B)=1$, so we can apply~\cref{construction:finger-move} and~\cref{lem:fiber-move} to find $F'$, regularly homotopic to $F$, and a convenient collection of Whitney discs $\W'$ for the double points of $F'$ with $\tw(F, \W)\neq \tw(F',\W')$.

If $F $ is $b$-characteristic, by definition $\lambda_{\Sigma }\vert_{\partial \bands(F )}$ is trivial and $\Theta$ is trivial on $\bands(F )$. As indicated above, the function $\Theta$, as well as which classes of $H_2(M,\Sigma ;\Z/2)$ can be represented by bands, only depends on the immersion $F$ up to regular homotopy. We need to show that  $t(F ,\W )$ does not depend on the choice of pairing of the double points, the choice of Whitney arcs, nor the choice of Whitney discs; see~\cref{fig:sheet-change}. Let $\W $ be a given initial choice of convenient collection of Whitney discs for the double points of $F $. Let $A$ be the corresponding collection of Whitney arcs for the double points of $F $.

The remainder of the proof is similar to Stong's~\cite{Stong}*{pp.~1311--3} and \cite{FQ}*{Section~10.8A}.
We will work with \emph{weak collections of framed Whitney discs} and the alternative count $t_{\Alt} \in \Z/2$, as in \cref{def:weak-Whitney,def:t_alt}. So the boundaries of our collections of Whitney discs might not be disjointly embedded, but the Whitney discs will be framed (as can always be arranged by boundary twisting). We will show that $t_{\Alt}(F ,\W )$ does not depend on the choice of weak collection of Whitney discs $\W $, and then use that $t_{\Alt}(F ,\W ) = t(F ,\W )$ for $\W $ a convenient collection (\cref{lem:alt-enough}).
	
\begin{claim}\label{argument:pairing-up-independence}	
Given a weak collection of Whitney discs $\W $ corresponding to some choice of pairing up of double points of $F $, then for any other choice of pairing, there exists a weak collection of Whitney discs $\mathcal{V} $ for that choice, so that $t_{\Alt}(\mathcal{V} )=t_{\Alt}(\W )$.
\end{claim}

\begin{proof}
Let $p_1,p_2,q_1,q_2$ be double points of $F $.
	Suppose that in the initial choice of data, $p_1$ and $p_2$ are paired by a Whitney disc $W_1\in \W $, and $q_1$ and $q_2$ by a Whitney disc $W_2\in \W $. Suppose we instead pair up $p_1$ and $q_2$ by some Whitney disc $V_1$. Then, as indicated in \cref{fig:long-whitney-arc},
$p_2$ and $q_1$ are also paired by a Whitney disc $V_2$, obtained as a union of $W_1$, $W_2$, and $V_1$. Then $(\W  \setminus \{W_1,W_2\})\cup \{V_1,V_2\}$ is a weak collection of framed Whitney discs.  The contribution of $V_1$ and $V_2$ to $t_{\Alt}(F ,(\W  \setminus \{W_1,W_2\})\cup \{V_1,V_2\})$ counts the intersections of $F $ with each disc $W_1$ and $W_2$ once, while it counts the intersections of $F $ with the disc $V_1$ twice.  Each intersection of $\partial V_1$ with $A\sm (\partial W_1 \cup \partial W_2)$ can be paired with an intersection of $\partial V_2$ with $A\sm (\partial W_1 \cup \partial W_2)$. Each intersection of $\partial V_1$ with $\partial W_1 \cup \partial W_2$ gives rise to two contributions to $t_{\Alt}$: an intersection of $\partial V_2$ with $\partial V_1$ and a self-intersection of $\partial V_2$.   Since intersections are counted mod 2 in the definition of $t_{\Alt}$, we see that
\[
t_{\Alt}(F ,(\W  \setminus \{W_1,W_2\})\cup \{V_1,V_2\})=t_{\Alt}(F ,\W )\in \Z/2
\]
as needed. Iterate this process to complete the proof of \cref{argument:pairing-up-independence}. \end{proof}

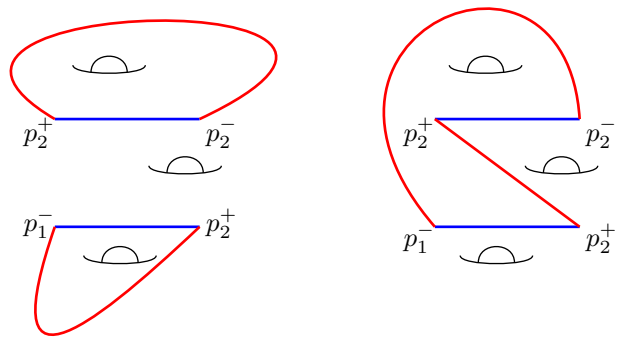
\begin{figure}[htb]
    \centering
   \begin{tikzpicture}
        \node[anchor=south west,inner sep=0] at (0,0){\includegraphics{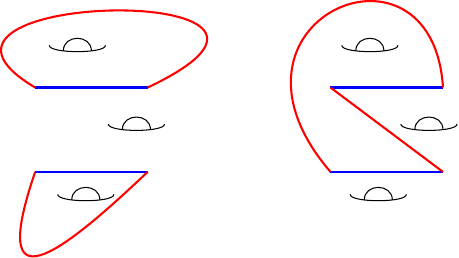}};
		\node at (0.4,2.7) {$p_2^+$};
		\node at (2.8,2.7) {$p_2^-$};
		\node at (0.4,1.5) {$p_1^-$};
	 \node at (2.8,1.5) {$p_2^+$};
			\node at (5.4,2.7) {$p_2^+$};
		\node at (7.8,2.7) {$p_2^-$};
		\node at (5.4,1.3) {$p_1^-$};
	 \node at (7.8,1.3) {$p_2^+$};
			\end{tikzpicture}
    \caption{Within $\Sigma$ we see the preimages $p_1^\pm$ and $p_2^\pm$, for the double points $p_1$ and $p_2$ of $F$ respectively. Blue denotes the Whitney arcs for $W_1$ while red denotes the Whitney arcs for the new disc $V_1$. On the left, the choice of sheets stays the same, while it changes on the right. Compare with~\cite{Stong}*{Figure~3}.}
\label{fig:sheet-change}
\end{figure}

Continuing with the proof of \cref{thm:embedding-obstruction}, next we check that $t_{\Alt}$ is independent of the choice of Whitney discs. This includes potentially changing the Whitney arcs and the choice of sheets at each double point.  Suppose we are given another weak collection of framed Whitney discs $\mathcal{V} $ for the double points of $F $. By applying~\cref{argument:pairing-up-independence}, we may assume that $\mathcal{V} $ corresponds to the same pairing of double points of $F $ as $\W $. Assume the collections are indexed so that $W_\ell\in \mathcal{W} $ and $V_\ell\in \mathcal{V} $ correspond to the same pair of double points. For each $i$, define the weak collection of Whitney discs
\[
\mathcal{U} _i:=\{V_1,V_2, \dots, V_i,W_{i+1}, W_{i+2}\dots, W_N \}
\]
where $\mathcal{U} _0=\mathcal{W} $ and $\mathcal{U} _{N}=\mathcal{V} $.
We will show that $t_{\Alt}(F ,\mathcal{U} _{i-1})=t_{\Alt}(F ,\mathcal{U} _{i})\in \Z/2$ for each $i$. Let~$A_{i}$ denote the collection of Whitney arcs for $\mathcal{U} _i$.
First we  prove a special case.
\setcounter{case}{0}
\begin{claim}\label{case:friendly-disc}
Suppose the Whitney disc $W_i$ is framed, embedded, with interior disjoint from $F $ and the Whitney discs $\mathcal{U}_{i-1} \sm \{W_i\}$, and with $\partial W_i$ disjoint from $A_{i}$, other than the endpoints. Then $t_{\Alt}(F ,\mathcal{U} _{i-1})=t_{\Alt}(F ,\mathcal{U} _{i})\in \Z/2$.
\end{claim}

\begin{proof}
A neighbourhood of $W_i$ is depicted in \cref{fig:band-strip}. Note that the two arcs of $\partial V_i$ lie in $A_i$ and thus my hypothesis only intersect the arcs in $\partial W_i$ at the endpoints. As described in the figure, we wish to perform the Whitney move using $W_i$ pushing towards the Whitney arc $a_i$ for $W_i$. Observe that the union of $V_i$ with a strip, corresponding to the unit outward pointing normal vector field of $a_i\subseteq \partial W_i$, is either an annulus or a M\"obius band; this requires a small isotopy of $V_i$ to ensure that the chosen vector field of $a_i$ is compatible with the Whitney arcs of $V_i$, as shown in \cref{fig:band-strip}. Denote the union of $V_i$ and the strip by $B$.

We show that $B\in\bands(F)$. For this we need to check that condition \eqref{item:defn-band-w1-condn} holds. From the right hand side of \cref{fig:band-strip}, one sees that $\partial B$ is homotopic in $\Sigma$ to the union of $\partial V_i$ and $\partial W_i$. The core $C$ of $B$ is given by the union of $a_i$ and either of the Whitney arcs of $V_i$. The Whitney arcs must induce opposite signs at the two double points, as explained in \cref{defn:opposite-signs}. The orientation conditions in the latter definition imply that the condition in \eqref{item:defn-band-w1-condn} holds, as we explain next. Let $p_1$ and $p_2$ denote the double points paired by $W_i$ (and $V_i$). Let $a_i$ and $b_i$ denote the Whitney arcs of $W_i$, and let $c_i$ and $d_i$ denote those of $V_i$. Begin by fixing local orientations of $M$ and both sheets of $\Sigma$ at $p_1$, so that the first agrees with the one determined by the latter two. Transport the local orientations of $\Sigma$ to $p_2$ via the Whitney arcs of $W_i$ and form the induced local orientation of $M$ at $p_2$. By \cref{defn:opposite-signs}, this does not agree with the local orientation of $M$ at $p_2$ determined by the one at $p_1$ by transporting along $a_i$. Continuing with the local orientations at $p_2$ determined in the previous step, transport the local orientations of $\Sigma$ back to $p_1$, this time along the Whitney arcs of $V_i$. Again by \cref{defn:opposite-signs}, the resulting induced local orientation of $M$ at $p_1$ agrees with the local orientation of $M$ transported to $p_1$ along $c_i$. In this circuit, we have constructed a new set of local orientations of $M$ and the two sheets of $\Sigma$ at $p_1$. Compared to the initial choice, the local orientation induced by the sheets of $\Sigma$ has changed by $\langle w_1(\Sigma), a_i\cup b_i\cup c_i \cup d_i\rangle=\langle w_1(\Sigma), \partial V_i\cup \partial W_i\rangle$. On the other hand, the local orientation of $M$ transported along $a_i\cup c_i$ has changed by $\langle w_1(M),a_i\cup c_i\rangle=\langle w_1(M),C\rangle$, where $C$ is the core of $B$ from above. Since the two orientations must agree, we have   $\langle w_1(\Sigma), \partial V_i\cup \partial W_i\rangle = \langle w_1(M),C\rangle$, as needed.

For the band $B$ as above, performing a finger move as in~\cref{construction:finger-move} creates $W_i$ as the standard Whitney disc, and $V_i$ as the new Whitney disc arising from the band. Here we used the fact that $\partial W_i$ and $\partial V_i$ only intersect at the endpoints.  Since $F$ is $b$-characteristic, the disc $V_i$, has trivial contribution to $t_{\Alt}(F ,\mathcal{U} _i)$ by~\cref{lem:fiber-move}. So does $W_i$ to $t_{\Alt}(F ,\mathcal{U}_{i-1} )$, since by hypothesis $\partial W_i$ is framed, embedded, and disjoint from $A_{i-1}\sm \{a_i,b_i\}\subseteq A_i$, and the interior of $W_i$ is disjoint from $F $.  Therefore $t_{\Alt}(F ,\mathcal{U} _{i-1})=t_{\Alt}(F ,\mathcal{U} _{i})\in \Z/2$ as asserted.
\end{proof}

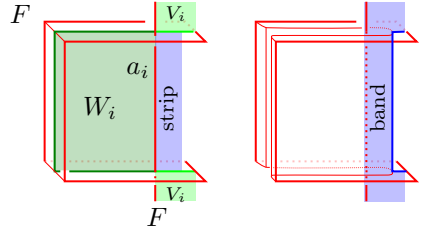
\begin{figure}[tb]
    \centering
    \begin{tikzpicture}
        \node[anchor=south west,inner sep=0] at (0,0){ \includegraphics{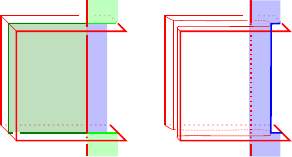}};
		\node at (1.5,-0.2) {$F $};
		\node at (0.75,1.25) {$W_i$};
		\node at (1.25,1.75) {$a_i$};
	 \node at (-0.3,2.5) {$F $};
			\node at (1.75,2.5) {$\mathsmaller{V_i}$};
			\node at (1.75,0.1) {$\mathsmaller{V_i}$};
			\node[rotate=90] at (1.65,1.25) {{\footnotesize strip}};
			\node[rotate=90] at (4.4,1.25) {{\footnotesize band}};
		\end{tikzpicture}
      \caption{Left: two sheets of the surface $\Sigma $ and two Whitney discs $W_i$ and $V_i$ between the same pair of double points. The disc $W_i$ is assumed to be framed, embedded, have interior disjoint from $F $ and the Whitney discs $\mathcal{U}_{i-1} \sm \{W_i\}$, and $\partial W_i$ disjoint from
      $A_i$. One of its Whitney arcs $a_i$ is also labelled. The blue strip to the right of $a_i$ is an extension of $W_i$ beyond its boundary, that is part of the data for the Whitney move.  Right: the result of the Whitney move. The strip and the disc $V_i$ from the previous panel have formed a band $B$ (blue). }\label{fig:band-strip}
\end{figure}
Now we prove the general case.
Denote the double points paired by $W_i$ by $p_1$ and $p_2$.
By a small isotopy, assume that, other than $p_1$ and $p_2$, the arcs of $\partial W_i$ intersect the arcs in $A_i$ in isolated double points in the interiors.
By performing a suitable finger move near $p_2$, split $W_i$ into new Whitney discs $W_i'$ and $U_1$, creating two new double points $q_1$ and $q_2$ in the process, paired by a standard trivial Whitney disc $U_2$, where $U_1$ satisfies the conditions of \cref{case:friendly-disc}. We choose both the base and tip of the finger arc to be closer to $p_2$ than any intersections of $\partial W_i$ with arcs in $A_i$, as well as any self-intersections of $\partial W_i$. See~\cref{fig:nonstandard-finger}. By construction, the points $p_1$ and $q_1$ are paired by $W_i'$, and the points $q_2$ and $p_2$ are paired by $U_1$. Here $U_1$ is framed, embedded, has interior disjoint from $F $ and the Whitney discs $\mathcal{U}_{i-1} \sm \{W_i\}$. In addition $\partial U_1$ is disjoint from $A_{i}$, other than at $p_2$, and is disjoint from $\partial W_i'$. These conditions will shortly allow us to apply \cref{case:friendly-disc} to $U_1$.

Let $F'$ denote the result of performing the finger move above to $F $. Note that
\begin{equation}\label{eqn:t-1}
  t_{\Alt}(F ,\mathcal{U}_{i-1} )=t_{\Alt}(F' ,(\mathcal{U}_{i-1} \sm \{W_i\}) \cup \{W_i',U_1\})
\end{equation}
by construction. Let $V_i'$ denote the Whitney disc obtained as the union of $V_i$, $W_i'$, and $U_2$, as in \cref{fig:long-whitney-arc}. Observe that the Whitney discs $U_1$ and $V_i'$ pair the same double points, namely $q_2$ and $p_2$. Consider the two collections of Whitney discs $(\mathcal{U}_{i-1}  \sm \{W_i\})\cup \{W_i', U_1\}$ and $(\mathcal{U}_{i-1}  \sm \{W_i\})\cup \{W_i', V_i'\}$, for the double points of $F' $. The two collections differ only in that one contains the disc $U_1$ and the other the disc $V_i'$. We will apply \cref{case:friendly-disc} to change between the two collections.  This is permitted since $U_1$ is framed, embedded, has interior disjoint from $F' $ and the Whitney discs $(\mathcal{U}_{i-1}  \sm \{W_i\})\cup \{W_i'\}$, and $\partial U_1$ is disjoint, other than at the endpoints, from the Whitney arcs of $(\mathcal{U}_{i-1}  \sm \{W_i\})\cup \{W_i', V_i'\}$, given by $A_i\cup \partial W_i'\cup \partial U_2$.

So by \cref{case:friendly-disc},
\begin{equation}\label{eqn:t-2}
t_{\Alt}(F' ,(\mathcal{U}_{i-1}  \sm \{W_i\})\cup \{W_i', U_1\})=t_{\Alt}(F',(\mathcal{U}_{i-1}  \sm \{W_i\})\cup \{W_i', V_i'\}).
\end{equation}
By the proof of \cref{argument:pairing-up-independence} (see \cref{fig:long-whitney-arc}),
\begin{equation}\label{eqn:t-3}
t_{\Alt}(F' ,(\mathcal{U}_{i-1}  \sm \{W_i\})\cup \{W_i', V_i'\})= t_{\Alt}(F' ,(\mathcal{U}_{i-1} \sm \{W_i\})\cup \{U_2,V_i\}).
\end{equation}
Since $U_2$ is trivial, we can use it to undo the Whitney move, and obtain
\begin{equation}\label{eqn:t-4}
t_{\Alt}(F' ,(\mathcal{U}_{i-1} \sm \{W_i\})\cup \{U_2,V_i\})=t_{\Alt}(F ,(\mathcal{U}_{i-1}  \sm \{W_i\})\cup \{V_i\})=t_{\Alt}(F ,\mathcal{U}_{i} ).
\end{equation}
The combination of \eqref{eqn:t-1}, \eqref{eqn:t-2}, \eqref{eqn:t-3}, and \eqref{eqn:t-4} imply $t_{\Alt}(F ,\mathcal{U}_{i-1} ) = t_{\Alt}(F ,\mathcal{U}_i )$.
This completes the proof that $t_{\Alt}$ is independent of the choices of Whitney discs, and therefore completes the proof that $t_{\Alt}$ is well defined.

Finally, by \cref{lem:alt-enough} we know that $t_{\Alt}(F ,\W ) = t(F ,\W )$ for $\W $ a convenient collection, so $t$ is well defined for convenient collections $\W $, as desired.
\end{proof}

\begin{figure}[htb]
    \centering
    \begin{tikzpicture}	
        \node[anchor=south west,inner sep=0] at (0,0){	    \includegraphics{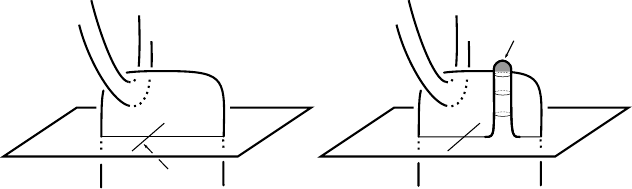}};
	   \node at (1.5,1) {$p_1$};
   	   \node at (2.25,1.1) {$W_i$};
   	   \node at (4.1,1) {$p_2$};
	   \node at (6.8,1) {$p_1$};
   	   \node at (7.5,1.1) {$W'_i$};
   	   \node at (9.5,1) {$p_2$};
   	   \node at (8.23,2.2) {$q_1$};
   	   \node at (8.85,2.2) {$q_2$};   	
\node at (8.9,2.65) {$U_2$};   	
\node at (8.95,1.1) {$U_1$};
\node at (3.2,0.2) {$\partial V_i$};   	
	   	   \end{tikzpicture}
    \caption{Splitting a Whitney disc $W_i$ into two Whitney discs. One of the new Whitney discs, $U_1$, pairing $p_2$ and $q_2$, satisfies the hypotheses of \cref{case:friendly-disc}. The other Whitney disc $W_i'$ intersects whatever $W_i$ intersected. The trivial Whitney disc $U_2$ pairing the new double points $q_1$ and $q_2$ is shown in grey. Note that $\partial W_i$ may intersect $\partial V_i$, or more generally other arcs in $A_i$, or itself.}\label{fig:nonstandard-finger}
\end{figure}
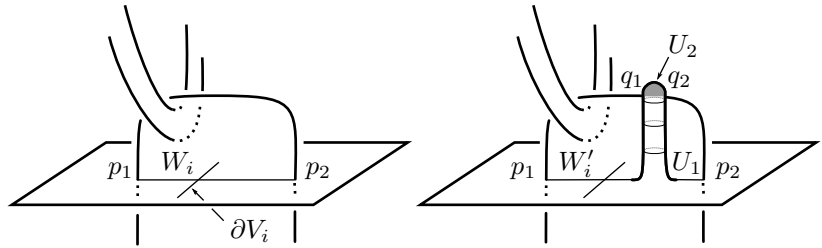

Next we recall the statement of \cref{thm:main} for the convenience of the reader.

\begin{reptheorem}{thm:main}
Let $F\colon (\Sigma,\partial \Sigma)\looparrowright (M, \partial M)$ be as in \cref{convention}. Suppose that $\mu (F)=0$ and that $F$ has algebraically dual spheres. If $F^\twist$ is not $b$-characteristic then $\km(F)=0$. If $F^\twist$ is $b$-characteristic then the secondary embedding obstruction satisfies \[\km(F) = \tw(F^{\twist},\W^{\twist})\in \Z/2\]
for every convenient collection of Whitney discs $\W^\twist$ pairing all the double points of $F^{\twist}$.
\end{reptheorem}

\begin{proof}
First we show that if $F^\twist$ is not $b$-characteristic then $\km(F)=0$.
By~\cref{lem:dependence_on_A} we reduce to the case that $\lambda_{\Sigma^\twist}\vert_{\partial \bands(F^\twist)}$ is trivial, which implies that the function $\Theta$ is well defined on $\bands(F^\twist)$. Since we assume that $F^\twist$ is not $b$-characteristic, there exists $B\in \bands(F^\twist)$ so that $\Theta(B)=1$, so we can apply~\cref{construction:finger-move} and~\cref{lem:fiber-move} to find a collection of Whitney discs $\mathcal{W}^\twist$ for the double points of $F^\twist$ with $t(F^\twist, \mathcal{W}^\twist)=0$. Then by~\cref{lem:tau}, we know that $\km(F)=0$.

By \cref{thm:embedding-obstruction}, if $F^\twist$ is $b$-characteristic, then $t(F^{\twist},\W^{\twist})$ is well defined, i.e.\ is independent of $\W^{\twist}$.  As in \cref{thm:embedding-obstruction} we denote the resulting invariant $t(F^{\twist})$.
We need to show that $\km(F)=t(F^{\twist})$.

Recall that $b$-characteristic implies $r$-characteristic by \cref{lem:r-b}, and also $r$-characteristic implies $s$-characteristic by \cref{rem:r-char-implies-s-char}.  Therefore \cref{lem:tau} applies, which says that if $t(F^\twist)=t(F^\twist, \mathcal{W}^\twist)=0$ then  $\km(F)=0$.
On the other hand, if $\km(F)=0$, then after a regular homotopy the double points of $F$ can be paired up by a convenient collection of Whitney discs with interiors disjoint from $F$. Using these Whitney discs to calculate $t(F^{\twist})$, and regular homotopy invariance of $t$ from \cref{thm:embedding-obstruction} it follows that $t(F^{\twist})=0$.

Thus we have shown that for $F^\twist$ $b$-characteristic and $F$ with algebraically dual spheres, $\km(F)=0$ if and only if $t(F^{\twist})=0$, or equivalently $\km(F)=t(F^{\twist}) \in \Z/2$, as desired.
\end{proof}

\section{Examples and applications}\label{sec:applications}

\begin{proposition}\label{prop:application}
Let $F\colon (\Sigma,\partial \Sigma)\looparrowright (M, \partial M)$ be as in \cref{convention} and assume that $\mu (F)=0$. If there are two orientation-preserving immersed loops in $\Sigma$ that intersect transversely in an odd number of points and are null-homotopic in $M$, then $F$ is not $b$-characteristic.
\end{proposition}

\begin{proof}
The two immersed loops in $\Sigma$ from the assumption bound immersed discs in $M$. These discs give classes in $\bands(F)$ by \cref{ex:tubing-band} and by assumption $\lambda_{\Sigma}|_{\bands(F)}$ is nontrivial. It follows by definition that $F$ is not $b$-characteristic.
\end{proof}

This applies to every simply connected target $M$ whenever $\Sigma$ has a component of positive genus.
\cref{prop:application} also implies \cref{cor:simply-connected-pos-genus} and \cref{cor:stabilisation}, whose statements we recall, as follows.

\begin{repcorollary}{cor:simply-connected-pos-genus}
	If $M$ is a simply connected $4$-manifold and $\Sigma$ is a connected, oriented surface with positive genus, then any generic immersion $F\colon (\Sigma,\partial\Sigma)\imra (M,\partial M)$ with vanishing self-intersection number is not $b$-characteristic.
Thus if~$F$ has an algebraically dual sphere then $\km(F)=0$, and since $\pi_1(M)$ is good the map~$F$ is regularly homotopic, relative to~$\partial \Sigma$, to an embedding.
\end{repcorollary}

\begin{proof}
	As $\pi_1(M)$ is trivial and $\Sigma$ has positive genus, $F$ is not $b$-characteristic by \cref{prop:application}. By \cref{thm:main}, it follows that $\km(F)=0$ if $F$ has an algebraically dual sphere. In this case $F$ is regularly homotopic, relative to $\partial \Sigma$, to an embedding, by \cref{thm:SET}. The theorem applies because $\pi_1(M)$ is good.	
\end{proof}

\begin{repcorollary}{cor:stabilisation}
Let $F\colon (\Sigma,\partial \Sigma)\looparrowright (M, \partial M)$ be as in \cref{convention},  with $\mu (F)=0$ and $\Sigma$ connected.
If $F'$ is obtained from~$F$ by an ambient connected sum with an embedding $S^1 \times S^1 \hookrightarrow S^4$, then $F'$ is not $b$-characteristic.
Thus if $F$ has an algebraically dual sphere then $\km(F')=0$, and if $\pi_1(M)$ is good then~$F'$ is regularly homotopic, relative to~$\partial \Sigma$, to an embedding.
\end{repcorollary}

\begin{proof}
	Since $F'$ is obtained from $F$ by an ambient connected sum with an embedding $S^1\times S^1\hookrightarrow S^4$, we can apply \cref{prop:application} to see that $F'$ is not $b$-characteristic.
	 By \cref{thm:main}, it follows that $\km(F')=0$ if $F$ has an algebraically dual sphere, as this sphere remains algebraically dual to $F'$. If in addition $\pi_1(M)$ is good, then by \cref{thm:SET} $F$ is regularly homotopic, relative to $\partial \Sigma$, to an embedding.
\end{proof}

\begin{example}
	To illustrate the difference between $r$-characteristic and $b$-characteristic surfaces we give an example of a surface that is $r$-characteristic but not $b$-characteristic. Consider any $r$-characteristic immersed sphere with trivial self-intersection number. Add a single trivial tube to obtain an immersed torus. As this will not change the intersection number with any closed surface, the new torus is still $r$-characteristic. But it fails to be $b$-characteristic by \cref{cor:stabilisation}.
\end{example}

\begin{example}\label{ex:km-handle}
We explain next why our methods allow us to obtain embeddings where \cite{FQ}*{Theorem~10.5A\,(1)} would not produce them (cf.\ the discussion directly following \cref{thm:SET}).

Let $f\colon S^2\looparrowright M$ be a generic immersion in a $4$-manifold with $\pi_1(M)$ good, equipped with an algebraically dual sphere and with $\km(f)=1$, for example a sphere representing a generator of $H_2(*\CP^2)$. Other such spheres may be constructed as in \cite{KLCLL19}*{Theorem~2}.
Let $T$ be a generic immersion of a torus produced by adding a trivial tube to $f$, i.e.\ by taking the ambient connected sum of $f$ with the standard embedding $S^1\times S^1\hookrightarrow S^4$. Then by \cref{cor:stabilisation} we see that $\km(T)=0$. Thus $T$ is regularly homotopic to an embedding since $\pi_1(M)$ is good. Fix a $1$-skeleton $\Sigma_0$ for
$S^1\times S^1$.
Then $T$ is not regularly homotopic to an embedding relative to $\Sigma_0$, since the Kervaire--Milnor invariant for $f$ restricted to the $2$-cell(s) $(S^1\times S^1)\sm \nu \Sigma_0$ considered as a map to $M \sm T(\nu \Sigma_0)$ equals $\km(f)=1$.

We emphasise that this holds for every choice of 1-skeleton $\Sigma_0 \subseteq S^1\times S^1$.  In order to apply the strategy of \cite{FQ}*{Theorem~10.5A(1)} to find an embedding, one needs to first make a judicious choice of finger moves.  But without our theory, there is no clear strategy for finding these finger moves.  To obtain an embedding obstruction in this way, matters are worse, since one would need to compute the Kervaire--Milnor invariant of the 2-skeleton for every choice of finger moves and for every choice of 1-skeleton.
\end{example}

\begin{example}\label{ex:nontrivialtorus}
We construct an immersed torus with nontrivial Kervaire--Milnor invariant. In contrast to \cref{prop:application}, the torus in this example is not $\pi_1$-trivial.  Consider an immersion $f_1$ of a $2$-sphere in $*\CP^2$ representing a generator of $H_2(*\CP^2;\Z)$ with trivial self-intersection number.  Let $K \colon S^1 \hookrightarrow S^3$ be an arbitrary knot and consider the embedding of a torus given by the product $f_2:= K \times \Id \colon S^1 \times S^1\hookrightarrow S^3 \times S^1$. Let $F$ denote the interior connected sum $f_1\# f_2\colon S^1\times S^1\looparrowright W:= *\CP^2\# (S^3\times S^1)$.

First we  claim that $F$ is $b$-characteristic.
To see this, we start by computing $H_2(W,S^1\times S^1;\Z/2)$ using the long exact sequence of the pair with $\Z/2$ coefficients:
\[
\begin{tikzcd}[column sep={1.75em}, row sep={1em}]
H_2(S^1\times S^1)\ar[two heads]{r}	&H_2(W) \ar["0"]{r}	&H_2(W,S^1\times S^1) \ar{r}	&H_1(S^1\times S^1)  \ar[two heads]{r}\ar["\cong" labl,draw=none]{d} &H_1(W)\ar["\cong" labl,draw=none]{d}.\\
&&&\Z/2 \oplus \Z/2	&\Z/2
\end{tikzcd}
\]

Therefore $H_2(W,S^1\times S^1;\Z/2) \cong \Z/2$ is generated by $S \times \{p\}$ where $S \subseteq S^3$ is a Seifert surface for the knot $K(S^1)$ and $p \in S^1$.
The intersection form of $S^1\times S^1$ restricted to $\partial S$ is trivial. Since $\Theta$ is well defined on homology classes we can compute it using $S$.  But $S$ has interior disjoint from the image of $F$, embedded boundary, and trivial relative Euler number, so $\Theta (S)=0$. If follows that $\Theta$ vanishes on all of $H_2(W,S^1\times S^1;\Z/2)$, in particular it vanishes on the subset $\bands(F)$. Thus $F$ is $b$-characteristic as claimed.

Observe that $\km(f_1)=1$ inside $*\CP^2$ (see e.g.\ ~\cite{FQ}*{Section~10.8}). We can pick a convenient collection of Whitney discs for $f_1$ in $*\CP^2$. Since $f_2$ is an embedding, these constitute a convenient collection of Whitney discs for $F$. It follows that $\km(F)=\km(f_1)=1$. Note that the choice of knot~$K$ was irrelevant, since for any two choices the resulting immersions $F$ are regularly homotopic and hence have equal Kervaire--Milnor invariant.
\end{example}

\begin{example}\label{ex:pi1Z}
	In the previous example we constructed a generically immersed torus in $*\CP^2\#(S^1\times S^3)$ with nontrivial Kervaire--Milnor invariant. In particular, this torus is not homotopic to an embedding (cf.\ \cref{subsection:homotopy-vs-reg-homotopy}). Now we show that in contrast to this every map $f$ from a closed surface $\Sigma$ to $S^1\times S^3$ is homotopic to an embedding.
	Note that these classes do not have algebraically dual spheres since $\pi_2(S^1\times S^3)=0$.  The surfaces in the regular homotopy class with $\mu(f)_1 =0$ are either not $b$-characteristic or $t(f^{\twist})$ vanishes.
	
	Since the projection $S^1\times S^3\to S^1$ is  3-connected, the induced map $[\Sigma,S^1\times S^3]\to[\Sigma,S^1]$ is bijective. In particular, the homotopy class of a map $f\colon \Sigma\to S^1\times S^3$ is determined by the induced map on fundamental groups.
	
	We first consider the case that $\Sigma$ is connected. Since $\pi_1(S^1\times S^3)\cong \Z$, we can find a generating set for $\pi_1(\Sigma)$ such that at most one generator is non-trivial in $\pi_1(S^1\times S^3)$. Thus there exists a decomposition $\Sigma=H\#\Sigma'$, where $H$ is either a sphere, a torus, or a Klein bottle, with respect to which $f$ can be written as an internal connected sum $T\#f'$, where $T$ is a map on $H$ and $f'$ is $\pi_1$-trivial. In particular, $f'$ is homotopic to an embedding inside a ball $D^4 \subseteq S^1 \times S^3$. It remains to show that $T$ is homotopic to an embedding, which will show that the connected sum is homotopic to an embedding.
	
	If $H$ is a sphere, we are done. If $H$ is a torus, let $i\colon S^1\times S^1\hookrightarrow S^3$ be an embedding. For each $k\in \Z$, define the embedding $h'_k\colon S^1\times S^1\to S^1\times(S^1\times S^1)$ by $(s,t)\mapsto (s^k,(s,t))$. Let $h_k:=(\Id\times i)\circ h'_k$. There exists some $k$ and some identification of $H$ with $S^1\times S^1$ such that $T$ and $h_k$ induce the same map on fundamental groups and thus are homotopic. If $H$ is a Klein bottle, let $p\colon H\to S^1$ be a fibre bundle with fibre $S^1$. For each $k\in\Z$ there exists an immersion $i\colon H\looparrowright S^3$ such that $h_k(x):=(p(x)^k,i(x))$ is an embedding $H \hookrightarrow S^1\times S^3$. As before, there exists some $k$ such that $T$ and $h_k$ are homotopic.
	
	The above embeddings can be realised as embeddings into $S^1\times D^3\subseteq S^1\times S^3$. The argument generalises to disconnected surfaces by picking disjoint copies of $S^1\times D^3$ in $S^1\times S^3$ for each connected component.
\end{example}

Next we prove \cref{prop:connectedsum-intro} and \cref{cor:arbitrary-genus-intro}, which we restate for the convenience of the reader.

\begin{repproposition}{prop:connectedsum-intro}
Let $M_1$ and $M_2$ be oriented $4$-manifolds. Let $F_1\colon (\Sigma_1, \partial \Sigma_1)\looparrowright (M_1, \partial M_1)$ and $F_2\colon (\Sigma_2, \partial \Sigma_2)\looparrowright (M_2, \partial M_2)$ be generic immersions of connected, compact, oriented surfaces, each with vanishing self-intersection number. If $F_i$ is $b$-characteristic for each $i$ then both the disjoint union
\[F_1\sqcup F_2\colon \Sigma_1 \sqcup \Sigma_2\looparrowright M_1\# M_2\]
and any interior connected sum
\[F_1\# F_2\colon \Sigma_1 \# \Sigma_2 \looparrowright M_1\# M_2\]
 are $b$-characteristic, and satisfy
 \[\tw(F_1 \sqcup F_2)=\tw(F_1\#F_2)=\tw(F_1)+\tw(F_2).\]
\end{repproposition}

\begin{proof}
The vanishing of the self-intersection number of $F_i$ is witnessed by a convenient collection of Whitney discs $\mathcal{W}_i$ in $M_i$, for each $i$. The union $\mathcal{W}_1\sqcup \mathcal{W}_2$, now considered in $M_1\# M_2$, shows that the intersection and self-intersection numbers of $F_1\sqcup F_2$, as well as for $F_1\# F_2$, vanish in $M_1\# M_2$. Since the union $\mathcal{W}_1\sqcup \mathcal{W}_2$ pairs all the double points of $F_1\sqcup F_2$ (resp.\ $F_1\# F_2$), and since components of $\mathcal{W}_i$ cannot intersect $F_j(\Sigma_j)$ for all $i\neq j$, the claimed relationship $\tw(F_1 \sqcup F_2)=\tw(F_1\#F_2)=\tw(F_1)+\tw(F_2)$ holds as long as $F_1\sqcup F_2$ and $F\# F_2$ are $b$-characteristic.

As a preliminary step, note that neither $F_i$ has a framed dual sphere in $M_i$, since otherwise it would not be $s$-characteristic, and therefore, not $b$-characteristic by \cref{lem:r-b}. As a result, $\Sigma_i^\twist=\Sigma_i$ for~$i=1,2$.

Next we consider the immersion $F_1\sqcup F_2\colon (\Sigma_1\sqcup \Sigma_2, \partial \Sigma_1\sqcup \partial \Sigma_2)\looparrowright M_1\#M_2$. Let $S\subseteq M_1\# M_2$ denote a connected sum $3$-sphere. Consider a band $[B] \in H_2(M_1 \# M_2,\Sigma_1\sqcup \Sigma_2)$. By (topological) transversality we can assume that $B$ is immersed, the double points of $B$ are disjoint from $S$, and the intersection $B\cap S$ corresponds to an embedded $1$-manifold in the interior of the domain of $B$, since $\partial B \subseteq \Sigma_1 \sqcup \Sigma_2 \subseteq (M_1 \# M_2) \sm S$. The image of this $1$-manifold in $S$ is embedded and bounds a collection of immersed (perhaps intersecting) discs in $S$. Surger $B$ using two copies each of these discs to produce $B_1\subseteq M_1$ and $B_2\subseteq M_2$, where each $B_i$ is an immersed collection of surfaces with $\partial B_i\subseteq \Sigma_i$.

Each component of $B_i$ can be replaced by a band as follows. Recall that since each $M_i$ and $\Sigma_i$ is oriented, there is no condition on Stiefel--Whitney classes for bands, and we need only arrange that each component is either a M\"obius band or an annulus. By considering the Euler characteristic, we see that each component is homeomorphic to either a sphere, an $\RP$, a disc, a M\"obius band, or an annulus.  Then use the tubing procedure from \cref{ex:tubing-band}
to replace each sphere, $\RP$, or disc component by a band. More precisely, choose a small disc on $\Sigma_1$ or $\Sigma_2$, as appropriate, away from all Whitney arcs and double points, and tube into the disc, sphere, or $\RP$.

Since each $F_i$ is $b$-characteristic, $\lambda_{\Sigma_1}\vert_{\partial\bands(F_i)}$ is trivial for each $i$. Therefore, $\lambda_{\Sigma_1\sqcup\Sigma_2}$ is trivial on $\partial B=\partial B_1\cup \partial B_2$.  It follows by \cref{lem:well-defined}\,\eqref{item-iv-lem-welldefined} that $\Theta\colon \bands(F_1\sqcup F_2)\to \Z/2$ is well defined.
By \cref{lem:quadratic}, $\Theta$ extends to a linear map $\langle\bands(F_1\sqcup F_2)\rangle \to \Z/2$ on the subspace $\langle\bands(F_1\sqcup F_2)\rangle  \subseteq H_2(M_1\#M_2,\Sigma_1 \sqcup \Sigma_2;\Z/2)$ generated by the bands.
Then since $[B_1 \cup B_2] = [B] \in H_2(M_1\#M_2,\Sigma_1 \sqcup \Sigma_2;\Z/2)$, we see that $\Theta(B)=\Theta(B_1)+\Theta(B_2)$.

For each $i$, the value of $\Theta(B_i)$ does not depend on whether the ambient manifold is $M_i$ or $M_1\# M_2$, since $B_i$ does not intersect $F_j(\Sigma_j)$ for all $i\neq j$ (see~\cref{def:thetaA}). Since each $F_i$ is $b$-characteristic, $\Theta(B)=\Theta(B_1)+\Theta(B_2)=0+0=0\in \Z/2$. This completes the proof that $F_1\sqcup F_2$ is $b$-characteristic.

Now we consider the connected sum $F_1\# F_2$. Let $B\in H_2(M_1\#M_2,\Sigma_1\# \Sigma_2)$ be a band. As above, we assume that the intersection $B\cap S$ corresponds to an embedded $1$-manifold in the domain of $B$. Unlike above, this may include embedded arcs with endpoints on the boundary. These endpoints are mapped to the intersection $(F_1\#F_2)(\Sigma_1\# \Sigma_2) \cap S$. By connecting the endpoints with arcs on $(F_1\#F_2)(\Sigma_1\# \Sigma_2) \cap S$, we again get a collection of closed circles in $S$, which bound an immersed collection of discs in $S$. Surger using these discs as before to produce collections $B_i\subseteq M_i$. Once again, each component of $B_i$ is homeomorphic to either a sphere, an $\RP$, a disc, a M\"obius band, or an annulus. By \cref{ex:tubing-band} applied to the sphere, $\RP$, and disc components, we may arrange that each component is a band. The argument of the previous paragraph now applies to show that $F_1\# F_2$ is $b$-characteristic.
\end{proof}

\begin{repcorollary}{cor:arbitrary-genus-intro}
For any $g$, there exists a smooth, closed $4$-manifold $M_g$, a closed, connected, oriented surface $\Sigma_g$ of genus $g$, and a smooth, $b$-characteristic, generic immersion $F \colon \Sigma_g \imra M_g$ with $t(F)\neq 0$ and therefore $\km(F)\neq 0$.
\end{repcorollary}

\begin{proof}
	By the same proof as in \cref{ex:nontrivialtorus}, for any knot $K$ the product $T:=K \times \Id\colon S^1 \times S^1\to S^3 \times S^1$ is an embedded $b$-characteristic torus. Since $T$ is an embedding, $t(T)=0$. A computation using the intersection form shows that a generic immersion $S\colon S^2\to \CP^2$ representing three times a generator of $H_2(\CP^2;\Z)$ is $s$-characteristic. Since $\pi_1(\CP^2)$ has no 2-torsion the map $S$ is also $r$-characteristic and thus $b$-characteristic by \cref{lem:r-b}. We will show that $\km(S)=1$.  This was the original example of Kervaire and Milnor~\cite{Kervaire-Milnor:1961-1}.
To see that $\km(S)=1$, represent $S$ in the following way. Take a cuspidal cubic, which is a smooth embedding of a $2$-sphere away from a single singular point. In a neighbourhood of the singular point we see a cone on the trefoil. Replace a neighbourhood of the singular point with an immersed disc $\Delta$ in $D^4$ with boundary the trefoil, and two double points that are paired by a framed Whitney disc that intersects $\Delta$ once.  Alternatively, we can compute $t(S)$ as $(\sigma(\CP^2) - S \cdot S)/8 = (1- 9)/8 \equiv 1 \mod{2}$, see \cref{section:KM-invariant}.
This gives us the case $g=0$.
Next, by \cref{prop:connectedsum-intro}, for every $g\in\N$
\[S\#^g T \colon \Sigma_g \ra \CP^2 \#^g (S^3\times S^1)\]
is a $b$-characteristic generic immersion of a closed surface of genus $g$ with nontrivial $t$, and therefore $\km(S\#^g T)\neq 0$. In particular $S\#^g T$ is not regularly homotopic to an embedding. Note these examples are smooth, but have no algebraically dual sphere. We could replace $(\CP^2,S)$ with $(\CP^2 \#^8 \ol{\CP^2},S')$ where $S'$ is a generic immersion representing the class $(3,1,\dots,1) \in \Z^9 \cong H_2(\CP^2 \#^8 \ol{\CP^2})$, to obtain an example with an algebraically dual sphere and $\km(S') = (-7-1)/8 \equiv 1 \mod{2}$.
\end{proof}

\begin{remark}
Let $M$ denote the infinite connected sum $\CP^2 \#^\infty (S^3\times S^1)$. The proof of the \cref{cor:arbitrary-genus-intro}, along with the formula from  \cref{prop:connectedsum-intro}, shows that for every $g$ there exists a smooth generic immersion $F\colon \Sigma_g\looparrowright M$ with $t(F)\neq 0$ and therefore $\km(F)\neq 0$. The following proposition shows that if there is such a compact $4$-manifold $M$ and such an $F$ then the 4-manifolds must have nonabelian fundamental group.  In other words, if there is an immersed surface in a 4-manifold with abelian fundamental group with nontrivial $\km$, then we give a bound on the complexity of that surface.
\end{remark}

\begin{proposition}\label{prop:so-big}
Let $M$ be a compact $4$-manifold such that $\pi_1(M)$ is abelian with $n$ generators. Let $F\colon \Sigma\looparrowright M$ be a $b$-characteristic generic immersion where $\Sigma$ is a closed, connected surface. Then the Euler characteristic satisfies $\chi(\Sigma)\geq -2n$.
\end{proposition}
\begin{proof}
Suppose that $\chi(\Sigma)<-2n$. Note that $\Sigma$ can be written as a connected sum of a genus $g$ orientable surface $\Sigma'$ for some $g>n$ with zero, one, or two copies of $\RP$. There exists a surjection $\Z^n\twoheadrightarrow \pi_1(M)$. Then the induced map $H_1(\Sigma')\to H_1(M)\cong \pi_1(M)$ admits a lift $H_1(\Sigma')\to \Z^n$, which has kernel of rank at least $2g-n>g$. So there exist closed curves $\gamma_1,\gamma_2$ in $\Sigma'\sm \mathring{D^2}\subseteq \Sigma$ that are null-homotopic in $M$ and $\lambda_{\Sigma}(\gamma_1,\gamma_2)\equiv 1\mod{2}$. It follows that $F$ is not $b$-characteristic.
\end{proof}

Next we prove our corollaries on knot theory from \cref{subsection:knots}.

\begin{repcorollary}{cor:M-slicing}
For every knot $K\subseteq S^3$,
\begin{enumerate}
	\item $g_M(K)=0$ for every simply connected $4$-manifold $M$ not homeomorphic to one of $S^4$, $\CP^2$, or $*\CP^2$;
	\item $g_{\CP^2}(K)\leq 1$ and $g_{\CP^2}( \#^3 T(2,3))=1$; and
    \item $g_{*\CP^2}(K)\leq 1$ and $g_{*\CP^2}( \#^2 T(2,3))=1$.
\end{enumerate}
\end{repcorollary}

\begin{proof}
Let $K\subseteq S^3$ be an arbitrary knot and let $M$ be an arbitrary closed, simply connected $4$-manifold. Let $\Delta'$ be a generically immersed disc bounded by $K$ in a collar $S^3 \times [0,1]$ of $\partial M^\circ$. Since $M$ is simply connected, every class in $H_2(M;\Z)\cong \pi_2(M)$ is represented by a generically immersed sphere. By assumption, $M$ is not homeomorphic to $S^4$ and thus $H_2(M;\Z)$ is nontrivial. Since~$M$ is closed, every primitive class $\alpha \in H_2(M;\Z)$ has an algebraic dual $\beta\in H_2(M;\Z)$, i.e.\ $\lambda(\alpha,\beta)=1$. Represent $\alpha$ and $\beta$ by generically immersed spheres, and tube the interior of $\Delta'$ into $\beta$ to obtain $\Delta$. Add local cusps to arrange $\mu(\Delta)=0$.

First we prove (1). In this case we claim that in the construction of $\Delta$ we can choose the primitive class $\alpha$ to satisfy $\lambda(\alpha,\alpha)\in 2\Z$, as we explain presently. Then the disc $\Delta$ constructed above is not $r$-characteristic, since $\Delta\cdot \alpha \not\equiv \alpha\cdot \alpha \mod{2}$ (see \cref{rem:r-char-implies-s-char}). By \cref{theorem:Stong}, this implies that $\km(\Delta)=0$. Since the disc $\Delta$ has the algebraically dual sphere $\alpha$ and $\pi_1(M)=1$ is good, by \cref{thm:SET}, $\Delta$ is homotopic rel.\ boundary to an embedding. To see the claim regarding $\alpha$, note that when $M\not\cong S^4,\CP^2,*\CP^2$, the group $H_2(M;\Z)$ has rank at least 2 by the classification of closed, simply connected $4$-manifolds up to homeomorphism. Then $H_2(M;\Z)$ has a summand isomorphic to $\Z\oplus \Z$, so the classes $x$, $y$, and $x+y$, for the generators $x,y$ of the $\Z$-factors, are primitive, and at least one of $\lambda(x,x)$, $\lambda(y,y)$, or $\lambda(x+y,x+y)$ is even.

In (2) and (3), we have $M = \CP^2$ or $*\CP^2$. The only primitive classes are $\pm [\CP^1]$, so we choose $\alpha = \beta = [\CP^1]$ in the construction of the first paragraph. We construct the disc $\Delta$ as before, but are no longer able to conclude that it is $r$-characteristic. However by \cref{cor:stabilisation} we know that the connected sum of $\Delta$ with an unknotted torus is homotopic to an embedding. This completes the proof of the first parts of (2) and (3).

Now we prove that $g_{\CP^2}( \#^3 T(2,3))=1$.
Let $K:= \#^3 T(2,3)$. Let $g^d_{\CP^2}(K)$ denote the minimal genus of a surface bounded by $K$ in $(\CP^2)^\circ$ in the homology class $d\in \Z\cong H_2(\CP^2;\Z)$. First we consider $d=\pm 1$, where the class is $b$-characteristic (or equivalently, $s$-characteristic; see \cref{lem:r-b}). As before, construct the disc $\Delta' \subseteq S^3 \times [0,1]$, and tube into $\CP^1$ to obtain the disc $\Delta$. We assume that $\Delta'$ has trivial self-intersection number, so $1=\Arf(K)=t(\Delta')$ by~\citelist{\cite{matsumoto78}\cite{freedman-kirby}\cite{CST-twisted-whitney}*{Lemma~10}}. Since $\CP^1$ is embedded disjointly from $\Delta'$, $t(\Delta) =1$. Thus by \cref{thm:embedding-obstruction}, $\Delta$ is not homotopic to an embedding and so $g^{\pm 1}_{\CP^2}(K)\neq 0$.

Next, let $\sigma_d(K):= \sigma_K(e^{\pi i \frac{d-1}{d}})$, where $\sigma_K$ denotes the Levine--Tristram signature function of~$K$.
By~\cites{gilmer81,viro70}, for even $d$
\[2g^d_{\CP^2}(K)+1 \geq \Big| \sfrac{d^2}{2} - 1- \sigma(K)\Big|,\]
while if $d$ is divisible by an odd prime $p$, then
\[2g^d_{\CP^2}(K)+1 \geq \Big| \sfrac{p^2-1}{2p^2} d^2 - 1 - \sigma_{d}(K) \Big|.\]
In our case, $\sigma(K)=\sigma_d(K)=-6$ for all $d$, and so $g^d_{\CP^2}(K)\geq 1$ for all $d\neq \pm 1$. This completes the argument that $g_{\CP^2}(K)=1$.

Finally we show that $g_{*\CP^2}( \#^2 T(2,3))=1$.
Write $K:=\#^2 T(2,3)$ and  let $g^d_{*\CP^2}(K)$ denote the minimal genus of a surface bounded by $K$ in $(*\CP^2)^\circ$ in the homology class $d\in  H_2(*\CP^2;\Z)$. For $d=\pm 1$, modify the argument above for the case of $\CP^2$, using that tubing into a sphere representing a generator of $H_2(*\CP^2;\Z)$ to obtain a disc $\Delta'$ adds 1 to the $t$ count, and so $1 = 1+\Arf(K) = t(\Delta')$. Therefore, again by \cref{thm:embedding-obstruction}, $g^{\pm 1}_{*\CP^2}(K) \neq 0$.
Next, for $*\CP^2$ the same inequalities from~\cites{gilmer81,viro70} hold, and $\sigma(K) = \sigma_d(K) = -4$ for all $d$.  Therefore applying the inequalities we see that  $g^d_{*\CP^2}(K) \geq 1$ for all $d$. It follows that $g_{*\CP^2}(K) =1$ as asserted.
\end{proof}

\begin{repcorollary}{cor:shake-genus}
For any knot $K\subseteq S^3$, $g^\mathrm{sh}_{\pm 1}(K)=\Arf(K)\in\{0,1\}$.
\end{repcorollary}

\begin{proof}
A generator of $H_2(X_{\pm 1}(K);\Z)$ can be represented by a generically immersed sphere $F$ which is $b$-characteristic (or equivalently, $s$-characteristic; see \cref{lem:r-b}), has trivial $\mu(F)$, and has an algebraically dual sphere. We also recall from \citelist{\cite{matsumoto78}\cite{freedman-kirby}\cite{CST-twisted-whitney}*{Lemma~10}} that~$\Arf(K)$ coincides with the count $t(F)$. Then by \cref{thm:main,thm:embedding-obstruction}, the sphere $F$ is homotopic to an embedding if and only if $\Arf(K)=0$. We have an embedded torus representative for both generators by~\cref{cor:stabilisation}.
\end{proof}

\def\MR#1{}
\bibliographystyle{gtart}
\bibliography{bib}
\end{document}